\documentclass[leqno, 10pt]{amsart}
\usepackage{amsmath, color}
\usepackage{amssymb, latexsym, mathrsfs, a4wide}
\usepackage[all]{xy}

 \newtheorem{definition}{Definition}[section]
 \newtheorem{theorem}[definition]{Theorem}
 \newtheorem{lemma}[definition]{Lemma}
 \newtheorem{proposition}[definition]{Proposition}

 \newtheorem*{theorem*}{Theorem}
\newtheorem*{proposition*}{Proposition}
\newtheorem*{lemma*}{Lemma}

 \theoremstyle{remark}
 \newtheorem{example}[definition]{Example}
 \newtheorem{remark}[definition]{Remark}

\newcommand{\op}[1]{\operatorname{#1}}

\newcommand{\acou}[2]{\ensuremath{\left\langle #1 , #2 \right\rangle}} 
\newcommand{\acoup}[2]{\ensuremath{\left(#1,#2\right)}}

\newcommand{\Tr}{\ensuremath{\op{Tr}}}
\newcommand{\tr}{\op{tr}}

\def\Xint#1{\mathchoice
{\XXint\displaystyle\textstyle{#1}}%
{\XXint\textstyle\scriptstyle{#1}}%
{\XXint\scriptstyle\scriptscriptstyle{#1}}%
{\XXint\scriptscriptstyle\scriptscriptstyle{#1}}%
\!\int}
\def\XXint#1#2#3{{\setbox0=\hbox{$#1{#2#3}{\int}$}
\vcenter{\hbox{$#2#3$}}\kern-.5\wd0}}

\def\dashint{\Xint-}

\newcommand{\bint}{\ensuremath{\dashint}}

\newcommand{\Str}{\op{Str}}

\newcommand{\ind}{\op{ind}}
\newcommand{\Ch}{\op{Ch}}

\newcommand{\GL}{\op{GL}}

\newcommand{\C}{\ensuremath{\mathbb{C}}} 
 
\newcommand{\N}{\ensuremath{\mathbb{N}}} 
 
\newcommand{\R}{\ensuremath{\mathbb{R}}} 
 
\newcommand{\Z}{\ensuremath{\mathbb{Z}}}

\newcommand{\ft}{\ensuremath{\mathfrak{t}}}

\newcommand{\Ca}[1]{\ensuremath{\mathcal{#1}}}
\newcommand{\cA}{\Ca{A}}

\newcommand{\cC}{\Ca{C}}

\newcommand{\cE}{\Ca{E}}
\newcommand{\cF}{\ensuremath{\mathcal{F}}}

\newcommand{\cH}{\ensuremath{\mathcal{H}}}

\newcommand{\cL}{\ensuremath{\mathcal{L}}}

\newcommand{\sD}{\ensuremath{{/\!\!\!\!D}}}
\newcommand{\sS}{\ensuremath{{/\!\!\!\!\!\;S}}}

\newcommand{\Hom}{\op{Hom}}

\newcommand{\End}{\ensuremath{\op{End}}}

\newcommand{\ran}{\op{ran}}

\newcommand{\dom}{\op{dom}}

\newcommand{\bt}{\ensuremath{\bullet}}

\newcommand{\HC}{\op{HC}}
\newcommand{\HP}{\op{HP}}

\numberwithin{equation}{section}

\begin{document}

\title{Index map, $\sigma$-connections, and Connes-Chern character in the setting of twisted spectral triples}
 \author{Rapha\"el Ponge}
 \address{Department of Mathematical Sciences, Seoul National University, Seoul, South Korea}
 \email{ponge.snu@gmail.com}
 \author{Hang Wang}
 \address{School of Mathematical Sciences,  University of Adelaide, Adelaide, Australia}
 \email{hang.wang01@adelaide.edu.au}

 \thanks{R.P.\ was partially supported by Research Resettlement Fund and Foreign Faculty Research Fund of Seoul National University and  Basic Research Grant 2013R1A1A2008802 of 
 National Research Foundation of Korea (South Korea)}
 
\begin{abstract}
    Twisted spectral triples are a twisting of the notion of spectral triples aiming at dealing with some type III geometric situations.  In the first part of the paper, we give a geometric construction of the index map of a twisted spectral 
    triple in terms of $\sigma$-connections on finitely generated projective modules. This clarifies the analogy with the indices of Dirac operators with coefficients in 
    vector bundles. In the second part, we give a direct construction of the Connes-Chern character of a twisted 
    spectral triple, both in the invertible and the non-invertible cases. Combining these two parts we obtain an analogue of the Atiyah-Singer index 
    formula for twisted spectral triples. 
\end{abstract}

\maketitle

\section{Introduction}
Motivated by type III geometric situations, e.g., the action of an arbitrary group of diffeomorphisms on a manifold, 
Connes-Moscovici~\cite{CM:TGNTQF} introduced the notion of a twisted spectral triple.  This is a modification of the 
usual definition of a spectral triple $(\cA,\cH,D)$, where the boundedness of commutators $[D,a]$, $a \in \cA$, is replaced by that
of twisted commutators $[D,a]_{\sigma}=Da-\sigma(a)D$, where $\sigma$ is a given automorphism of the 
algebra $\cA$. Examples include the following:
\begin{itemize}
    \item  Conformal deformations of ordinary spectral triples~\cite{CM:TGNTQF}. 
    
    \item Twistings of ordinary spectral triples by scaling automorphisms~\cite{Mo:LIFTST}.

    \item  Conformal Dirac spectral triples $( C^{\infty}(M)\rtimes G, L^{2}_{g}(M,\sS),\ \sD_{g})_{\sigma}$, 
    where $\sD_{g}$ is the Dirac operator acting on spinors and $G$ is a group of conformal 
    diffeomorphisms~\cite{CM:TGNTQF}. 

    \item  Spectral triples over noncommutative tori associated to conformal weights~\cite{CT:GBTNC2T, CM:MCNC2T}. 

    \item  Twisted spectral triples associated to various quantum statistical systems, including 
    Connes-Bost systems, graphs, and supersymmetric Riemann gas~\cite{GMT:T3sSTQSS}. 
    
    \item Twisted spectral triples associated to some continuous crossed-product algebras~\cite{IM:CPEST}. 
\end{itemize}
We refer to Section~\ref{sec:TwistedST} for a review of the first and the third examples. Connes-Moscovici~\cite{CM:TGNTQF} 
showed that, as for ordinary spectral triples, the datum of a twisted spectral $(\cA,\cH,D)_{\sigma}$ gives rise to a well defined index map $\ind_{D}:K_{0}(\cA)\rightarrow \frac{1}{2}\Z$. 
Moreover, in the $p$-summable case, this index map is computed by the pairing of the $K$-theory $K_{0}(\cA)$ with a Connes-Chern character in ordinary cyclic cohomology. 

One goal of this paper is to present a geometric  interpretation of the index map of a twisted spectral triple. 
First, instead of compressing idempotents by $D$ and its inverse as in~\cite{CM:TGNTQF} (see also~\cite{FK:TSTCCF}), we define the index map in terms of Fredholm indices of  
the following operators, 
\begin{equation*}
    \sigma(e)De:e\cH^{q}\rightarrow \sigma(e)\cH^{q}, \qquad e\in M_{q}(\cA), \ e^{2}=e.
\end{equation*}
This definition is totally analogous to the definition of the index map of an ordinary spectral triple mentioned in~\cite{Mo:EIPDNCG}. 

In the case of an ordinary spectral triple, the index map is usually defined in terms of selfadjoint idempotents, since any 
idempotent is equivalent to a selfadjoint idempotent. For a twisted spectral triple 
$(\cA,\cH,D)_{\sigma}$ the relevant  
notion of selfadjointness is meant with respect to the $\sigma$-involution $a\rightarrow \sigma(a)^{*}$. We shall say that such an idempotent is $\sigma$-selfadjoint. 
In general, it is not clear that an idempotent is equivalent to a $\sigma$-selfadjoint idempotent. For this  
reason, it is important to define the index map for \emph{arbitrary} idempotents. As a result, for a twisted 
spectral triple the index map \emph{a priori} takes values in $\frac{1}{2}\Z$. Nevertheless, when the automorphism has a 
suitable square root, it can be shown that any idempotent is equivalent to a 
$\sigma$-selfadjoint idempotent  and the index map is integer-valued 
(Lemma~\ref{lm:CriteriaSigmaAdjointIntegerIndex}). The precise condition is called the \emph{ribbon} condition (see~Definition~\ref{def:Index.ribbon}) 
and is satisfied by all the main examples of twisted spectral triples. 

As it turns out, the aforementioned construction of the index map is only a special case of a more geometric construction in 
terms of couplings of the operator $D$ with  
$\sigma$-connections on finitely generated projective modules. We refer to Section~\ref{sec:IndexMapSigmaConnections} for the precise 
definition of a $\sigma$-connection. This is the twisted analogue of the usual notion of a connection. The two notions actually 
agrees when $\sigma=\op{id}$. Given a $\sigma$-connection on a finitely generated projective module $\cE$, the definition of the coupled operator $D_{\nabla^{\cE}}$ is similar to 
the coupling of a Dirac operator with a connection on an auxiliary vector 
bundle (see Section~\ref{sec:IndexMapSigmaConnections} for the precise 
definition). In the special case $\cE=e\cA^{q}$ we recover the operator $\sigma(e)De$ by using the so-called Grassmannian
$\sigma$-connection, which is the twisted analogue of the Grassmannian connection. We then show that the operator 
$D_{\nabla^{\cE}}$ is Fredholm and we have
\begin{equation}
    \ind_{D,\sigma}[\cE]=\ind D_{\nabla^{\cE}}.
    \label{eq:Intro.geometric-index-formula}
\end{equation}
This provides us with a geometric interpretation of the index map of a twisted spectral triple. In the case of an ordinary spectral triple we recover
 the geometric interpretation of the index map mentioned in~\cite{Mo:EIPDNCG}. 
 The above  formula exhibits a close analogy with the definition of the standard Fredholm index map of a Dirac operator (the construction of which 
is recalled in Section~\ref{sec:Fredholm.Dirac.Operator}). In particular, we recover the latter in the special case of an ordinary Dirac 
spectral triple (see the discussion on this point at the end of Section~\ref{sec:IndexMapSigmaConnections}).

Another goal of this paper is to give a direct construction of the Connes-Chern character of a $p$-summable 
twisted spectral triple $(\cA,\cH,D)_{\sigma}$. In~\cite{CM:TGNTQF} the Connes-Chern character is defined as the difference 
of Connes-Chern characters of a pair of bounded Fredholm modules canonically associated to the twisted spectral triple. 
This is the same passage as in~\cite{Co:NCDG} from the unbounded Fredholm module picture to the bounded Fredholm module 
picture. One advantage of our definition of the index map is the following index formula (Proposition~\ref{lem:CCC.index-formula-Des}):
\begin{equation}
    \ind \sigma(e)De = \frac{1}{2}\Str\left((D^{-1}[D, e]_{\sigma})^{2k+1}\right), \qquad e=e^{2}\in M_{q}(\cA),
    \label{eq:Intro.index-formula}
\end{equation}where $k$ is any integer~$\geq \frac{1}{2}(p-1)$ and $D$ is assumed to be invertible. It is immediate
that the right-hand side is the pairing of $e$ is with the cochain given by
\begin{equation*}
    \tau_{2k}^{D}(a^{0},\ldots,a^{2k})=c_{k}\Str\left(D^{-1}[D, a^{0}]_{\sigma}\cdots D^{-1}[D, 
    a^{2k}]_{\sigma}\right), \qquad a^{j}\in \cA,
\end{equation*}where $c_{k}$ is a normalization constant. This is the same cochain used in~\cite{CM:TGNTQF} 
to define the Connes-Chern character of a twisted spectral triple. We give a direct proof that $\tau_{2k}^{D}$ is a normalized cyclic cocycle whose class in periodic cyclic cohomology is 
independent of $k$ (Proposition~\ref{prop:Cochian.ConnesChernChar}). The Connes-Chern character $\Ch(D)_{\sigma}$ is 
then defined as the class in periodic cyclic cohomology of any cocycles $\tau_{2k}^{D}$. 

We also use some care to define the Connes-Chern character when $D$ is non-invertible by passing to the 
unital invertible double, which we define as a twisted spectral triple over the augmented unital algebra $\tilde{\cA}=\cA\oplus \C$. 
In the invertible case, we thus obtain two definitions  of the Connes-Chern character, but these two 
definitions agree (see Proposition~\ref{prop:CC.equivalence-tau-overlinetau}). This uses the homotopy invariance of the Connes-Chern character, a detailed proof of which is given 
in Appendix~\ref{app:homotopy-invariance}. 

With the use of the Connes-Chern character and the geometric interpretation~(\ref{eq:Intro.geometric-index-formula}) of the index map we obtain the 
following index formula: for any finitely generated projective module $\cE$ and $\sigma$-connection $\nabla^{\cE}$ on 
$\cE$, 
\begin{equation*}
    \ind D_{\nabla^{\cE}}=\acou{\Ch(D)_{\sigma}}{[\cE]}.
\end{equation*}
This is the analogue for twisted spectral triples of the Atiyah-Singer index formula. 

We have attempted to give very detailed accounts of the constructions of the index map and Connes-Chern character of twisted spectral triples. 
 It is hoped that the details of these constructions will also be helpful to readers who are primarily  interested in understanding
  these constructions in the setting of ordinary spectral triples. 

The paper is organized as follows. In Section~\ref{sec:TwistedST}, we review some important definitions and examples regarding twisted 
spectral triples. In Section~\ref{sec:Fredholm.Dirac.Operator}, we recall the construction of the 
Fredholm index map of a Dirac operator. 
In Section~\ref{sec:IndexTwistedSpectralTriple}, we present the construction of the index map of a twisted spectral triple and 
single out a simple condition ensuring us it is integer-valued. 
In Section~\ref{sec:IndexMapSigmaConnections}, we give a geometric description of 
the index map of a twisted spectral triple in terms of $\sigma$-connections on finitely generated projective modules.  
In Section~\ref{sec:CyclicCohomChernChar}, we review the main definitions and properties of cyclic cohomology and 
periodic cyclic cohomology and their pairings with $K$-theory. 
In Section~\ref{sec:Connes-Chern}, we give a direct construction of the Connes-Chern character of a twisted spectral 
triple for both the invertible and non-invertible cases. In Appendix~\ref{app:PfLemCanoHermMetric} and Appendix~\ref{app:H(E)topIndepHermitianMetricE}, we present proofs of two
technical lemmas from Section~\ref{sec:IndexMapSigmaConnections}. In Appendix~\ref{app:homotopy-invariance}, we give a 
detailed proof of the homotopy invariance of the Connes-Chern 
character of a twisted spectral triple. 

\section*{Acknowledgements}
The authors would like to thank the following institutions for their hospitality during the 
preparation of this manuscript: Seoul National University (HW), Mathematical Sciences Center of Tsinghua University, Kyoto University (Research Institute of Mathematical Sciences and 
Department of Mathematics), and the University of Adelaide (RP),  Australian National University, Chern 
Institute of Mathematics of Nankai University, and Fudan University (RP+HW).

\section{Twisted Spectral Triples. }\label{sec:TwistedST}
In this section, we review various definitions and examples regarding twisted spectral triples. 

\subsection{Twisted spectral triples} We start by recalling the definition of an ordinary spectral triple.

\begin{definition}
A spectral triple $(\cA, \cH, D)$ consists of the following data: 
\begin{enumerate}
\item A $\Z_2$-graded Hilbert space $\mathcal{H}=\mathcal{H}^+\oplus \mathcal{H}^-$.
\item An involutive unital algebra $\mathcal{A}$ represented by bounded operators on $\cH$ preserving its $\Z_{2}$-grading.
\item A selfadjoint unbounded operator $D$ on $\mathcal{H}$ such that for all $a\in\mathcal{A},$
\begin{enumerate}
    \item $D$ maps $\dom(D)\cap \cH^{\pm}$ to $\cH^{\mp}$. 
    \item The resolvent $(D+i)^{-1}$ is a compact operator.
    \item $a \dom(D) \subset \dom(D)$ and $[D, a]$ is bounded for all $a \in \cA$. 
\end{enumerate}
\end{enumerate}   
\end{definition}

\begin{example}
\label{ex:DiracSpectralTriple}
The paradigm of a spectral triple is given by a Dirac spectral triple, 
\begin{equation*}
( C^{\infty}(M), L^{2}_{g}(M,\sS), \sD_{g}),
\end{equation*}
where $(M^{n},g)$ is a compact spin Riemannian manifold ($n$ even) and $\sD_{g}$ is its Dirac operator 
acting on the spinor bundle $\sS$.  In this case the $\Z_{2}$-grading of $L^{2}(M,\sS)$ arises from the 
$\Z_{2}$-grading $\sS=\sS^{+}\oplus \sS^{-}$ of the spinor bundle in terms of positive and negative spinors. 
\end{example}

The definition of a twisted spectral triple is  similar to that of an ordinary spectral triple, 
except for some ``twist'' given by the conditions (3) and (4)(b) below.

 \begin{definition}[\cite{CM:TGNTQF}]\label{TwistedSpectralTriple}
A twisted spectral triple $(\cA, \cH, D)_{\sigma}$ consists of the following: 
\begin{enumerate}
\item A $\Z_2$-graded Hilbert space $\mathcal{H}=\mathcal{H}^{+}\oplus \mathcal{H}^{-}$.
\item An involutive unital algebra $\mathcal{A}$ represented by even bounded operators on $\cH$.

\item An automorphism $\sigma:\cA\rightarrow \cA$ such that $\sigma(a)^{*}=\sigma^{-1}(a^{*})$ for all $a\in 
\cA$. 
\item An odd selfadjoint unbounded operator $D$ on $\mathcal{H}$ such that 
\begin{enumerate}
    \item The resolvent $(D+i)^{-1}$ is compact.
    \item $a \dom(D) \subset \dom(D)$ and $[D, a]_{\sigma}:=Da-\sigma(a)D$ is bounded for all $a \in \cA$. 
\end{enumerate}
\end{enumerate}   
\end{definition}

\begin{remark}\label{rmk:TST.sigma-involution}
    The condition that $\sigma(a)^{*}=\sigma^{-1}(a^{*})$ for all $a\in \cA$ exactly means that the map $a\rightarrow \sigma(a)^{*}$ is an involutive antilinear anti-automorphism of $\cA$. 
\end{remark}

\begin{remark}
    Throughout the paper we shall further assume that the algebra $\cA$ is closed under holomorphic functional calculus. This 
    implies that an element $a \in \cA$ is invertible if and only if it is invertible in $\cL(\cH)$. This also implies that all the algebras 
    $M_{q}(\cA)$, $q \in \N$, are closed under holomorphic functional calculus.
\end{remark}

\begin{remark}
    The boundedness of twisted commutators naturally appears in the setting of quantum groups, but in the attempts of 
    constructing twisted spectral triples over quantum groups the compactness of the resolvent of $D$ seems to fail 
    (see~\cite{DA:QGTST, KS:TSTQSU2, KW:TSTCDC}). We also refer to~\cite{KW:TSTCDC} for relationships between 
    twisted spectral triples and Woronowicz's covariant differential calculi. 
\end{remark}

\subsection{Conformal deformations of ordinary spectral triples}
An important class of examples of twisted spectral triples arises from \emph{conformal deformations} (i.e., inner 
twistings) of ordinary spectral triples.  

Let us start with a Dirac spectral triple $( C^{\infty}(M), L^{2}_{g}(M,\sS), \sD_{g})$ 
associated to a compact Riemannian spin oriented manifold $(M^{n},g)$ of even dimension. Consider a conformal change of 
metric, 
\begin{equation*}
    \hat{g}=k^{-2}g, \qquad k\in C^{\infty}(M), \ k>0.
\end{equation*}We then can form a new Dirac spectral triple $(C^{\infty}(M), L^{2}_{\hat{g}}(M,\sS), \sD_{\hat{g}})$. 
Bearing this in mind, note that the inner product of $L^{2}_{g}(M,\sS)$ is given by 
    \begin{equation*}
        \acou{\xi}{\eta}_{g}:= \int_{M} \acoup{\xi(x)}{\eta(x)}\sqrt{g(x)}d^{n}x, \qquad \xi,\eta \in L^{2}_{g}(M,\sS),
    \end{equation*}where $\acoup{\cdot}{\cdot}$ is the Hermitian metric of $\sS$ (and $n= \dim M$). Consider the linear 
    isomorphism $U:L^{2}_{g}(M,\sS)\rightarrow 
    L^{2}_{\hat{g}}(M,\sS)$ given by
    \begin{equation*}
        U\xi= k^{\frac{n}{2}}\xi \qquad \forall \xi \in L^{2}_{{g}}(M,\sS). 
    \end{equation*}We observe that $U$ is a unitary operator since, for all $\xi\in L^{2}_{g}(M,\sS)$, we have
    \begin{equation*}
      \acou{U\xi}{U\xi}_{\hat{g}}= \int_{M}\acoup{k(x)^{\frac{n}{2}}\xi(x)}{k(x)^{\frac{n}{2}}\xi(x)} 
      \sqrt{k(x)^{-2}g(x)}d^{n}x=\acou{\xi}{\xi}_{g}. 
    \end{equation*}Moreover, the conformal invariance of the Dirac operator (see, e.g., \cite{Hi:HS}) means that 
    \begin{equation}
        \sD_{\hat{g}}=k^{\frac{n+1}{2}}\sD_{g}k^{\frac{-n+1}{2}}. 
        \label{eq:TwsitedST.conformal-invariance-Dirac}
    \end{equation}
    Thus,
    \begin{equation}
    \label{eq:Conformal-Invariance}
        U^{*}\sD_{\hat{g}}U=k^{-\frac{n}{2}}\left( 
        k^{\frac{n+1}{2}}\sD_{g}k^{\frac{-n+1}{2}}\right)k^{\frac{n}{2}}=\sqrt{k}\; \sD_{g}\!\sqrt{k}.
    \end{equation}Therefore, we obtain the following result. 
    
\begin{proposition}\label{prop:ConformalChangeDiracST}
        The spectral triples $(C^{\infty}(M),L^{2}_{\hat{g}}(M,\sS),\sD_{\hat{g}})$ 
        and $(C^{\infty}(M),L^{2}_{g}(M,\sS), \sqrt{k}\; \sD_{g}\!\sqrt{k})$ are unitarily equivalent.
\end{proposition}

\begin{remark}
\label{rem:kLip}
 Whereas the definition of $(C^{\infty}(M),L^{2}_{\hat{g}}(M,\sS),\sD_{\hat{g}})$ requires $k$ to be 
    smooth, in the definition of  $(C^{\infty}(M),L^{2}_{g}(M,\sS), \sqrt{k}\; 
    \sD_{g}\!\sqrt{k})$ it is enough to assume that $k$ is a positive Lipschitz function. 
\end{remark}

More generally, let $(\cA,\cH,D)$ be an ordinary spectral and $k$ a positive element of $\cA$. If we replace $D$ by its 
conformal deformation $kDk$ then, when $\cA$ is noncommutative, the triple $(\cA,\cH,kD{k})$ need not be an ordinary spectral triple. However, as the 
following result shows, it always gives rise to a twisted spectral triple. 

\begin{proposition}[\cite{CM:TGNTQF}]
\label{Prop:ConformalPerturbation}Consider the automorphism $\sigma:\cA\rightarrow \cA$ defined by
    \begin{equation}
        \sigma(a)= k^{2}a k^{-2} \qquad \forall a \in \cA.
        \label{eq:TwistedST.sigmah}
    \end{equation}
Then $(\mathcal{A}, \mathcal{H}, kDk)_{\sigma}$ is a twisted spectral triple. 
\end{proposition}
\begin{remark}
    The main property to check is the boundedness of twisted commutators $[kDk,a]_{\sigma}$, $a \in \cA$. This follows 
    from the equalities,
    \begin{equation*}
      [kDk,a]_{\sigma}=(kDk)a-(k^{2}ak^{-2})(kDk)=k(D(kak^{-1})-(kak^{-1})D)k=k[D,kak^{-1}]k. 
    \end{equation*}
\end{remark}

\begin{remark}
    We refer to~\cite{PW:NCGCGI.PartIII} for a generalization of the above construction in terms  ``pseudo-inner twistings'' of 
    ordinary spectral triples. We note that this construction also encapsulates the construction of twisted spectral triples over 
    noncommutative tori associated to conformal weights of~\cite{CT:GBTNC2T}. 
\end{remark}

\subsection{Conformal Dirac spectral triple}
The conformal Dirac spectral triple of~\cite{CM:TGNTQF} is a nice illustration of the geometric relevance of twisted spectral 
triples. Let $\Gamma$ be the diffeomorphism group of a compact manifold $M$. In order to study the action of $\Gamma$ 
on $M$, noncommutative geometry suggests seeking for a spectral triple over the crossed-product algebra 
$C^{\infty}(M)\rtimes \Gamma$, i.e., the algebra with generators $f\in C^{\infty}(M)$ and 
$u_{\varphi}$, $\varphi\in \Gamma$, with relations,
\begin{equation*}
   u_{\varphi}^{*}= u_{\varphi}^{-1}=u_{\varphi^{-1}}, \qquad u_{\varphi}f=(f\circ \varphi)u_{\varphi}.
\end{equation*}The first set of relations implies that any unitary representation of $C^{\infty}(M)\rtimes \Gamma$ 
induces a unitary representation of $\Gamma$. The second set of relations shows the compatibility with the pushforward 
of functions of by diffeomorphisms. 

The manifold structure is the only diffeomorphism-invariant differentiable structure on $M$, so in particular $M$ does 
not carry a diffeomorphism-invariant metric. This prevents us from constructing a unitary representation of $\Gamma$ in 
an $L^{2}$-space of tensors or differential forms or a first-order (pseudo)differential operator $D$ with a 
$\Gamma$-invariant principal symbol (so as to ensure the boundedness of commutators $[D,u_{\varphi}]$). As observed by 
Connes~\cite{Co:CCTFCF} we can bypass this issue by passing to the total space of the metric bundle $P\rightarrow M$ (seen as 
a ray subbundle of the bundle $T^{*}M\odot T^{*}M$ of symmetric $2$-tensors). As it turns out, the metric bundle $P$ 
carries a wealth of diffeomorphism-invariant structures, including a  diffeomorphism-invariant Riemannian structure. 
The construction of a spectral triple over $C^{\infty}_{c}(P)\rtimes \Gamma$ was carried out by 
Connes-Moscovici~\cite{CM:LIFNCG} who also computed its Connes-Chern character~\cite{CM:HACCTIT,CM:DCCHASTG}. The 
passage from the base manifold $M$ to the metric bundle $P$ is the geometric counterpart of the  well known passage from type III 
factors to type II factors by taking crossed-products with the action of $\R$. 

Even if there is a Thom isomorphism $K_{*}(C^{\infty}_{c}(P)\rtimes \Gamma)\simeq  K_{*}(C^{\infty}(M)\rtimes \Gamma)$, 
it would be desirable to work directly with the base manifold. As mentioned above there are obstructions to doing so when 
dealing with the full group of diffeomorphism. However, as observed by Connes-Moscovici, if we restrict our attention 
to a group of diffeomorphisms preserving a conformal structure, then we are able to construct a spectral triple 
provided we relax the definition of an ordinary spectral triple to that of a \emph{twisted} spectral triple. This 
construction can be explained as follows. 

Let $M^{n}$ be a compact (closed) spin oriented manifold of even dimension $n$ equipped with a conformal structure 
$\cC$, i.e., a conformal class of Riemannian metrics. We denote by $G$ (the identity component of) the group of (smooth) 
 orientation-preserving diffeomorphisms of $M$ preserving the  conformal and spin structures.  
Let $g$ be a metric in the conformal class $\cC$ with associated Dirac operator $\sD_{g}:C^{\infty}(M,\sS)\rightarrow C^{\infty}(M,\sS)$ acting on the sections of the spinor bundle
$\sS=\sS^{+}\oplus \sS^{-}$.  We also denote by $L^{2}_{g}(M,\sS)$ the corresponding Hilbert space of $L^{2}$-spinors. 

If $\phi:M\rightarrow M$ is a diffeomorphism preserving the conformal class $\cC$, then there is a unique function 
$k_{\phi}\in C^{\infty}(M)$, $k_{\phi}>0$, such that 
\begin{equation}
    \phi_{*}g=k_{\phi}^{2}g.
    \label{eq:TwistedST.conformal-factor}
\end{equation}In addition, $\phi$ uniquely lifts to a unitary vector bundle isomorphism $\phi^{\sS}:\sS \rightarrow 
\phi_{*}\sS$, i.e., a unitary section of $\Hom(\sS, \phi_{*}\sS)$ (see~\cite{BG:SODVM}). 
We then let $V_{\phi}:L^{2}_{g}(M,\sS)\rightarrow L^{2}_{g}(M,\sS)$ be the bounded operator given by
\begin{equation}
    V_{\phi}u(x) = \phi^{\sS}\left( u\circ \phi^{-1}(x)\right) \qquad \forall u \in L^{2}_{g}(M,\sS)\ \forall x \in M.
    \label{eq:TwistedST.Vphi}
\end{equation}
The map $\phi \rightarrow V_{\phi}$ is a representation of $G$ in $L^{2}_{g}(M,\sS)$, but this is not a unitary 
representation. In order to get a unitary representation we need to take into account the Jacobian 
$|\phi'(x)|=k_{\phi}(x)^{n}$ of 
$\phi \in G$. This is achieved by using the unitary operator $U_{\phi}:L^{2}_{g}(M,\sS)\rightarrow 
L^{2}_{g}(M,\sS)$ given by 
\begin{equation}
    U_{\phi}=k_{\phi}^{\frac{n}{2}}V_{\phi}, \qquad \phi \in G.
    \label{eq:TwistedST.Uphi}
\end{equation}Then $\phi \rightarrow U_{\phi}$ is a unitary representation of $G$ in $L^{2}_{g}(M,\sS)$. This enables 
us to represent the elements of the crossed-product algebra $C^{\infty}(M)\rtimes G$ as linear combinations of 
operators $fU_{\phi}$ on $L^{2}_{g}(M,\sS)$, where $\phi\in G$ and $f\in C^{\infty}(M)$ acts by scalar multiplication. These operators are subject to the relations, 
\begin{equation*}
    U_{\phi^{-1}}=U_{\phi}^{-1}=U_{\phi}^{*} \qquad \text{and} \qquad  U_{\phi}f =(f\circ \phi^{-1})U_{\phi}. 
\end{equation*}
We then let $\sigma_{g}$ be the automorphism of $C^{\infty}(M)\rtimes G$ given by
\begin{equation}
    \sigma_{g}(fU_{\phi}):=k_{\phi}fU_{\phi} \qquad \forall f \in C^{\infty}(M) \ \forall \phi \in G.
    \label{eq:TwistedST.automorphism-conformal-DiracST}
\end{equation}

\begin{proposition}[\cite{CM:TGNTQF, Mo:LIFTST}]\label{prop:TwistedST.automorphism-conformal-DiracST} 
    The triple $( C^{\infty}(M)\rtimes G, L^{2}_{g}(M,\sS), \,\sD_{g})_{\sigma_{g}}$ is a twisted 
    spectral triple.
\end{proposition}
\begin{remark}
    The bulk of the proof is showing the boundedness of the twisted commutators $[\sD_{g},U_{\phi}]_{\sigma_{g}}$, 
    $\phi \in G$. We remark that 
    \begin{equation*}
        U_{\phi}\sD_{g}U_{\phi}^{*}=k_{\phi}^{\frac{n}{2}}(V_{\phi}\sD_{g}V_{\phi}^{-1})k_{\phi}^{-\frac{n}{2}}=  
       k_{\phi}^{\frac{n}{2}} 
       \sD_{\phi_{*}g}k_{\phi}^{-\frac{n}{2}}=k_{\phi}^{\frac{n}{2}}\sD_{k_{\phi}^{2}g}k_{\phi}^{-\frac{n}{2}}. 
    \end{equation*}Thus, using the conformal invariance law~(\ref{eq:TwsitedST.conformal-invariance-Dirac}), we  get
    \begin{equation*}
        U_{\phi}\sD_{g}U_{\phi}^{*}=k_{\phi}^{\frac{n}{2}}\left(  k_{\phi}^{-\left(\frac{n+1}{2}\right)}\sD_{g}k_{\phi}^{\frac{n-1}{2}} \right)k_{\phi}^{-\frac{n}{2}}= 
                k_{\phi}^{-\frac{1}{2}}\sD_{g}k_{\phi}^{-\frac{1}{2}}
    \end{equation*}
    Using this we see that the twisted commutator $[\sD_{g},U_{\phi}]_{\sigma_{g}}=\sD_{g}U_{\phi}-k_{\phi}U_{\phi}\sD_{g} $ is equal to
    \begin{equation*}
       \left( 
        \sD_{g}k_{\phi}^{\frac{1}{2}}-k_{\phi}(U_{\phi}\sD_{g}U_{\phi}^{*})k_{\phi}^{\frac{1}{2}}\right)k_{\phi}^{-\frac{1}{2}}U_{\phi} 
        = \left( 
        \sD_{g}k_{\phi}^{\frac{1}{2}}-k_{\phi}^{\frac{1}{2}}\sD_{g}\right)k_{\phi}^{-\frac{1}{2}}U_{\phi} = 
        [\sD_{g},k_{\phi}^{\frac{1}{2}}]k_{\phi}^{-\frac{1}{2}}U_{\phi}.
            \end{equation*}This shows that $[\sD_{g},U_{\phi}]_{\sigma_{g}}$ is bounded. 
\end{remark}

\section{The Fredholm Index Map of a Dirac Operator}
\label{sec:Fredholm.Dirac.Operator}
In this section, we recall  how the datum of a Dirac operator gives rise to an additive index map in $K$-theory. In the 
next two sections we shall generalize this construction to \emph{arbitrary} twisted spectral triples. 

Let $(M^n, g)$ be a compact spin oriented Riemannian manifold of even dimension $n$ and 
let $\sD_g: C^{\infty}(M, \sS)\rightarrow C^{\infty}(M, \sS)$ be the associated Dirac operator acting on sections of the spinor bundle. 
As $n$ is even, the spinor bundle splits as $\sS=\sS^+\oplus\sS^-$, where $\sS^+$ (resp., $\sS^-$) is the bundle of 
positive (resp., negative) spinors.
The Dirac operator is odd with respect to this $\Z_{2}$-grading, and so it takes the form,
\begin{equation*}
\sD_{g}=\begin{pmatrix} 0 & \sD^{-}_{g} \\ \sD^{+}_{g} & 0 \end{pmatrix}, \qquad \sD^{\pm}_{g}: C^{\infty}(M, \sS^{\pm})\rightarrow C^{\infty}(M, \sS^{\mp}).
\end{equation*}
Let $E$ be a Hermitian bundle over $M$ and $\nabla^E: C^{\infty}(M, E)\rightarrow C^{\infty}(M, T^*M\otimes E)$ be a Hermitian connection on $E.$
The operator $\sD_{\nabla^E}: C^{\infty}(M, \sS\otimes E)\rightarrow C^{\infty}(M, \sS\otimes E)$ is defined by 
\begin{equation*}
\sD_{\nabla^E}=\sD_{g}\otimes1_E+c(\nabla^E),
\end{equation*}
where $c(\nabla^E)$ is given by the composition,
\begin{equation*}
    C^{\infty}(M, \sS\otimes E) \xrightarrow{1_{\sS}\otimes\nabla^{E}} C^{\infty}(M, \sS\otimes T^*M\otimes E) 
   \xrightarrow{c\otimes 1_E} C^{\infty}(M, \sS\otimes E),
\end{equation*}
where $c: \sS\otimes T^*M\rightarrow\sS$ is the Clifford action of $T^*M$ on $\sS$.
With respect to the splitting $\sS\otimes E=(\sS^+\otimes E)\oplus(\sS^-\otimes E)$, the operator $\sD_{\nabla^E}$ takes the form,
\begin{equation}
\label{eq:Coupled.Dirac.Operator}
\sD_{\nabla^{E}}=\begin{pmatrix}0 & D_{\nabla^E}^- \\ D_{\nabla^E}^+ & 0\end{pmatrix}, \qquad \sD_{\nabla^E}^{\pm}: C^{\infty}(M, \sS^{\pm}\otimes E)\rightarrow C^{\infty}(M, \sS^{\mp}\otimes E).
\end{equation}
As $\nabla^{E}$ is a Hermitian connection, the operator $\sD_{\nabla^{E}}$ is formally selfadjoint, i.e., $(\sD_{\nabla^E}^+)^*=\sD_{\nabla^E}^-.$ 
Moreover, $\sD_{\nabla^E}$ is an elliptic differential operator, and hence is Fredholm. 
We then define its Fredholm index by 
\begin{equation*}  
\ind\sD_{\nabla^E}:=\ind\sD_{\nabla^E}^{+}=\dim\ker\sD_{\nabla^E}^+-\dim\ker\sD_{\nabla^E}^-.
\end{equation*}
This index is computed by the local index formula of Atiyah-Singer~\cite{AS:IEO1, AS:IEO3}, 
\begin{equation*}
\ind\sD_{\nabla^E}=(2i\pi)^{-\frac{n}{2}}\int_M\hat{A}(R^M)\wedge\Ch(F^E),
\end{equation*}
where $\hat{A}(R^M)=\det{}^{\frac12}\left(\frac{R^M/2}{\sinh(R^M/2)}\right)$ is the (total) $\hat{A}$-form of the Riemann curvature $R^M$ and 
$\Ch(F^E)=\Tr\left( e^{-F^E}\right)$ is the (total) Chern form of the curvature $F^E$ of the connection $\nabla^E$.

We observe that even without using the Atiyah-Singer index formula, it is not difficult to see that the value of $\ind\sD_{\nabla^E}$ depends only on the class of $E$ in the 
$K$-theory group $K^{0}(M)$. First, it is immediate to see that its value is independent of the choice of the Hermitian structure of $E$ and the 
Hermitian connection $\nabla^E$, since the principal symbol of $\sD_{\nabla^E}$ does not depend on these data. Second, 
let $\phi:E\rightarrow E'$ be a vector bundle isomorphism. We push forward the Hermitian metric of $E$ to a Hermitian 
metric on $E'$, so that pushing forward the connection $\nabla^{E}$ we get a Hermitian connection on  $E'$. Then 
$\sD_{\nabla^{E'}}=(1_{\sS}\otimes \phi_{*})\sD_{\nabla^{E}}(1\otimes \phi^{*})$, so that $\ker 
\sD_{\nabla^{E}}^{\pm}\simeq \ker \sD_{\nabla^{E'}}^{\pm}$, and hence $\ind\sD_{\nabla^{E'}}=\ind\sD_{\nabla^E}$. In 
addition, let $F$ be another Hermitian vector bundle  equipped with a Hermitian connection $\nabla^F$. 
We equip $E\oplus F$ with the connection $\nabla^{E\oplus F}=\nabla^E\oplus\nabla^F$. 
Then, with respect to the splitting  
$\sS\otimes(E\oplus F)=(\sS\otimes E)\oplus(\sS\otimes F)$, we have $\sD_{\nabla^{E\oplus F}}=\sD_{\nabla^E}\oplus\sD_{\nabla^F}$, 
so that $\ker\sD^{\pm}_{\nabla^{E\oplus F}}=\ker\sD^{\pm}_{\nabla^E}\oplus\ker\sD^{\pm}_{\nabla^{F}}$. Thus,
\begin{equation*}
\ind\sD_{\nabla^{E\oplus F}}=\ind\sD_{\nabla^E}+\ind\sD_{\nabla^F}.
\end{equation*}
It follows from all these that the index $\ind \sD_{\nabla^{E}}$ depends only on the $K$-theory class of $E$, and there actually 
is a well-defined additive map,
\begin{equation}
\ind_{\sD}: K^0(M)\rightarrow\Z,
\label{eq:Dirac.index-map1}
\end{equation}
such that, for any Hermitian vector bundle $E$ equipped with a Hermitian connection $\nabla^E$, we have
\begin{equation}
\ind_{\sD}[E]=\ind \sD_{\nabla^E}.
\label{eq:Dirac.index-map2}
\end{equation}

Let $H_{[0]}(M, \C)=\bigoplus_{i\geq 0} H_{2i}(M, \C)$ be the even de Rham homology of $M$ and $H^{[0]}(M, 
\C)=\bigoplus_{i\geq 0} H^{2i}(M, \C)$ its even de Rham cohomology. Composing the natural duality pairing between 
$H_{[0]}(M, \C)$ and $H^{[0]}(M, \C)$ with the Chern character map $\Ch: K^0(M)\rightarrow H^{[0]}(M, \C)$, we obtain a bilinear pairing,
\begin{equation}
\label{eq:pairing.deRhamHomology.KTheory}
\acou{\cdot}{\cdot}: H_{[0]}(M, \C)\times K^0(M)\rightarrow\C,
\end{equation}
so that, for any closed even de Rham current $C$ and any vector bundle $E$ over $M$, we have
\begin{equation*}
\acou{[C]}{[E]}=\acou{C}{\Ch(F^E)},
\end{equation*}
where $F^E$ is the curvature of any connection on $E$. Then the Atiyah-Singer index formula can be rewritten as 
\begin{equation}
\ind D_{\nabla^{E}}=(2i\pi)^{-\frac{n}{2}}\acou{\hat{A}(R^M)^{\wedge}}{\Ch(F^E)}=(2i\pi)^{-\frac{n}{2}}\acou{\left[\hat{A}(R^M)^{\wedge}\right]}{[E]},
\label{eq:Dirac.Atiyah-Singer}
\end{equation}
where $[\hat{A}(R^M)^{\wedge}]$ is the homology class of the Poincar\'e dual of the $\hat{A}$-form $\hat{A}(R^M)$.

Finally, we stress out that the definition of $\sD_{\nabla^E}$ does not require the connection $\nabla^E$ to be Hermitian, 
and so the construction of $\sD_{\nabla^E}$ holds for any connection $\nabla^E$ on $E.$
In this general case, the operator $\sD_{\nabla^E}$ need not be selfadjoint, but it still is Fredholm and of the form~(\ref{eq:Coupled.Dirac.Operator}). 
\emph{A priori} we could consider the two Fredholm indices $\ind\sD_{\nabla^E}^+$ and $\ind\sD_{\nabla^E}^-$ separately. 
When $\nabla^E$ is Hermitian, we have $\ind\sD_{\nabla^{E}}^-=\ind(\sD_{\nabla^E}^+)^*=-\ind\sD_{\nabla^E}^+.$
The value of these indices are independent of the the choice of Hermitian connection, so we see that 
$\ind\sD_{\nabla^E}^+=-\ind\sD_{\nabla^E}^-$ even when $\nabla^{E}$ is not Hermitian.
In any case, we equivalently could define the index of $\sD_{\nabla^{E}}$ by
\begin{equation}
\ind\sD_{\nabla^E}=\frac12\left(\ind\sD_{\nabla^E}^+-\ind\sD_{\nabla^E}^-\right).
\label{eq:Dirac.average-index}
\end{equation}

\section{The Index Map of a Twisted Spectral Triple}
\label{sec:IndexTwistedSpectralTriple}
Let $(\cA, \cH, D)_{\sigma}$ be a twisted spectral triple. As observed by Connes-Moscovici~\cite{CM:TGNTQF},  the datum of $(\cA, \cH, 
D)_{\sigma}$ gives rise to a well-defined index map $\ind_{D,\sigma}:K_{0}(\cA)\rightarrow \frac{1}{2}\Z$. 
The 
definition of the index map in~\cite{CM:TGNTQF} is based on the observation that the phase $F=D|D|^{-1}$ defines an \emph{ordinary} Fredholm module 
over $\cA$ (namely, the pair $(\cH,F)$). The index map is then defined in terms of compressions of $F$ by 
idempotents. As we shall now explain, we also can define the index map by using twisted versions of the compression of 
the operator $D$ by idempotents. This construction is actually a special case of the coupling of $D$ by $\sigma$-connections which 
will be described in the next section.  We still need to deal with this special case in order to carry out the more general 
construction in the next section. 
 
Let $e$ be an idempotent in $M_{q}(\cA)$, $q \in \N$. We regard $e\cH^{q}$ as a closed subspace of the Hilbert space 
$\cH^{q}$, so that $e\cH^{q}$ is a Hilbert space with the induced inner product. As the action of $\cA$ on $\cH$ is by \emph{even} 
operators on $e\cH^{q}\cap (\cH^{\pm})^{q}=e(\cH^{\pm})^{q}$, and so we have the orthogonal splitting 
$e\cH^{q}=e(\cH^{+})^{q}\oplus e(\cH^{-})^{q}$. In addition, the action of $\cA$ preserves the 
domain of $D$, so we see that $e(\dom (D))^{q}=(\dom (D))^{q}\cap e\cH^{q}$. We then let $D_{e,\sigma}$ be  the unbounded operator  from 
$e\cH^{q}$ to $\sigma(e)\cH^{q}$ given by
\begin{equation}
    D_{e,\sigma}:=\sigma(e) (D\otimes 1_{q}), \qquad \dom (D_{e,\sigma})=e(\dom (D))^{q}.
    \label{eq:Index.Desigma}
 \end{equation}   
 We note that, as $D$ is an odd operator, with respect to the orthogonal splitting 
$\cH=\cH^{+}\oplus \cH^{-}$ it takes the form, 
\begin{equation}
    D= 
    \begin{pmatrix}
        0 & D^{-} \\
        D^{+} & 0
    \end{pmatrix}, \qquad D^{\pm}:\dom(D)\cap \cH^{\pm}\rightarrow \cH^{\mp}.
    \label{eq:TST.decompositionD}
\end{equation}
Incidentally,  with respect to the orthogonal splittings $e\cH^{q}=e (\cH^{+})^{q}\oplus e (\cH^{-})^{q}$ and 
 $\sigma(e)\cH^{q}=\sigma(e) (\cH^{+})^{q}\oplus \sigma(e)(\cH^{-})^{q}$ the operator 
$D_{e,\sigma}$ takes the form, 
\begin{equation*}
    D_{e,\sigma}= 
    \begin{pmatrix}
        0 & D^{-}_{e,\sigma} \\
        D_{e,\sigma}^{+} & 0
    \end{pmatrix}, \qquad D_{e,\sigma}^{\pm}=\sigma(e)(D^{\pm}\otimes 1_{q}).
\end{equation*}

In order to determine the adjoint of $D_{e,\sigma}$ we make the following observation. 

\begin{lemma}
   Let $S_{e}:e\cH^{q}\rightarrow e^{*}\cH^{q}$ be the restriction to $e\cH^{q}$ of $e^{*}$ (which we represent as an operator on 
   $\cH^{q}$). Then $S_{e}$ is a linear isomorphism from $e\cH^{q}$ onto $e^{*}\cH^{q}$ such that 
    \begin{equation}
        \acou{S_{e}\xi_{1}}{\xi_{2}}=\acou{\xi_{1}}{\xi_{2}} \qquad \forall \xi_{j}\in e\cH^{q}.
        \label{eq:IsoEHEstarH}
    \end{equation}    
\end{lemma}
 \begin{proof}
     Let $\xi_{j}\in e\cH^{q}$, $j=1,2$. Then
     \begin{equation*}
         \acou{S_{e}\xi_{1}}{\xi_{2}}=\acou{e^{*}\xi_{1}}{\xi_{2}}=\acou{\xi_{1}}{e\xi_{2}}=\acou{\xi_{1}}{\xi_{2}}. 
     \end{equation*}In particular, when $\xi_{2}=\xi_{1}$ we get $ \acou{S_{e}\xi_{1}}{\xi_{1}}=\|\xi_{1}\|^{2}$, which 
     shows that $S_{e}$ is one-to-one. 
     
     Let $\eta\in e^{*}\cH^{q}$. Then $\acou{\eta}{\cdot}_{|e\cH^{q}}$ is a continuous linear form on $e\cH^{q}$, so 
     there exists $\tilde{\eta}\in e\cH^{q}$ so that $\acou{\eta}{\xi}=\acou{\tilde{\eta}}{\xi}$ for all $\xi \in 
     e\cH^{q}$. Therefore, for all $\xi\in \cH^{q}$,
     \begin{equation*}
         \acou{\eta}{\xi}=\acou{e^{*}\eta}{\xi}=\acou{\eta}{e\xi}=\acou{\tilde{\eta}}{e\xi}=\acou{e^{*}\tilde{\eta}}{\xi}.
     \end{equation*}Thus $\eta=e^{*}\tilde{\eta}=S_{e}\tilde{\eta}$. This shows that $S_{e}$ is onto. As $S_{e}$ is 
     one-to-one we then deduce that $S_{e}$ is a linear isomorphism. The proof is complete.
 \end{proof}
 
 The above lemma holds for the idempotent $\sigma(e)$ as well. In what follows, we denote by $S_{\sigma(e)}$ the linear isomorphism  from  $\sigma(e)\cH^{q}$ to 
 $ \sigma(e)^{*}\cH^{q}$ induced by $\sigma(e)^{*}$.

\begin{lemma}
    Let $D_{e,\sigma}^{*}$ be the adjoint of $D_{e,\sigma}$. Then
    \begin{equation}
    \label{eq:D_esigmaAdjoint}
        D_{e,\sigma}^{*}=S_{e}^{-1}D_{\sigma(e)^{*},\sigma}S_{\sigma(e)}.
    \end{equation}
\end{lemma}
\begin{proof}
Let $D_{e,\sigma}^{\dagger}$ be the operator given by the following graph,
    \begin{equation*}
    G(D_{e,\sigma}^{\dagger})=\left\{ (\xi,\eta)\in \sigma(e)^{*}\cH^{q}\times e^{*}\cH^{q}; \ \acou{\xi}{D_{e,\sigma}\zeta}=\acou{\eta}{\zeta} \quad \forall 
    \zeta \in \dom (D_{e,\sigma})\right\}.
\end{equation*}We note that the graph of $D_{e,\sigma}^{*}$ is 
 \begin{equation*}
      G(D_{e,\sigma}^{*})=\left\{ (\xi,\eta)\in \sigma(e)\cH^{q}\times e\cH^{q}; \ \acou{\xi}{D_{e,\sigma}\zeta}=\acou{\eta}{\zeta} \quad \forall 
    \zeta \in \dom (D_{e,\sigma}) \right\}.   
 \end{equation*} It then follows from~(\ref{eq:IsoEHEstarH}) that a pair $(\xi,\eta)\in \sigma(e)\cH^{q}\times e\cH^{q}$ is 
contained in $G(D_{e,\sigma}^{*})$ if and only if $(S_{\sigma(e)}\xi,S_{e}\eta)$ lies in 
$G(D_{e,\sigma}^{\dagger})$. That is, $S_{e}D_{e,\sigma}^{*}=D_{e,\sigma}^{\dagger}S_{\sigma(e)}$. Therefore, showing 
that $D_{e,\sigma}^{*}=S_{e}^{-1}D_{\sigma(e)^{*},\sigma}S_{\sigma(e)}$ is equivalent to showing that the operators 
$D_{e,\sigma}^{\dagger}$ and $D_{\sigma(e)^{*},\sigma}$ agree. 

 Let $(\xi,\eta)\in G(D_{e^{*}})$. For all $\zeta \in \dom (D_{e})$, we have
    \begin{equation*}
     \acou{D_{\sigma(e)^{*},\sigma}\xi}{\zeta} =\acou{\sigma(e)^{*}(D\otimes 
        1_{q})e^{*}\xi}{\zeta}=\acou{\xi}{\sigma(e)(D\otimes 1_{q})e\zeta} =\acou{\xi}{D_{e,\sigma}\zeta}.
    \end{equation*}Thus $(\xi,D_{\sigma(e)^{*},\sigma}\xi)$ belongs to $G(D_{e,\sigma}^\dagger)$. Then $G(D_{\sigma(e)^{*},\sigma})$ is contained in 
    $G(D_{e}^\dagger)$, i.e., $D_{e}^\dagger$ is an extension of $D_{\sigma(e)^{*},\sigma}$. 

   Let $(\xi,\eta)\in G(D_{e}^\dagger)$ and set $R:=\sigma(e)(D\otimes 1_{q})(1-e)$. We note that 
    \begin{equation*}
       R=\sigma(e)(D\otimes 1_{q})(1-e)=\sigma(e)\left\{(1-\sigma(e))(D\otimes 1_{q})-[D\otimes 
       1_{q},e]_{\sigma}\right\}=-\sigma(e)[D\otimes 1_{q},e]_{\sigma}.
    \end{equation*}
    Thus, $R$ is a bounded operator. Incidentally, its adjoint $R^*$ is a bounded operator as well. Set 
    $\tilde{\eta}=\eta+(1-e^{*})R^{*}S_{\sigma(e)}^{-1}\xi$ and let $ \zeta \in (\dom 
D)^{q}$.  As $e\zeta\in \dom (D_{e,\sigma})$ and the subspaces $e^{*}\cH^{q}$ and $(1-e)\cH^{q}$ are orthogonal to each other, we have
  \begin{equation*}
     \acou{\tilde{\eta}}{\zeta}  = \acou{\eta}{e\zeta}+ \acou{\eta}{(1-e)\zeta}+\acou{(1-e^{*})R^{*}S_{\sigma(e)}^{-1}\xi}{\zeta}= 
     \acou{\xi}{D_{e,\sigma}(e\zeta)}+\acou{S_{\sigma(e)}^{-1}\xi}{R((1-e)\zeta)}. 
 \end{equation*}
 Moreover, as $\xi\in \sigma(e)^{*}\cH^{q}$ and $R((1-e)\zeta)\in\sigma(e)\cH^{q}$,  using~(\ref{eq:IsoEHEstarH}) we see that 
 \(\acou{S_{\sigma(e)}^{-1}\xi}{R((1-e)\zeta)}$ agrees with $\acou{\xi}{R((1-e)\zeta)}.\) Therefore, $ \acou{\tilde{\eta}}{\zeta}$ is equal to
    \begin{equation*}
      \acou{\xi}{D_{e,\sigma}(e\zeta)+R((1-e)\zeta)}=\acou{\xi}{\sigma(e)(D\otimes 
      1_{q})\zeta}=\acou{\sigma(e)^{*}\xi}{(D\otimes 1_{q})\zeta}= \acou{\xi}{(D\otimes 1_{q})\zeta}. 
   \end{equation*}This shows that $(\xi,\tilde{\eta})$ lies in the graph of the operator $(D\otimes 1_{q})^{*}$, which agrees with 
   $D\otimes 1_{q}$ since $D$ is selfadjoint. Thus $(\xi,\tilde{\eta})$ lies in the graph of $D\otimes 1_{q}$. Therefore, we see that 
   $\xi$ is contained in both $(\dom (D))^{q}$ and $\sigma(e)^{*}\cH^{q}$, so it lies in $(\dom (D))^{q} \cap 
   \sigma(e)^{*}\cH^{q}=\sigma(e)^{*}(\dom (D))^{q}=\dom(D_{\sigma(e)^{*},\sigma})$. This shows that $\dom 
   D_{e,\sigma}^{\dagger}$ is contained in $\dom (D_{\sigma(e)^{*},\sigma})$. 
   As $D_{e,\sigma}^\dagger$ is an extension of $D_{\sigma(e)^{*},\sigma}$ we then deduce that the two operators agree. As explained above this proves that $D_{e,\sigma}^{*}=S_{e}^{-1}D_{\sigma(e)^{*},\sigma}S_{\sigma(e)}$. 
   The proof is complete.     
\end{proof}

\begin{lemma}\label{lem:index.Desigma-Fredholm} 
    The operator $D_{e,\sigma}$ is closed and Fredholm, and we have
\begin{equation}
\label{eq:IndexD_eSigmaPM}
    \ind D_{e,\sigma}^{\pm}=\dim \ker D_{e,\sigma}^{\pm}-\dim \ker D_{\sigma(e)^{*},\sigma}^{\mp}.
\end{equation}
\end{lemma}
\begin{proof}Substituting $\sigma(e)^{*}=\sigma^{-1}(e^{*})$ for $e$ in~(\ref{eq:D_esigmaAdjoint}) shows that 
    $D_{\sigma(e)^{*},\sigma}^{*}=S^{-1}_{\sigma(e)^{*}}D_{e,\sigma}S_{e^{*}}$, i.e., 
    $D_{e,\sigma}=S_{\sigma(e)^{*}}D_{\sigma(e)^{*},\sigma}^{*}S_{e^{*}}^{-1}$. As $D_{\sigma(e)^{*},\sigma}^{*}$  is a 
    closed operator and the operators $S_{\sigma(e)^{*}}$ and $S_{e^{*}}^{-1}$ are bounded, we see that $D_{e,\sigma}$ 
    is a closed operator. 
    
Let $D^{-1}$ be the partial inverse of $D$ and $P_{0}$ is the orthogonal projection onto $\ker D$. Set $Q_{e,\sigma}:=e(D^{-1}\otimes 1_{q})$, which we regard as a bounded 
operator from $\sigma(e)\cH^{q}$ to $e\cH^{q}$. Note that $Q_{e,\sigma}$ is a compact operator. Moreover, on $\sigma(e)\cH^{q}$ we have 
\begin{align*}
 D_{e,\sigma} Q_{e,\sigma}=\sigma(e)(D\otimes 1_{q})e(D^{-1}\otimes 1_{q})=&\sigma(e)+\sigma(e)[D\otimes 
 1_{q},e]_{\sigma}(D^{-1}\otimes 1_{q})-\sigma(e)(P_0\otimes1_q)\\
 =&1+ e[D\otimes 1_{q},e](D^{-1}\otimes 1_{q})-\sigma(e)(P_0\otimes1_q). 
\end{align*}
Likewise, on $e(\dom (D))^{q}$ we have 
\begin{equation*}
    Q_{e,\sigma}D_{e,\sigma}=e(D^{-1}\otimes 1_{q})\sigma(e)(D\otimes 1_{q})= 1-eD^{-1}[D\otimes 1_{q},e]_{\sigma}-e(P_0\otimes1_q)e.
\end{equation*}As $D^{-1}, P_0$ are compact operators and $[D\otimes 1_{q},e]_{\sigma}$ is bounded, we see that $Q_{e, \sigma}$ 
inverts $D_{e, \sigma}$ modulo compact operators. It then follows that $D_{e, \sigma}$ is a Fredholm operator.

We note that $S_{e}$ and $S_{\sigma(e)}$ are even operators, so the equality~(\ref{eq:D_esigmaAdjoint}) means that $\left( 
D^{\pm}_{e,\sigma}\right)^{*}=S_{e}^{-1}D^{\mp}_{\sigma(e)^{*},\sigma}S_{\sigma(e)}^{-1}$. Therefore, the operator 
$S_{e}$ induces an isomorphism $\ker D^{\mp}_{\sigma(e)^{*},\sigma}\simeq \ker \left( 
D^{\pm}_{e,\sigma}\right)^{*}$. Thus,
\begin{equation*}
     \ind D_{e,\sigma}^{\pm}=\dim \ker D_{e,\sigma}^{\pm}-\dim \ker \left( 
D^{\pm}_{e,\sigma}\right)^{*}=\dim \ker D_{e,\sigma}^{\pm}-\dim \ker D_{\sigma(e)^{*},\sigma}^{\mp}.
\end{equation*}
The proof is complete.
 \end{proof}

We define the index of $D_{e,\sigma}$ by
\begin{equation}
    \ind D _{e,\sigma}:= \frac{1}{2} \left(\ind  D_{e,\sigma}^{+}  - \ind  D_{e,\sigma}^{-}\right). 
    \label{eq:Index-index-Desigma}
\end{equation}
Thanks to~(\ref{eq:IndexD_eSigmaPM}) we have
\begin{equation}
\label{eq:IndexDesigma}
  \ind D _{e,\sigma}= \frac{1}{2} \left(  \dim \ker D_{e,\sigma}^{+}+ \dim \ker D_{\sigma(e)^{*},\sigma}^{+} -
  \dim \ker D_{e,\sigma}^{-}-\dim \ker D_{\sigma(e)^{*},\sigma}^{-}\right).
\end{equation}
In particular, when $\sigma(e)^{*}=e$ we get
\begin{equation*}
    \ind D_{e,\sigma}= \dim \ker D_{e,\sigma}^{+}-\dim \ker D_{e,\sigma}^{-}. 
\end{equation*}

Let $g\in \op{Gl}_{q}(\cA)$ and set $\hat{e}=g^{-1}eg$. On $(\dom (D))^{q}$ the operator  $\sigma(\hat{e})(D\otimes 
1_{q})\hat{e}$ agrees with
\begin{multline*}
   \sigma(g)^{-1}\sigma(e)\sigma(g) (D\otimes 1_{q})g^{-1}eg \\
   =  \sigma(g)^{-1}\sigma(e)\sigma(g) \sigma(g^{-1}) (D\otimes 1_{q})eg  
    + \sigma(g)^{-1}\sigma(e)\sigma(g)[D\otimes 1_{q},g^{-1}]_{\sigma}eg \\
     =  \sigma(g)^{-1}D_{e,\sigma}g+  \sigma(g)^{-1}\sigma(e)\sigma(g)[D\otimes 1_{q},g^{-1}]_{\sigma}eg. 
\end{multline*}
As $[D\otimes 1_{q},g^{-1}]_{\sigma}$ is a bounded operator, we see that $D^{\pm}_{\hat{e},\sigma}$ and 
$\sigma(g)^{-1}(D_{e,\sigma}^{\pm})g$ agree up to a bounded operator. 
It then follows that $D^{\pm}_{\hat{e},\sigma}$ and $D_{e,\sigma}^{\pm}$ have the same Fredholm index. Thus, 
\begin{equation}
    \ind D_{g^{-1}eg,\sigma}=\ind D_{e,\sigma} \qquad \forall g \in \op{Gl}_{q}(\cA).
    \label{eq:Index.similarity-twisted}
\end{equation}
Moreover, if $e'\in M_{q'}(\cA)$ is another idempotent, then, with respect to the splittings $(e\oplus e')(\cH^{\pm})^{q}= 
e(\cH^{\pm})^{q}\oplus e'(\cH^{\pm})^{q'}$ and $\sigma(e\oplus e')(\cH^{\pm})^{q}= 
\sigma(e)(\cH^{\pm})^{q}\oplus \sigma(e')(\cH^{\pm})^{q'}$, we have $D_{e\oplus e',\sigma}^{\pm}=D_{e,\sigma}^{\pm}\oplus 
D_{e',\sigma}^{\pm}$. We then see that $\ind D_{e\oplus e',\sigma}^{\pm}=\ind D_{e,\sigma}^{\pm} + \ind
D_{e',\sigma}^{\pm}$. Thus,
\begin{equation*}
    \ind D_{e\oplus e',\sigma}=\ind D_{e,\sigma}+\ind D_{e',\sigma}. 
\end{equation*}
Therefore, we arrive at the following statement. 

\begin{proposition}[\cite{CM:TGNTQF}] There is a unique additive map $\ind_{D,\sigma}:K_{0}(\cA)\rightarrow \frac{1}{2} \Z$ such that
\begin{equation}
    \ind_{D,\sigma}[e]=\ind D_{e,\sigma} \qquad \forall e \in M_{q}(\cA), \ e^{2}=e.
    \label{eq:Index.index-map}
\end{equation}    
\end{proposition}

As pointed out in Remark~\ref{rmk:TST.sigma-involution} the fact that $\sigma(a)^{*}=\sigma^{-1}(a^{*})$ for all $a\in \cA$ means that the map 
$a\rightarrow \sigma(a)^{*}$ is an  involutive antilinear anti-automorphism of $\cA$, which we shall call the 
$\sigma$-involution. An element $a \in \cA$ is selfadjoint with respect to this involution if and only if 
$\sigma(a)^{*}=a$. As~(\ref{eq:IndexDesigma}) shows that when $\sigma(e)^{*}=e$ the index of $D_{e,\sigma}$ is an integer. 
While an idempotent in $M_{q}(\cA)$ is always conjugate to a selfadjoint idempotent, in general it need not be 
conjugate to an idempotent which is selfadjoint with respect to the $\sigma$-involution.  Nevertheless, this property 
holds under a further assumption on the automorphism $\sigma$. 

\begin{definition}\label{def:Index.ribbon}
 The automorphism $\sigma$ is called ribbon when it has a square root in the sense there is an automorphism $\tau:\cA\rightarrow  \cA$ such that
\begin{equation}
    \sigma(a)=\tau(\tau(a)) \quad \text{and} \quad \tau(a)^{*}=\tau^{-1}(a^{*}) \qquad \text{for all $a\in \cA$}.
    \label{eq:Index.square-root-sigma}
\end{equation}  
\end{definition}

\begin{remark}
    The terminology \emph{ribbon} is used in analogy with the theory of quantum groups, where a quasi-triangular 
    Hopf algebra is called ribbon when a certain element admits a square root compatible with the quasi-triangular 
    structure (see, e.g., \cite{Ma:QGP}). 
\end{remark}

\begin{lemma}
\label{lm:CriteriaSigmaAdjointIntegerIndex}
 Assume that the  automorphism $\sigma$ is ribbon.Then
    \begin{enumerate}
        \item[(i)]  Any idempotent $e\in M_{q}(e)$, $q \in \N$, is conjugate to an idempotent which is selfadjoint with 
        respect to the $\sigma$-involution.
    
        \item[(ii)]  The index map $\ind_{D,\sigma}$ is integer-valued. 
    \end{enumerate}
\end{lemma}
\begin{proof}
 The 2nd part follows by combining the first part with~(\ref{eq:Index.similarity-twisted}) and~(\ref{eq:IndexD_eSigmaPM}). Therefore, we only need to prove the first part. In addition, 
 without any loss of generality, we may assume that $q=1$ in the first part. Thus let $e$ be an idempotent element of 
 $\cA$. 
 
 Let  us briefly recall  how we construct a selfadjoint idempotent in $\cA$ which is conjugate to $e$ (see, e.g., 
 \cite[Prop.~4.6.2]{Bl:KTOA} for more details). Set $a=e-e^{*}$ and $b=1+aa^{*}$. Observing that $b$ is an invertible element of $\cA$ which commutes with $e$ and $e^{*}$, define 
 $p=ee^{*}b^{-1} $. It can be checked that $p^{2}=p^{*}=p$, i.e., $p$ is a selfadjoint idempotent of $\cA$. Moreover, 
 if we set $g=1-p+e$, then $g$ has inverse $g^{-1}=1+p-e$ and $g^{-1}pg=e$. 
 
 We remark that the above construction holds \emph{verbatim} if we replace the involution $a\rightarrow a^{*}$ by 
 \emph{any} other involutive antilinear anti-automorphism of $\cA$, provided it can be shown that the corresponding operator 
 $b$ is invertible (in $\cA)$. Thus, if we substitute $\sigma(e)^{*}$ for $e^{*}$ and we assume that $b:=1+\sigma(a)^{*}a$ is invertible, where 
 $a=e-\sigma(e)^{*}$, then $p:=e\sigma(e)^{*}b^{-1}$ is such that $p^{2}=\sigma(p)^{*}=p$ 
 and $g^{-1}pg=e$ where $g:=1-p+e$ has inverse $g^{-1}=1+p-e$. Therefore, the main question at stake is to 
 show that $b$ is invertible. 
 
 Let $\tau$ be a square root of $\sigma$ in the sense of~(\ref{eq:Index.square-root-sigma}).  Then
 \(
     \tau(\sigma(a)^{*})=\tau\circ \sigma^{-1}(a^{*})=\tau^{-1}(a^{*})=\tau(a)^{*}.
 \) Thus, 
  \begin{equation*}
      \tau(b)=1+\tau(\sigma(a)^{*}a)=1+\tau(\sigma(a)^{*})\tau(a)=1+ \tau(a)^{*}\tau(a).
 \end{equation*}As $ \tau(a)^{*}\tau(a)$ is a positive element of $\cA$ we see that $\tau(b)$ is invertible, and hence 
 $b$ is invertible as well. The proof is complete. 
\end{proof}

As we shall now see, the ribbon condition~(\ref{eq:Index.square-root-sigma}) is satisfied by the automorphisms occurring in the main examples of twisted spectral triples. 

\begin{example}
    Assume that $\sigma(a)=kak^{-1}$ where $k$ is a positive invertible element of $\cA$. Then $\sigma$ has the square root 
    $\tau(a)=k^{\frac{1}{2}}ak^{-\frac{1}{2}}$. We note that $k^{\frac{1}{2}}$ is an element of $\cA$ since $\cA$ is 
    closed under holomorphic functional calculus. 
\end{example}

More generally, we have the following.
\begin{example}
\label{ex:ribbon.1PG}
   Suppose that $\sigma$ agrees with the value at $t=-i$ of the analytic extension of a strongly continuous one-parameter group of 
    isometric $*$-isomorphisms $(\sigma_{t})_{t\in \R}$. This condition is called (1PG) in~\cite{CM:TGNTQF}. In this 
    case $\sigma$ is ribbon with square 
    root~$\tau:=\sigma_{\left|t=-i/2\right.}$.      
 \end{example}

 \begin{remark}
By a result of Bost~\cite{Bo:POKTSDNC} the analytic extension of a 
     strongly continuous one-parameter group of isometric isomorphisms on an involutive Banach algebra always exists on a dense subalgebra which is closed under 
    holomorphic functional calculus.     
\end{remark}
 
 \begin{remark}
     Connes-Moscovici~\cite{CM:TGNTQF} showed when the condition (1PG) holds $\ind D_{e,\sigma}^{+}$ is equal to $-\ind 
     D_{e,\sigma}^{-}$, so that in this case the index map $\ind_{D,\sigma}$ is integer-valued. 
 \end{remark}

 \begin{example}
     The ribbon condition is also satisfied by the automorphism $\sigma_{g}$ appearing in the construction of the 
     conformal Dirac spectral triple. From the definition~(\ref{eq:TwistedST.automorphism-conformal-DiracST}) of $\sigma_{g}$ we see that a square root satisfies the ribbon condition~(\ref{eq:Index.square-root-sigma}). Indeed, a square root satisfying~(\ref{eq:Index.square-root-sigma}) is 
   given by the automorphism $\tau_{g}$ defined by
\begin{equation*}
    \tau_{g}(fU_{\phi}):=\sqrt{k_{\phi}}fU_{\phi} \qquad \forall f \in C^{\infty}(M) \ \forall \phi \in G.
\end{equation*}In fact, $\sigma_g$ satisfies the (1PG) condition with respect to the one-parameter group of isometric $*$-isomorphisms 
$\sigma_t$, $t\in \R$, given by $\sigma_t(fU_\phi)=k_\phi^{it}fU_\phi$. 
 \end{example}
 
\section{Index Map and $\sigma$-Connections}
\label{sec:IndexMapSigmaConnections}
In this section, we present a geometric description of the index map of a twisted spectral triple in terms of  
couplings by $\sigma$-connections on finitely generated projective modules (i.e, noncommutative vector bundles). As we 
shall explain in the next section, this description makes it much more transparent the analogy 
with the construction of the index map in the commutative case in terms of Dirac operators coupled with connections (see, e.g, \cite{BGV:HKDO}).  
We refer to~\cite{Mo:EIPDNCG} for a similar description of the index map in the case of ordinary spectral triples. 

Throughout this section we let $(\cA,\cH,D)_{\sigma}$ be a twisted spectral triple. 
In addition, we let $\cE$ be a finitely generated projective right module over $\cA$, i.e., $\cE$ is the direct summand of a finite rank free module 
$\cE_{0}\simeq \cA^{q}$.  In order to define $\sigma$-connections we 
need to introduce the notion of ``$\sigma$-translation''. 

\begin{definition}\label{def:sigmaTranslate}
    A $\sigma$-translation of $\cE$ is given by a pair $(\cE,\sigma^{\cE})$, where
    \begin{itemize}
        \item[(i)] $\cE^{\sigma}$ is   finitely generated projective right module over $\cA$, called $\sigma$-translate.  
    
        \item[(ii)]  $\sigma^{\cE}:\cE\rightarrow \cE^{\sigma}$ is a $\C$-linear isomorphism such that
                  \begin{equation}
                       \sigma^{\cE}(\xi a)=\sigma^{\cE}(\xi)\sigma(a)  \qquad \text{for all $\xi\in \cE$ and $a\in \cA$}.
	        \label{eq:sigma^E(xi a)}
                  \end{equation}    
    \end{itemize}
\end{definition}

\begin{remark}
    The condition~(\ref{eq:sigma^E(xi a)}) means that $\sigma^{\cE}$ is a right module isomorphism from $\cE$ onto $\cE^{(\sigma)}$, where $\cE^{(\sigma)}$ is 
$\cE^{\sigma}$ equipped with the action $(\xi,a)\rightarrow \xi\sigma(a)$. In particular, when $\sigma=\op{id}$ a 
$\sigma$-translation of $\cE$ is simply given by a right-module isomorphism $\sigma^{\cE}:\cE\rightarrow \cE^{\sigma}$. 
Therefore, a canonical choice of $\sigma$-translation is to take $(\cE, \op{id})$. This will always the choice we shall 
make when $\sigma=\op{id}$. 
\end{remark}

\begin{remark}\label{rmk:sigmae}.
   Suppose that $\cE=e\cA^{q}$, for some idempotent $e\in M_{q}(\cA)$, $q\geq 1$. The automorphism $\sigma$ lifts to $\cA^{q}$ by 
\begin{equation*}
    \sigma(\xi)=(\sigma(\xi_{j})) \qquad \forall \xi=(\xi_{j})\in \cA^{q}. 
\end{equation*}
Note that $\sigma$ is a $\C$-linear isomorphism of $\cA^{q}$ onto itself and maps $e\cA^{q}$ onto 
$\sigma(e)\cA^{q}$, and so it induces a $\C$-linear isomorphism $\sigma^{e}:e\cA^{q}\rightarrow \sigma(e)\cA^{q}$. 
Moreover, for all $\xi=(\xi_{j})\in e\cA^{q}$ and $a\in \cA$,
\begin{equation}
    \sigma_{e}(\xi a)=\sigma\left( 
    (\xi_{j}a)\right)=\left(\sigma(\xi_{j}a)\right)=\left(\sigma(\xi_{j})\right)\sigma(a)=\sigma_{e}(\xi)\sigma(a). 
    \label{eq:sigma_e(xi a)}
\end{equation}That is, $\sigma_{e}$ satisfies~(\ref{eq:sigma^E(xi a)}). Therefore, the pair $(\sigma(e)\cA^{q},\sigma_{e})$ is a $\sigma$-translation of $e\cA^{q}$. This 
will be our canonical choice of $\sigma$-translation when $\cE$ is of the form $e\cA^{q}$, with $e^{2}=e\in 
M_{q}(\cA)$, $q\geq 1$.
\end{remark}

\begin{remark}\label{rmk:sigma-translate}
    In general, a $\sigma$-translation is obtained as follows. By definition $\cE$ is a direct-summand of a free module 
    $\cE_{0}\simeq \cA^{q}$. Let $\phi:\cE_{0}\rightarrow \cA^{q}$ be a right module isomorphism. The 
image of $\cE$ by $\phi$ is a right module of the form $e\cA^{q}$ for some idempotent $e \in M_{q}(\cA)$. Set $\cE^{\sigma}:=\phi^{-1}\left( 
\sigma(e)\cA^q\right)$; this is a direct summand of $\cE_{0}$. The isomorphism $\phi$ induces isomorphisms of right modules,
\begin{equation*}
    \phi_{e}:\cE\longrightarrow e\cA^{q}\qquad \text{and}\qquad \phi_{\sigma(e)}:\cE^{\sigma}\longrightarrow \sigma(e)\cA^{q}.
\end{equation*}
Set $\sigma^{\cE}=\left( \phi_{\sigma(e)}\right)^{-1}\circ \sigma^{e}\circ \phi_{e}$, where 
$\sigma_{e}:e\cA^{q}\rightarrow \sigma(e)\cA^{q}$ is the $\sigma$-lift introduced in Remark~\ref{rmk:sigmae}. Then $\sigma^{\cE}$ is a 
$\C$-linear isomorphism from $\cE$ onto $\cE^{\sigma}$. Moreover, using~(\ref{eq:sigma^E(xi a)}) we see that, for all $\xi \in \cE$ and $a \in \cA$, we have
\begin{equation*}
    \sigma^{\cE}(\xi a)= \left( \phi_{\sigma(e)}\right)^{-1}\circ \sigma_{e}(\phi(\xi)a)=\left( 
    \phi_{\sigma(e)}\right)^{-1}\left(  \sigma_{e}(\xi)\sigma(a)\right)=\sigma^{\cE}(\xi )\sigma(a). 
\end{equation*}This shows that $\sigma^{\cE}$ satisfies~(\ref{eq:sigma^E(xi a)}), and so $(\cE^{\sigma},\sigma^{\cE})$ is a $\sigma$-translation of $\cE$. 

Conversely, given a $\sigma$-translation $(\cE,\sigma^{\cE})$ and a right-module isomorphism $\phi:\cE\rightarrow 
e\cA^{q}$, the map $\phi_{\sigma}= \sigma^e \circ \phi \circ \left(\sigma^{\cE}\right)^{-1}$ is a $\C$-linear 
isomorphism from  $\cE^{\sigma}$ onto $\sigma(e)\cA^{q}$ such that $\sigma^{\cE}=\left( \phi_{\sigma}\right)^{-1}\circ 
\sigma^{e}\circ \phi$. We note that~(\ref{eq:sigma^E(xi a)}) implies that
\begin{equation*}
    \left(\sigma^{\cE}\right)^{-1}(\xi a) \qquad \text{for all $\xi\in \cE$ and $a\in \cA$}.
\end{equation*}
Combining this with~(\ref{eq:sigma_e(xi a)}) shows that, for all $\xi$ and $a \in \cA$, we have 
\begin{equation*}
    \phi_{\sigma}(\xi a)=   \sigma^e \circ \phi \left( 
 \left(\sigma^{\cE}\right)^{-1}(\xi)\sigma^{-1}(a)\right)= \sigma^{e}\left( \phi \circ \left(\sigma^{\cE}\right)^{-1}(\xi) 
 \sigma^{-1}(a)\right)=  \phi_{\sigma}(\xi)a. 
\end{equation*}
Therefore $\phi_{\sigma}$ is a right-module isomorphism from $\cE^{\sigma}$ onto $\sigma(e)\cA^{q}$. 
In addition, we note that when $\sigma$ is ribbon Lemma~\ref{lm:CriteriaSigmaAdjointIntegerIndex} enables us to choose the idempotent $e$ 
so that $\sigma(e)=e^{*}$. 
\end{remark}
 
Following~\cite{CM:TGNTQF} we consider the space of twisted 1-forms, 
 \begin{equation*}
    \Omega^{1}_{D,\sigma}(\cA)=\left\{\Sigma a^{i}[D, b^{i}]_{\sigma}; a^{i}, b^{i} \in\cA \right\}.
\end{equation*}We note that $ \Omega^{1}_{D,\sigma}(\cA)$ is a subspace of $\cL(\cH)$. Moreover, it is naturally an $(\cA,\cA)$-bimodule, since
\begin{equation*}
    a^{2}(a^{1}[D,b^{1}]_{\sigma})b^{2}= a^{2}a^{1}[D,b^{1}b^{2}]_{\sigma}-a^{2}a^{1}\sigma(b^{1})[D,b^{2}]_{\sigma}
   \qquad  \forall  a^{j}, b^{j} \in \cA.
\end{equation*}
 We also have a ``twisted'' differential $d_{\sigma}:\cA \rightarrow \Omega^{1}_{D,\sigma}(\cA)$ defined by 
\begin{equation}
\label{eq:TwistedDifferential}
    d_{\sigma}a:= [D,a]_{\sigma} \qquad \forall a \in \cA.
\end{equation}
This is a $\sigma$-derivation, in the sense that
\begin{equation}
\label{eq:SigmaDerivation}
    d_{\sigma}(ab)=(d_{\sigma}a)b+\sigma(a)d_{\sigma}b\qquad \forall a, b \in \cA. 
\end{equation}

In what follows we let $(\cE^{\sigma},\sigma^{\cE})$ be a 
$\sigma$-translation of $\cE$.
\begin{definition}
\label{def:SigmaConnection}
A $\sigma$-connection on $\cE$ is a $\C$-linear map $\nabla: \cE\rightarrow 
\cE^{\sigma}\otimes_{\cA}\Omega^1_{D,\sigma}(\cA)$ such that 
\begin{equation}
\label{eq:SigmaConnectionModuleMulti}
 \nabla(\xi a)=(\nabla\xi) a+\sigma^{\cE}(\xi)\otimes d_{\sigma}a \qquad \forall \xi\in\cE \ \forall a\in\cA.   
\end{equation}    
 \end{definition}

\begin{example}
 Suppose that $\cE=e\cA^{q}$ with $e=e^{2}\in M_{q}(\cA)$. Then a natural $\sigma$-connection on $\cE$ is the \emph{Grassmannian $\sigma$-connection} $\nabla_{0}^{\cE}$ defined by 
 \begin{equation}
 \label{eq:GrassmannianSigmaConnection}
     \nabla_{0}^{\cE}\xi= \sigma(e)(d_{\sigma}\xi_{j}) \qquad \text{for all $\xi=(\xi_{j})$ in $\cE$}.
 \end{equation}    
\end{example}

\begin{lemma}
\label{lem:DifferenceSigmaConnection}
    The set of $\sigma$-connections on $\cE$ is an affine space modeled on $\op{Hom}_{\cA}(\cE,\cE^{\sigma}\otimes \Omega^{1}_{D,\sigma}(\cA))$.
\end{lemma}
\begin{proof}
    It follows from~(\ref{eq:SigmaConnectionModuleMulti}) that two $\sigma$-connections on $\cE$ differ by an element of 
    $\op{Hom}_{\cA}(\cE,\cE^{\sigma}\otimes \Omega^{1}_{D,\sigma}(\cA))$. Therefore, the only issue at stake is to 
    show that the set of $\sigma$-connections is nonempty. This is a true fact when $\cE=e\cA^{q}$ with $e=e^{2}\in 
    M_{q}(\cA)$ since in this case there is always the Grassmannian $\sigma$-connection. 
    
    In general, as shown in Remark~\ref{rmk:sigma-translate}, we always can find an idempotent $e\in M_{q}(\cA)$, $q\geq 1$, and right module isomorphisms 
    $\phi:\cE\rightarrow e\cA^{q}$ and $\phi^{\sigma}:\cE^{\sigma}\rightarrow 
    \sigma(e)\cA^{q}$ satisfying~(\ref{eq:sigma^E(xi a)}). We then can pullback to $\cE$ any connection $\nabla$ on $\cA^{q}$ to the 
    linear map  $\nabla^{\cE}: \cE\rightarrow \cE^{\sigma}\otimes_{\cA}\Omega^1_{D,\sigma}(\cA)$ defined by
    \begin{equation*}
        \nabla^{\cE}=\left(\left( \phi^{\sigma}\right)^{-1}\otimes 1_{\Omega^1_{D,\sigma}(\cA)}\right)\circ \nabla \circ\phi.
    \end{equation*}For $\xi \in \cE$ and $a\in \cA$ we have
    \begin{align*}
         \nabla^{\cE}(\xi a)= \left(\left( \phi^{\sigma}\right)^{-1}\otimes 1_{\Omega^1_{D,\sigma}(\cA)}\right)\circ 
         \nabla \left (\phi(\xi)a\right)& = \left(\nabla^{\cE}\xi\right) a+ 
         \left(\left(\phi^{\sigma}\right)^{-1}\circ\sigma\circ \phi\right)(\xi)\otimes d_{\sigma}a \\
         & =  \left(\nabla^{\cE}\xi\right) a
         +\sigma^{\cE}(\xi)\otimes d_{\sigma}a.
    \end{align*}Thus $\nabla^{\cE}$ is a $\sigma$-connection on $\cE$. The proof is complete. 
\end{proof}

In what follows we denote by $\cE'$ the dual $\cA$-module $\Hom_{\cA}(\cE,\cA)$. 
 
\begin{definition}
     A Hermitian metric on $\cE$ is a map $\acoup{\cdot}{\cdot}:\cE\times \cE \rightarrow \cA$ such that
     \begin{enumerate}
        \item $\acoup{\cdot}{\cdot}$ is $\cA$-sesquilinear, i.e., it is $\cA$-antilinear with respect to the first 
        variable and $\cA$-linear with respect to the second variable. 
        
         \item  $(\cdot, \cdot)$ is positive, i.e., $\acoup{\xi}{\xi}\geq 0$ for all $\xi \in \cE$.
     
         \item   $(\cdot, \cdot)$ is nondegenerate, i.e., $\xi \rightarrow \acoup{\xi}{\cdot}$ 
         is an $\cA$-antilinear isomorphism from $\cE$ onto $\cE'$.
     \end{enumerate}
 \end{definition}

  \begin{remark}
     Using (2) and a polarization argument it can be shown $(\xi_{1},\xi_{2})=(\xi_{2},\xi_{1})^{*}$ for all 
     $\xi_{j}\in \cA$.
  \end{remark}

 \begin{example}\label{ex:CanonicalHermitianStructure}
  The canonical Hermitian structure on the free module $\cA^{q}$ is given by
  \begin{equation}
  \label{eq:CanonicalHermitianStructure}
      \acoup{\xi}{\eta}_{0}=\xi_{1}^{*}\eta_{1}+\cdots + \xi_{q}^{*}\eta_{q} \qquad \text{for all $\xi=(\xi_{j})$ and 
      $\eta=(\eta_{j})$ in $\cA^{q}$}.
  \end{equation}
  \end{example}
 
\begin{lemma}
\label{lem:CanonicalHermitMetric-eA^q}
    Suppose that $\cE=e\cA^{q}$ with $e=e^{2}\in M_{q}(\cA)$. Then the canonical Hermitian metric of $\cA^{q}$ induces 
    a Hermitian metric on $\cE$.
\end{lemma}
\begin{proof}
    See Appendix~\ref{app:PfLemCanoHermMetric}. 
\end{proof}

\begin{remark}
      Let $\phi:\cE\rightarrow \cF$ be an isomorphism of finitely generated projective modules and assume $\cF$ carries 
      a Hermitian metric $\acoup{\cdot}{\cdot}_{\cF}$. Then using $\phi$ we can pullback the Hermitian metric of $\cF$ to  the 
      Hermitian metric on $\cE$ given by
      \begin{equation*}
          \acoup{\xi_{1}}{\xi_{2}}_{\cE}:=\acoup{\phi(\xi_{1})}{\phi(\xi_{2})}_{\cF} \qquad \forall \xi_{j} \in \cE.
      \end{equation*}
      In particular, if we take $\cF$ to be of the form $e\cA^{q}$ with $e=e^{2}\in M_{q}(\cA)$, then we can 
      pullback the canonical Hermitian metric $\acoup{\cdot}{\cdot}_{0}$ to a Hermitian metric on $\cE$. 
  \end{remark}  
  
 From now on we assume that $\cE$ and its $\sigma$-translate carry Hermitian metrics.  We denote by $\cH(\cE)$ the pre-Hilbert space 
consisting of $\cE\otimes_\cA \cH$ equipped with the Hermitian inner product, 
 \begin{equation}
 \label{eq:HermitianInnerProductH(E)}
     \acou{\xi_{1}\otimes \zeta_{1}}{\xi_{2}\otimes \zeta_{2}}:= \acou{\zeta_{1}}{(\xi_{1},\xi_{2})\zeta_{2}}, \qquad 
     \xi_{j}\in \cE,  \zeta_{j} \in \cH, 
 \end{equation}where $\acoup{\cdot}{\cdot}$ is the Hermitian metric of $\cE$. 
 
 \begin{lemma}
 \label{lem:H(E)topIndepHermitianMetricE}
    $ \cH(\cE)$ is a Hilbert space whose topology is independent of the choice of the Hermitian inner product of $\cE$. 
 \end{lemma}
 \begin{proof}
     See Appendix~\ref{app:H(E)topIndepHermitianMetricE}. 
 \end{proof}
 
 \begin{remark}
     In~\cite{Mo:EIPDNCG} the Hilbert space $\cH(\cE)$ is defined as the completion of $\cE\otimes_\cA \cH$ with respect to the 
     Hermitian inner product~(\ref{eq:HermitianInnerProductH(E)}). As Lemma~\ref{lem:H(E)topIndepHermitianMetricE} shows this pre-Hilbert space is already complete. 
 \end{remark}

 We note there is a natural $\Z_{2}$-grading on $\cH(\cE)$ given by
 \begin{equation}
 \label{eq:Z_2GradingH(E)}
     \cH(\cE)=\cH^{+}(\cE)\oplus \cH^{-}(\cE), \qquad \cH^{\pm}(\cE):=\cE\otimes_{\cA}\cH^{\pm}.
 \end{equation}
 
 We also form the $\Z_{2}$-graded Hilbert space 
 $\cH(\cE^{\sigma})$ as above.  In addition, we let $\nabla^{\cE}$ be a $\sigma$-connection on $\cE$. 
 Regarding $\Omega_{D,\sigma}^{1}(\cA)$ as a subalgebra of 
$\cL(\cH)$ we have a natural left-action $c:\Omega_{D,\sigma}^{1}(\cA)\otimes_{\cA}\cH \rightarrow \cH$ given by 
\begin{equation*}
    c(\omega\otimes \zeta)=\omega(\zeta) \qquad \text{for all $\omega\in \Omega_{D,\sigma}^{1}(\cA)$ and $\zeta \in 
    \cH$}.
\end{equation*}
We denote by $c\left(\nabla^{\cE}\right)$ the composition $(1_{\cE^{\sigma}}\otimes c)\circ (\nabla^{\cE}\otimes 
1_{\cH}):\cE\otimes\cH \rightarrow \cE^{\sigma}\otimes\cH$.  Thus, for $\xi\in \cE$ and $\zeta \in \cH$, and upon writing
 $\nabla^{\cE}\xi=\sum \xi_{\alpha} \otimes\omega_{\alpha}$ with $\xi_{\alpha}\in \cE^{\sigma}$ and 
 $\omega_{\alpha}\in  \Omega^{1}_{D,\sigma}(\cA)$, we have
 \begin{equation}
 \label{eq:CliffordNabla}
        c\left(\nabla^{\cE}\right)(\xi\otimes\zeta)=\sum \xi_{\alpha}\otimes\omega_{\alpha}(\zeta). 
\end{equation}

In what follows we regard the domain of $D$ as a left $\cA$-module, which is possible since the action of $\cA$ on 
$\cH$ preserves $\dom(D)$. 

\begin{definition}
 \label{def:DNablaE}
   The coupled operator $D_{\nabla^{\cE}}:\cE\otimes_{\cA} \dom(D) \rightarrow \cH(\cE^{\sigma})$ is  defined by
\begin{equation}
\label{eq:Index.Dnabla}
    D_{\nabla^{\cE}}(\xi\otimes \zeta):=\sigma^{\cE}(\xi)\otimes D\zeta + c(\nabla^{\cE})(\xi\otimes \zeta) \qquad 
    \text{for all $\xi\in \cE$ and $\zeta \in \dom(D)$}.
\end{equation}   
\end{definition}


\begin{remark} 
As the operators $\sigma^{\cE}$, $D$ and $\nabla^{\cE}$ are not module maps, we need to check the compatibility 
of~(\ref{eq:Index.Dnabla}) with the action of $\cA$. This is a consequence of~(\ref{eq:TwistedDifferential}). Indeed, if $\xi \in \cE$ and $\zeta \in 
\dom(D)$, then, for all $a \in \cA$,
\begin{align*}
c\left(\nabla^{\cE}\right)(\xi a\otimes \zeta)= (1\otimes c)\left( \nabla^{\cE}(\xi a)\otimes \zeta\right) 
& = (1\otimes c)\left( \left(\nabla^{\cE}\xi\right) a\otimes \zeta  +\sigma^{\cE}(\xi)\otimes d_{\sigma}(a)\otimes \zeta\right) \\ 
& = c\left( \nabla^{\cE}\right)(\xi \otimes a \zeta)+ \sigma^{\cE}(\xi)\otimes d_{\sigma}(a)\zeta.
\end{align*}Thus,
\begin{align*}
  D_{\nabla^{\cE}}(\xi a\otimes \zeta)-    D_{\nabla^{\cE}}(\xi\otimes a\zeta) & = \sigma^{\cE}(\xi a)\otimes D\zeta + 
  \sigma^{\cE}(\xi)\otimes d_{\sigma}(a)\zeta - \sigma^{\cE}(\xi)\otimes D(a \zeta)\\ & = \sigma^{\cE}(\xi)\sigma(a) 
  \otimes D\zeta - \sigma^{\cE}(\xi)
  \otimes \sigma(a) D\zeta, 
  \end{align*}which shows that $  D_{\nabla^{\cE}}(\xi a\otimes \zeta)=   D_{\nabla^{\cE}}(\xi\otimes a\zeta)$ in 
  $\cE^{\sigma}\otimes_{\cA}\cH$. 
\end{remark}

\begin{remark}
    With respect to the $\Z_2$-gradings~(\ref{eq:Z_2GradingH(E)}) for $\cH(\cE)$ and $\cH(\cE^{\sigma})$ the operator $D_{\nabla^{\cE}}$ takes the form, 
\begin{equation*}
 D_{\nabla^{\cE}}=   \begin{pmatrix}
    0    & D_{\nabla^{\cE}}^{-} \\
        D_{\nabla^{\cE}}^{+} & 0
    \end{pmatrix}, \qquad D_{\nabla^{\cE}}^{\pm}:\cE\otimes_{\cA}\dom (D^{\pm})\longrightarrow \cH^{\mp}(\cE^{\sigma}).
\end{equation*}That is, $D_{\nabla^\cE}$ is an odd operator.
\end{remark}

Suppose that $\cE=e\cA^{q}$ with $e=e^{2}\in M_{q}(\cA)$. Then there is a canonical isomorphism $U_{e}$ from 
$\cH(\cE)$ to $e\cH^{q}$ given by
\begin{equation*}
    U_{e}(\xi\otimes \zeta)=\left(\xi_{j}\zeta\right)_{1\leq j \leq q} \qquad \text{for all $\xi=(\xi_{j})\in \cE$ and $\zeta \in \cH$},
\end{equation*}where $\cE=e\cA^{q}$ is regarded as a submodule of $\cA^{q}$. The inverse of $U_e$ is given by 
\begin{equation*}
    U^{-1}_{e}\left( (\zeta_{j})\right) = \sum e_{j}\otimes \zeta_{j} \qquad \text{for all $(\zeta_{j})\in e\cH^{q}$},
\end{equation*}where $e\cH^{q}$ is regarded as a subspace of $\cH^{q}$ and $e_{1},\ldots,e_{q}$ are the column vectors 
of $e$. We also note that $U_{e}$ is a graded isomorphism. 

\begin{lemma}
\label{lm:EqualityOfD}
    Suppose that $\cE=e\cA^{q}$ as above and let $\nabla_{0}^{\cE}$ be the Grassmannian $\sigma$-connection of $\cE$. Then     
    \begin{equation*}
        U_{\sigma(e)}D_{\nabla_{0}^{\cE}}U_{e}^{-1}=D_{e,\sigma},
    \end{equation*}where $D_{e,\sigma}$ is defined in~(\ref{eq:Index.Desigma}). 
\end{lemma}
\begin{proof}
The image of $\cE\otimes_{\cA}\dom(D)$ under $U_{e}$ is $e(\dom(D))^q=\dom (D_{e,\sigma})$.  Let $\zeta \in \dom(D)$ 
and let $\xi=(\xi_{j})$ be in $\cE\subset \cA^{q}$. Then 
\begin{equation*}
   c\left(\nabla_{0}^{\cE}\right)(\xi\otimes \zeta)= \sum \sigma(e_{j})\otimes 
   \left(d_{\sigma}\xi_{j}\right)\zeta = 
  \sum \sigma(e_{j})\otimes D(\xi_{j}\zeta) - \sum \sigma(e_{j})\otimes \sigma(\xi_{j})D\zeta .
\end{equation*}
The fact that $\xi\in \cE$ means that $\xi=e\xi=\sum e_{j}\xi_{j}$. Thus, 
\begin{equation*}
  \sum \sigma(e_{j})\otimes \sigma(\xi_{j})D\zeta = \sum \sigma(e_{j}\xi_{j})\otimes D\zeta =\sigma^{\cE}(\xi)\otimes 
  D\zeta.
\end{equation*}
Therefore,
\begin{equation}
\label{eq:D_nabla_0^E}
    D_{\nabla_{0}^{\cE}}(\xi\otimes\eta)=\sigma^{\cE}(\xi)\otimes D\eta+ 
    c\left(\nabla_{0}^{\cE}\right)(\xi\otimes \zeta)=\sum \sigma(e_{j})\otimes D(\xi_{j}\zeta). 
\end{equation}

For $j=1,\ldots,q$ set $\zeta_{j}=\xi_{j}\zeta$, so that $U_{e}(\xi\otimes \zeta)=(\xi_{j}\zeta)_{1\leq j \leq 
q}=(\zeta_{j})_{1\leq j \leq q}$. From~(\ref{eq:D_nabla_0^E}) we get
\begin{equation*}
    U_{\sigma(e)}\left(D_{\nabla_{0}^{\cE}}(\xi\otimes\eta)\right)_{i}= \sum_{j} \sigma(e_{ij})D(\zeta_{j}). 
\end{equation*}
In view of the definition of the operator $D_{e,\sigma}$ in~(\ref{eq:Index.Desigma}) we see that
\begin{equation*}
   U_{\sigma(e)}D_{\nabla_{0}^{\cE}}(\xi\otimes\eta) = \sigma(e)\left( D\zeta_{j}\right)_{1\leq j \leq q}=D_{e,\sigma}\left( 
   (\zeta_{j})_{1\leq j \leq q}\right) =D_{e,\sigma}U_{e}(\xi\otimes \zeta).
\end{equation*}
This proves the lemma. 
\end{proof}
\begin{remark}
    As $U_{e}$ and $U_{\sigma(e)}$ are graded isomorphisms, we further have 
    \begin{equation}
    \label{eq:DesigmaDNabla_0EgradedUE}
        U_{\sigma(e)}^{\mp}D_{\nabla_{0}^{\cE}}^{\pm}\left(U_{e}^{\pm}\right)^{-1}=D_{e,\sigma}^{\pm}.
    \end{equation}
\end{remark}

\begin{lemma}\label{lem:sConnections.FredhlomnessDes}
  For any $\sigma$-connection $\nabla^{\cE}$ on $\cE$,  the operator $D_{\nabla^{\cE}}$ is closed and Fredholm. 
\end{lemma}
\begin{proof}
Let us first assume that $\cE=e\cA^{q}$ and $\cE^{\sigma}=\sigma(e)\cA^{q}$ with $e=e^{2}\in M_{q}(\cA)$. As shown in 
Section~\ref{sec:IndexTwistedSpectralTriple},  the operator $D_{e,\sigma}$ is Fredholm, 
so it follows from Lemma~\ref{lm:EqualityOfD} that 
$D_{\nabla_{0}^{\cE}}$ is a Fredholm operator as well. By Lemma~\ref{lem:DifferenceSigmaConnection} the difference 
$\nabla^{\cE}-\nabla^{\cE}_{0}$ lies in  $\Hom_{\cA}\left(\cE,\cE^{(\sigma)}\otimes 
\Omega_{D,\sigma}^{1}(\cA)\right)$. This implies that $D_{\nabla^{\cE}}-D_{\nabla^{\cE}_{0}}$ is a bounded odd 
operator. As $D_{\nabla^{\cE}_{0}}$ is closed and Fredholm by Lemma~\ref{lem:index.Desigma-Fredholm}, we then deduce that so is $D_{\nabla^{\cE}}$. 
    
In general, thanks to Remark~\ref{rmk:sigma-translate} we always can find an idempotent $e\in M_{q}(\cA)$, $q\geq 1$, and right module isomorphisms 
$\phi:\cE\rightarrow e\cA^{q}$ and $\phi^{\sigma}:\cE^{\sigma}\rightarrow 
    \sigma(e)\cA^{q}$ satisfying~(\ref{eq:sigma^E(xi a)}). 
    We equip $e\cA^{q}$ with the 
    Hermitian metric $\acoup{\cdot}{\cdot}_{\phi}:=\acoup{\phi^{-1}(\cdot)}{\phi^{-1}(\cdot)}_{\cE}$, 
   where $\acoup{\cdot}{\cdot}_{\cE}$  is the Hermitian 
    metric of $\cE$. Likewise, we equip $\sigma(e)\cA^{q}$ with the 
    Hermitian metric  
    $\acoup{\cdot}{\cdot}_{\phi^{\sigma}}:=\acoup{\left(\phi^{\sigma}\right)^{-1}(\cdot)}{\left(\phi^{\sigma}\right)^{-1}(\cdot)}_{\cE^{\sigma}}$, 
   where $\acoup{\cdot}{\cdot}_{\cE^{\sigma}}$ is the Hermitian 
    metric of $\cE^{\sigma}$. 
    We then have unitary operators $U_{\phi}:\cH(\cE)\rightarrow 
    \cH(e\cA^{q})$ and $U_{\phi^{\sigma}}:\cH(\cE^{\sigma})\rightarrow \sigma(e)\cA^{q}$ given by
    \begin{equation}
        U_{\phi}:=\phi\otimes 1_{\cH} \qquad \text{and}\qquad U_{\phi^{\sigma}}:=\phi^{\sigma}\otimes 1_{\cH}.
        \label{eq:sConnections.Uphi}
    \end{equation}
    In addition, we denote by $\nabla^{\phi_{*}\cE}$ the $\sigma$-connection on $e\cA^{q}$ defined by
    \begin{equation}
        \nabla^{\phi_{*}\cE}:=\left(\phi^{\sigma}\otimes 1_{\Omega_{D,\sigma}^{1}(\cA)}\right)\circ \nabla^{\cE}\circ 
        \phi^{-1}.
        \label{lem:sConnections.pushforward-sconnection}
    \end{equation}
    
 Let $\xi\in \cE$ and $\zeta \in \dom(D)$. Set $\eta=\phi(\xi)$ and $\nabla^{\cE}\xi=  \sum 
 \xi_{\alpha}\otimes \omega_{\alpha}$ with $\xi_{\alpha}\in \cE^{\sigma}$ and $\omega_{\alpha}\in 
 \Omega_{D,\sigma}^{1}(\cA)$. Then $\nabla^{\phi_{*}\cE}\eta=\left(\phi^{\sigma}\otimes 1_{\Omega_{D,\sigma}^{1}(\cA)}\right)( \nabla^{\cE}\xi)= \sum 
     \phi^{\sigma}(\xi_{\alpha})\otimes \omega_{\alpha}$. Thus, 
     \begin{equation}
     \label{eq:D_nabla^phi_*cEP1}
         c\left(\nabla^{\phi_{*}\cE}\right)(\eta \otimes \zeta)= \sum \phi^{\sigma}(\xi_{\alpha})\otimes 
         \omega_{\alpha}(\zeta)=U_{\phi^{\sigma}} \circ c\left(\nabla^{\cE}\right)(\xi\otimes \zeta).
     \end{equation}
 We also note that $\sigma(\eta)=\sigma\circ \phi(\xi)=\phi^{\sigma}\circ \sigma^{\cE}(\xi)$, and hence
 \begin{equation}
 \label{eq:D_nabla^phi_*cEP2}
     \sigma(\eta)\otimes D\zeta= (\phi^{\sigma}\otimes 1_{\cH})\left(\sigma^{\cE}(\xi)\otimes 
     D\zeta\right)=U_{\phi^{\sigma}}\left(\sigma^{\cE}(\xi)\otimes 
     D\zeta\right).
 \end{equation}
Combining~(\ref{eq:D_nabla^phi_*cEP1}) and~(\ref{eq:D_nabla^phi_*cEP2}) we get
\begin{equation*}
    D_{\nabla^{\phi_{*}\cE}}U_{\phi}(\xi\otimes \zeta)=D_{\nabla^{\phi_{*}\cE}}(\eta\otimes \zeta)=\sigma(\eta)\otimes 
    D\zeta+  c\left(\nabla^{\phi_{*}\cE}\right)(\eta \otimes \zeta)=U_{\phi^{\sigma}}D_{\nabla^{\cE}}(\xi\otimes \zeta).
\end{equation*}
This shows that
\begin{equation}
\label{eq:D_nabla^EeA^q}
    D_{\nabla^{\cE}}=  U_{\phi^{\sigma}}^{-1}D_{\nabla^{\phi_{*}\cE}}U_{\phi}.
\end{equation}
By the first part of the proof we know that the operator $D_{\nabla^{\phi_{*}\cE}}$ is closed and Fredholm. 
As $U_{\phi}$ and $U_{\phi^{\sigma}}$ are isomorphisms we then deduce that  the operator $D_{\nabla^{\cE}}$ is closed 
and Fredholm. The proof is complete. 
\end{proof}

In analogy with~(\ref{eq:Dirac.average-index}) and~(\ref{eq:Index-index-Desigma}) we make the following definition. 

\begin{definition}
  Given any $\sigma$-connection $\nabla^{\cE}$ on $\cE$, the index of $D_{\nabla^{\cE}}$ is defined by
  \begin{equation}
  \label{eq:index.D.coupled.sigma.connection}
    \ind D_{\nabla^{\cE}}:=\frac{1}{2}\left( \ind D_{\nabla^{\cE}}^{+}-\ind D_{\nabla^{\cE}}^{-}\right).  
  \end{equation}
\end{definition}

We are now in a position to prove the main result of this section. 

\begin{theorem}\label{thm.IndexTwisted-connection}
Let $\cE$ be a finitely generated projective right $\cA$-module. Then, 
for any $\sigma$-connection $\nabla^{\cE}$ on $\cE$, we have
    \begin{equation}
        \ind_{D,\sigma}[\cE]= \ind D_{\nabla^{\cE}}.
        \label{eq:s-connections.index-formula}
    \end{equation}
\end{theorem}
\begin{proof}
Let us first assume that $\cE=e\cA^{q}$ with $e=e^{2}\in M_{q}(\cA)$ and let $\nabla_{0}^{\cE}$ be the 
$\sigma$-Grassmannian connection. 
As shown in the proof of Lemma~\ref{lem:sConnections.FredhlomnessDes}, the Fredholm operators $D_{\nabla^{\cE}}$ and $D_{\nabla^{\cE}_{0}}$ differ by a 
bounded odd operator, and so $\ind D_{\nabla^{\cE}}^{\pm}=\ind D_{\nabla^{\cE}_{0}}^{\pm}$. Moreover, it follows 
from~(\ref{eq:DesigmaDNabla_0EgradedUE}) that $\ind D_{\nabla^{\cE}_{0}}^{\pm}= \ind  D_{e,\sigma}^{\pm}$. Thus,
\begin{equation}
   \ind D_{\nabla^{\cE}}^{\pm}=\ind D_{\nabla^{\cE}_{0}}^{\pm}= \ind  D_{e,\sigma}^{\pm}=\ind_{D,\sigma}[e]. 
   \label{eq:sConnections.index-formula-eAq}
\end{equation}
    
In general, we can find an idempotent $e\in M_{q}(\cA)$ and right module isomorphisms $\phi:\cE\rightarrow e\cA^{q}$ and $\phi^{\sigma}:\cE^{\sigma}\rightarrow 
    \sigma(e)\cA^{q}$ satisfying~(\ref{eq:sigma^E(xi a)}). We equip $e\cA^{q}$ and $\sigma(e)\cA^{q}$ with Hermitian 
    metrics as in the proof of Lemma~\ref{lem:sConnections.FredhlomnessDes}, and let $\nabla^{\phi_{*}\cE}$ be the 
    $\sigma$-connection on $e\cA^{q}$ given by~(\ref{lem:sConnections.pushforward-sconnection}). Then from~(\ref{eq:D_nabla^EeA^q}) we have
    \begin{equation*}
    D_{\nabla^{\cE}}=  U_{\phi^{\sigma}}^{-1}D_{\nabla^{\phi_{*}\cE}}U_{\phi},
\end{equation*}
where $U_{\phi}:\cH(\cE)\rightarrow \cH(e\cA^{q})$ and $U_{\phi^{\sigma}}:\cH(\cE^{\sigma})\rightarrow 
\cH(\sigma(e)\cA^{q})$ are the unitary operators given by~(\ref{eq:sConnections.Uphi}). As 
$U_{\phi}$ and $U_{\phi^{\sigma}}$ are even isomorphisms we see that $\ind D_{\nabla^{\cE}}^{\pm}=\ind D_{\nabla^{\phi_{*}\cE}}^{\pm}$. 
Combining this with~(\ref{eq:sConnections.index-formula-eAq}) we then get
\begin{equation*}
   \ind D_{\nabla^{\cE}}=\ind D_{\nabla^{\phi_{*}\cE}}=\ind_{D,\sigma}[e]=    \ind_{D,\sigma}[\cE].  
\end{equation*}
The proof is complete. 
\end{proof}

We conclude this section by looking at the index formula~(\ref{eq:s-connections.index-formula}) in the example of a Dirac spectral triple $(C^{\infty}(M), L^{2}(M,\sS),\sD_{g})$, 
where $(M^{n},g)$ is a compact Riemannian spin oriented manifold of even dimension. Let $E$ be a vector 
bundle over $M$ and $\nabla^{E}$ a connection on $E$. Then $\cE:=C^{\infty}(M,E)$ is a finitely generated 
projective module over the algebra $\cA:=C^{\infty}(M)$. We observe that, $\cA$ is a commutative algebra we can identify left and right modules over $\cA$. 
It would be more convenient to work with left modules instead of right modules as we have been doing so far.  This provides us with a natural identification of 
$\cA$-modules $\cE_{1}\otimes_{\cA} \cE_{2}\simeq 
    \cE_{2}\otimes_{\cA} \cE_{1}$ for the tensor products of two modules $\cE_{1}$ and $\cE_{2}$; the isomorphism is given by the flip map 
    $\xi_1\otimes \xi_2\rightarrow \xi_2\otimes \xi_1$. 
   
   Let $c:\Lambda^{*}_{\C}T^{*}M\rightarrow \End \sS$ be the Clifford representation. Then, for all $a$ and $b$ in 
   $\cA$, 
   \begin{equation}
       a[\sD_{g},b]=a c(db)=c(adb).
       \label{eq:VW.commutator-c-Dirac}
   \end{equation}Therefore, we see that
    \begin{equation*}
        \Omega_{\sD_{g}}^{1}(\cA)=\op{Span}\left\{ c(\omega); \ \omega\in C^{\infty}(M,T^{*}_{\C}M)\right\}.
    \end{equation*} 
    Note that $\nabla^{E}$ is a linear map from $\cE=C^{\infty}(M,E)$ to $C^{\infty}(M, T^{*}M\otimes 
    E)=C^{\infty}(M,T^{*}_{\C}M)\otimes_{\cA}\cE$. Consider the linear map $\nabla^{\cE}$ from $\cE$ to 
    $\Omega_{\sD_{g}}^{1}(\cA)\otimes_{\cA}\cE\simeq \cE\otimes_{\cA}\Omega_{\sD_{g}}^{1}(\cA)$ defined by
    \begin{equation*}
        \nabla^{\cE}:=(c\otimes 1_{\cE})\circ \nabla^{E}.
    \end{equation*}
   Let $\xi\in \cE$ and $a \in \cA$. Using~(\ref{eq:VW.commutator-c-Dirac}) we get
   \begin{equation*}
       \nabla^{\cE}(a\xi)=(c\otimes 1_{\cE})\left(da\otimes 
       \xi+a\nabla^{E}\xi\right)=c(da)+a\nabla^{\cE}\xi=[\sD_{g},a]\xi+a\nabla^{\cE}\xi. 
   \end{equation*}Therefore, we see that $\nabla^{\cE}$ is a connection on the finitely generated projective module 
   $\cE$.   
   
As $\nabla^{\cE}$ is a connection on $\cE,$ we can form the operator 
$\sD_{\nabla^{\cE}}:=(\sD_{g})_{\nabla^{\cE}}$. Set $\cH=L^{2}_{g}(M,E)$. In what follows we identify $\cH(\cE)=\cE\otimes_{\cA}\cH$ with $\cH\otimes_{\cA}\cE\simeq 
 L^{2}(M,\sS\otimes E)$, so that we can regard $ \sD_{\nabla^{\cE}}$ as an unbounded operator of  $L^{2}(M,\sS\otimes E)$. 
Let $\zeta \in C^{\infty}(M,\sS)$ and $\xi \in \cE$. Let us write $\nabla^{E}\xi=\sum \omega_{\alpha}\otimes \xi_{\alpha}$, where 
   $\omega_{\alpha}\in C^{\infty}(M,T^{*}_{\C}M)$ and $\xi_{\alpha}\in \cE$. For each $\alpha$ let us also write $\omega_{\alpha}=\sum 
   a_{\alpha\beta}db_{\alpha\beta}$, with $a_{\alpha\beta}$ and $b_{\alpha\beta}$ in $C^{\infty}(M)$. Then,
   using~(\ref{eq:CliffordNabla}) and~(\ref{eq:VW.commutator-c-Dirac}), we see that $ \sD_{\nabla^{\cE}}(\zeta\otimes \xi)$ is equal to
   \begin{equation*}
     \sD_{g}\zeta \otimes \xi +\sum_{\alpha,\beta} 
       (a_{\alpha\beta}[\sD_{g},b_{\alpha\beta}])\zeta \otimes \xi_{\alpha}= 
       \sD_{g}\zeta \otimes \xi +\sum_{\alpha} c(\omega_{\alpha})\zeta 
       \otimes \xi_{\alpha}=\sD_{\nabla^{E}}(\zeta\otimes \xi).
   \end{equation*}
   Thus, under the identification $\cH(\cE)\simeq L^{2}(M,\sS\otimes E)$, the operators $\sD_{\nabla^{\cE}}$ and 
   $\sD_{\nabla^{E}}$ agree. Combining this with~(\ref{eq:Dirac.index-map2}) and~(\ref{eq:s-connections.index-formula}) 
   we then deduce that, for $\sigma=\text{id}$, 
   \begin{equation}
       \ind_{\sD_{g},\sigma}[\cE]=\ind D_{\nabla^{\cE}}=\ind D_{\nabla^{E}}=\ind_{\sD_{g}}[E],
       \label{eq:s-connections.Dirac}
   \end{equation}where the second index map is the Fredholm index map~(\ref{eq:index.D.coupled.sigma.connection}). Thus under the Serre-Swan isomorphism $K^{0}(M)\simeq K_{0}(C^{\infty}(M))$ 
   this Fredholm index map  agrees with the index map~(\ref{eq:Dirac.index-map1}).

\section{Cyclic Cohomology and Pairing with $K$-theory}
\label{sec:CyclicCohomChernChar}
In this section, we review the main definitions and properties regarding cyclic cohomology and its pairing with $K$-theory. 
Cyclic cohomology was discovered by Connes~\cite{Co:NCDG} and Tsygan~\cite{Ts:UMN} independently. For more details on cyclic cohomology we refer to~\cite{Co:NCG, Lo:CH}. 
Throughout this section we let $\cA$ be a unital algebra over $\C$. 

 \subsection{Cyclic cohomology} The Hochschild cochain-complex of $\cA$ is defined as follows. The space of $m$-cochains 
$C^{m}(\cA)$, $m\in \N_{0}$, consists of $(m+1)$-linear maps $\varphi:\cA^{m+1}\rightarrow \C$.
The Hochschild coboundary $b:C^{m}(\cA)\rightarrow C^{m+1}(\cA)$, $b^{2}=0$, is given by
\begin{align}
\label{eq:HochCoboundary}
    b\varphi(a^{0},\ldots,a^{m+1})= & \sum_{j=0}^{m}(-1)^{j}\varphi(a^{0},\ldots,a^{j}a^{j+1},\ldots,a^{m+1})\\ &+ 
    (-1)^{m+1}\varphi(a^{m+1}a^{0},\ldots, a^{m}), \qquad a^{j}\in \cA.
\end{align}

A cochain $\varphi\in C^{m}(\cA)$ is called \emph{cyclic} when $T\varphi=\varphi$, 
where the operator $T:C^{m}(\cA)\rightarrow C^{m}(\cA)$ is defined by
\begin{equation}
\label{eq:Operator.T}
    T\varphi(a^{0},\ldots,a^{m})=(-1)^{m}\varphi(a^{m},a^{0},\ldots,a^{m-1}), \qquad a^{j}\in \cA.
\end{equation}
We denote by $C_{\lambda}^{m}(\cA)$ the space of cyclic $m$-cochains. As $b(C^{\bt}_{\lambda}(\cA))\subset 
C^{\bt+1}_{\lambda}(\cA)$, we obtain a subcomplex $(C^{\bt}_{\lambda}(\cA),b)$, the cohomology of which is denoted 
$\HC^{\bt}(\cA)$ and called the cyclic cohomology of $\cA$. 

The operator $B:C^{m}(\cA)\rightarrow C^{m-1}(\cA)$ is given by
\begin{equation}
\label{eq:PeriodicityBoundary}
 B=  AB_{0}(1-T) , \qquad \text{where}\ A=1+T+\cdots + T^{m},
\end{equation}
and the operator $B_{0}:C^{m}(\cA)\rightarrow C^{m-1}(\cA)$ is defined by
\begin{equation}
\label{eq:B_0Boundary}
    B_{0}\varphi(a^{0},\ldots,a^{m-1})=\varphi(1,a^{0},\ldots,a^{m-1}), \qquad a^{j}\in \cA.
\end{equation}
Note that $B$ is annihilated by cyclic cochains. Moreover, it can be checked that $B^{2}=0$ and 
$bB+Bb=0$. Therefore, in the terminology of Kassel~\cite{Ka:JAlg87} and Loday~\cite[\S 2.5.13]{Lo:CH}, we obtain a mixed cochain-complex $(C^{\bt}(\cA), 
b,B)$, which is called the \emph{cyclic mixed cochain-complex} of $\cA$. Associated to this mixed complex is the periodic 
cyclic complex $(C^{[\bt]}(\cA),b+B)$, where 
\begin{equation*}
    C^{[i]}(\cA)=\bigoplus_{k=0}^{\infty} C^{2k+i}(\cA), \qquad i=0,1,
\end{equation*}
and we regard $b$ and $B$ as operators between $C^{[0]}(\cA)$ and $C^{[1]}(\cA)$. The corresponding cohomology is 
called the \emph{periodic cyclic cohomology} of $\cA$ and is denoted by $\HP^{\bt}(\cA)$. Note that a periodic cyclic 
cocycle is a \emph{finitely
supported} sequence $\varphi=(\varphi_{2k+i})$ with $\varphi_{2k+i}\in C^{2k+i}(\cA)$, $k\geq 0$, such that
\begin{equation*}
    b\varphi_{2k+i}+B\varphi_{2k+2+i}=0 \qquad \text{for all $k\geq 0$}. 
\end{equation*}

As the operator ${B}$ is annihilated by cyclic cochains, any cyclic 
$m$-cocycle $\varphi$ is naturally identified with the periodic cyclic cocycle $(0,\ldots,0,\varphi,0,\ldots)\in 
C^{[i]}(\cA)$, where $i$ is the parity of $m$. This gives rise to natural morphisms,
\begin{equation*}
    \HC^{2k+\bt}(\cA) \longrightarrow \HP^{\bt}(\cA), \qquad k\geq 0. 
\end{equation*}
Connes' periodicity operator $S:C^{m}_{\lambda}(\cA)\rightarrow 
C^{m+2}_{\lambda}(\cA)$ is obtained from the cup product with the unique cyclic $2$-cocycle on 
$\C$ taking the value $1$ at $(1,1,1)$ (see~\cite{Co:NCDG, Co:NCG}). 
 Equivalently, 
 \begin{equation*}
  S=\frac{1}{(m+1)(m+2)}\sum_{j=1}^{m+1}(-1)^{j}S_{j},   
 \end{equation*}where the operator $S_{j}:C^{m}_{\lambda}(\cA)\rightarrow C^{m+2}_{\lambda}(\cA)$ is given by
 \begin{align}
     S_{j}\varphi(a^{0},\ldots,a^{m+2}) = & \sum_{0\leq l\leq 
     j-2}(-1)^{l}\varphi(a^{0},\ldots,a^{l}a^{l+1},\ldots,a^{j}a^{j+1},\ldots,a^{m+2})\nonumber \\ & + 
     (-1)^{j+1}\varphi(a^{0},\ldots,a^{j-1}a^{j}a^{j+1},\ldots,a^{m+2}).
     \label{eq:CC.operatorSj}
 \end{align}
 Here the operator $S$ is normalized so that, for any cochain $\varphi\in  C^{m+1}(\cA)$, we have 
 \begin{equation}
     \text{$b\varphi$ is cyclic} \ \Longrightarrow \ \text{$B\varphi$ is a cyclic cocycle and $SB\varphi=-b\varphi$ in 
     $\HC^{m+2}(\cA)$}.
     \label{eq:SB=-b}
 \end{equation}
Incidentally, if $\varphi$ is a cyclic cocycle, then $S\varphi$ is a cyclic cocycle whose class in $\HP^{\bt}(\cA)$ agrees 
 with that defined by $\varphi$ (\emph{cf.}~\cite{Co:NCDG, Co:NCG}). Furthermore, Connes~\cite[Theorem~II.40]{Co:NCDG} proved that
\begin{equation}
    \varinjlim \left(\HC^{2k + \bt}(\cA),S\right) = \HP^{\bt}(\cA), 
     \label{eq:CC.injective-limit-HP}
\end{equation}where the left-hand side is the inductive limit of the direct system $(\HC^{2k+\bt}(\cA),S)$. 

It is sometimes convenient to ``normalize'' the cyclic mixed complex. More precisely, we say that a cochain $\varphi\in C^{m}(\cA)$ is \emph{normalized} when  
\begin{equation}
\label{eq:NormalizedCondition}
    \varphi(a^{0},\ldots,a^{m})=0 \quad \text{whenever $a^{j}=1$ for some $j\geq 1$}.   
\end{equation}
We denote by $C_{0}^{m}(\cA)$ the space of normalized $m$-cochains. As the operators $b$ and $B$ preserve the space 
$C_{0}^{\bt}(\cA)$, we obtain a mixed subcomplex $(C^{\bt}_{0}(\cA), b,B)$ of the cyclic mixed complex. Note that 
$B=B_{0}(1-T)$ on $C^{\bt}_{0}(\cA)$.  We denote by $\op{HP}_{0}^{\bt}(\cA)$ the cohomology of the normalized periodic 
complex $(C_{0}^{[\bt]}(\cA),b+B)$, where  $C_{0}^{[\bt]}(\cA)=\bigoplus_{k\geq 0}C^{2k+\bt}_{0}(\cA)$. Furthermore
(see~\cite[Corollary~2.1.10]{Lo:CH}), the inclusion of $C^{\bt}_{0}(\cA)$ in $C^{\bt}(\cA)$ gives rise to an isomorphism, 
\begin{equation}
\label{eq:HP_0=HP}
   \HP^{\bt}_{0}(\cA)\simeq \HP^{\bt}(\cA). 
\end{equation}

 \begin{remark}
     For $m\in \N_{0}$ let $C_{\lambda,0}^{m}(\cA)$ be the space of normalized cyclic $m$-chains. We get a subcomplex 
     $(C_{\lambda,0}^{\bt}(\cA),b)$ of the cyclic complex. We denote by $\HC^{\bt}_{0}(\cA)$ the cohomology of this 
     complex. We observe that if $\varphi\in C_{\lambda,0}^{m}(\cA)$ then the cyclicity and normalization condition~(\ref{eq:NormalizedCondition}) 
     imply that
     \begin{equation}
     \label{eq:property.normalized.cyclic.cochain}
           \varphi(1,a^{1},\ldots,a^{m})=(-1)^{m}\varphi(a^{1},\ldots,a^{m},1)=0 \qquad \forall a^{j}\in \cA.
      \end{equation}
      In fact (\emph{cf.}~\cite[\S 2.2.13]{Lo:CH}), we have
      \begin{equation}
      \label{eq:HC_0&HC}
          \HC^{\bt}(\cA)\simeq \HC^{\bt}(\C)\oplus \HC^{\bt}_{0}(\cA).
      \end{equation}The cyclic cohomology of $\C$ has dimension 1 in even degree and is zero in odd degree. Thus, 
      given any $0$-cochain $\varphi_{0}\in C_{\lambda}^{0}(\cA)$ such that $\varphi_{0}(1)= 1$, the isomorphism~(\ref{eq:HC_0&HC}) can be 
      rewritten as
      \begin{equation}
      \label{eq:HC&HC_0.even.odd}
        \HC^{2k}(\cA)\simeq \C\left[S^{k}\varphi_{0}\right]\oplus \HC^{2k}_{0}(\cA) \qquad \text{and} \qquad   
        \HC^{2k+1}(\cA)\simeq \HC^{2k+1}_{0}(\cA). 
      \end{equation}
      Observe that $S^{2k}\varphi_{0}$ is cohomologous in $\HP^{0}(\cA)$ to $\varphi_{0}$, which is a normalized 
      cocycle. Therefore, combining~(\ref{eq:CC.injective-limit-HP}) and~(\ref{eq:property.normalized.cyclic.cochain}) gives the isomorphism~(\ref{eq:HP_0=HP}). 
 \end{remark}

 \begin{example}
\label{ex:current.cochain}
    Let $\cA=C^{\infty}(M)$, where $M$ is a closed manifold. For $m=0,1,\ldots,n$, let $\Omega_{m}(M)$ be the space of 
    $m$-dimensional currents. Any current $C\in \Omega_{m}(M)$ defines a cochain $\varphi_{C}\in C^{m}(\cA)$ by
        \begin{equation}
       \label{eq:current.cochain}
       \varphi_{C}(f^{0},\ldots,f^{m}) = \frac{1}{m!}\acou{C}{f^{0}df^{1}\wedge \cdots \wedge df^{m}}, \qquad 
       f_{j}\in C^{\infty}(M). 
    \end{equation}
    Note that $\varphi_{C}$ is a normalized cochain. Moreover it can be checked that  $b \varphi_{C}=0$ and $B\varphi_{C}=\varphi_{d^{t}C}$, where $d^{t}$ is the de Rham 
    boundary for currents. Therefore, we obtain a morphism from the mixed complex $(\Omega_{\bt}(M), 0, d^{t})$ to the 
    cyclic mixed complex of $\cA=C^{\infty}(M)$. In particular, we have a natural linear map, 
    \begin{equation}
    \label{eq:beta}
       \alpha^{M}: H_{[\bt]}(M,\C)\longrightarrow \HP^{\bt}(C^{\infty}(M)), \qquad H_{[i]}(M,\C):=\bigoplus_{k\geq 
        0}H_{2k+i}(M,\C), \quad i=0,1,
    \end{equation}
    where $H_{2k+i}(M,\C)$ is the de Rham homology of $M$ of degree $2k+i$. 
\end{example}

\subsection{Pairing with $K$-theory} There are two equivalent ways to define the pairing between $\HP^{0}(\cA)$ and $K_{0}(\cA)$ (see~\cite{Co:NCDG, 
GS:OCCTSM}). Given any even 
\emph{normalized} cochain $\varphi=(\varphi_{2k})$  
 and an idempotent $e\in M_{q}(\cA)$, define
\begin{equation}
    \acou{\varphi}{e}:= \tr\# \varphi_{0}(e)+\sum_{k\geq 1}(-1)^{k}\frac{(2k)!}{k!}\tr\# 
    \varphi_{2k}\biggl(e-\frac{1}{2}, e,\ldots,e\biggr),
    \label{eq:cyclic.pairing-KT}
\end{equation}where $ \tr\# \varphi_{2k}$ is the $2k$-cochain on $M_{q}(\cA)=M_{q}(\C)\otimes \cA$ given by
\begin{equation*}
    \tr\# \varphi_{2k}(\mu^{0}\otimes a^{0},\ldots,\mu^{2k}\otimes a^{2k})= \tr\left[ \mu^{0}\cdots 
     \mu^{k}\right]\varphi_{2k}(a^{0},\ldots,a^{2k}),  \qquad  \mu^{j}\in M_{q}(\C),\ a^{j}\in \cA.
\end{equation*}It can be checked that when $\varphi$ is an even normalized periodic cocycle, the value of 
$\acou{\varphi}{e}$ depends only on the class of $\varphi$ in $\HP^{0}_{0}(\cA)$ and the class of $e$ in $K_{0}(\cA)$. 
Combining this with~(\ref{eq:HP_0=HP}) we then obtain a bilinear pairing,
\begin{equation}
    \acou{\cdot}{\cdot}:\HP^{0}(\cA)\times K_{0}(\cA)\rightarrow\C.
    \label{eq:cyclic.pairing-HP0-K0}
\end{equation}

In addition, given any cyclic $2k$-cocycle $\varphi$ it can be shown (see Remark~\ref{rem:pairing-cyclic-cocycles} below) that
\begin{equation}
\label{eq:pairing.cyclic.cocycle.Chern.char}
    \acou{\varphi}{e}=(-1)^{k}\frac{(2k)!}{k!}\tr \#\varphi(e,...,e)
    \qquad \forall e\in M_{q}(\cA), \ e^{2}=e. 
\end{equation}In fact, the right-hand side of~(\ref{eq:pairing.cyclic.cocycle.Chern.char}) depends only the class of $\varphi$ in 
$\HC^{2q}(\cA)$ and is invariant under the periodicity operators $S$ (see~\cite{Co:NCDG, Co:NCG}). 
Therefore, under the inductive limit~(\ref{eq:CC.injective-limit-HP}) this gives rise to a pairing between $\HP^{0}(\cA)$ and 
$K_{0}(\cA)$ which agrees with the pairing~(\ref{eq:cyclic.pairing-KT}).

\begin{example}
 Suppose now that $\cA=C^{\infty}(M)$, where $M$ is a closed manifold, and let $e\in M_{q}(C^{\infty}(M))$, where $e^{2}=e$. 
Consider the vector bundle $E=\op{ran}e$, which we regard as a subbundle of the trivial vector bundle $E_{0}=M\times 
\C^{q}$. We note that by Serre-Swan theorem any vector bundle over $M$ is isomorphic to a vector bundle of this form. 
We equip $E$ with the Grassmannian connection $\nabla^{E}$ defined by $e$, 
so that 
\begin{equation*}
    \nabla_{X}^{E}\xi=e(X\xi_{j}) \qquad \text{for all $X\in C^{\infty}(M,TM)$ and $\xi=(\xi_{j})\in \cE^q=C^{\infty}(M, E)^q$}.
\end{equation*}
The curvature of $\nabla^{E}$ is $F^{E}=e(de)^{2}=e(de)^{2}e$, and so its Chern form is given by
\begin{equation*}
    \Ch(F^{E})=\sum (-1)^{k}\frac{1}{k!}\tr\left[ e(de)^{2k}\right]\in \Omega^{[0]}(M).
\end{equation*}
Let $C=(C_{2k})$ be a closed even de Rham current and denote by $\varphi_{C}$ the associated cocycle defined 
by~(\ref{eq:current.cochain}). Noting that all the cochains $\varphi_{C_{2k}}$ satisfy~(\ref{eq:property.normalized.cyclic.cochain}), we obtain
\begin{equation}
    \acou{C}{\Ch(F^{E})}=\sum (-1)^{k}\frac{1}{k!}\acou{C_{2k}}{\tr\left[ e(de)^{2k}\right]}=\sum  
    (-1)^{k}\frac{(2k)!}{k!}\tr\#\varphi_{C_{2k}}(e,\ldots,e)= \acou{\varphi_{C}}{e}.
    \label{eq:CC.manifold-Chern-character}
\end{equation}Therefore, the pairing~(\ref{eq:pairing.cyclic.cocycle.Chern.char}) between even periodic cyclic cohomology and $K$-theory reduces to the 
classical pairing~(\ref{eq:pairing.deRhamHomology.KTheory}) between de Rham homology and $K$-theory. 
\end{example}

 \begin{remark}\label{rem:pairing-cyclic-cocycles}
     The equality~(\ref{eq:pairing.cyclic.cocycle.Chern.char}) is proved as follows. Thanks to~(\ref{eq:HC&HC_0.even.odd}) we know that
  \begin{equation}
      \varphi=S^{k}\varphi_{0}+\psi \qquad \bmod b\left( C_{\lambda}^{2k-1}(\cA)\right),
      \label{eq:decomposition-phi}
  \end{equation}where $\varphi_{0}\in C_{\lambda}^{0}(\cA)$ is such that $\varphi_{0}(1)\neq 0$ and $\psi$ is a 
  \emph{normalized} cyclic $2k$-cocycle. As $\varphi_{0}$ and $S^{k}\varphi_{0}$ are cohomologous in $\HP^{0}(\cA)$, 
  we see that $\varphi$ is cohomologous in $\HP^{0}(\cA)$ to the \emph{normalized} even cocycle 
  $\varphi_{0}+\psi$. Thus, 
  \begin{equation}
  \label{eq:pairing.cyclic.cocycle.Chern.char.calculation}
      \acou{\varphi}{e}=\acou{\varphi_{0}}{e}+\acou{\psi}{e}.
  \end{equation}
 Set $c_{k}=(-1)^{k}(k!)^{-1}(2k)!$. Using~(\ref{eq:CC.operatorSj}) it can be checked that
 \begin{equation}
     \acou{\varphi_{0}}{e}=\tr \#\varphi_{0}(e)= c_{k}\tr \#(S^{k}\varphi_{0})(e,\ldots,e).   
     \label{eq:pairing-phi0-e}
 \end{equation}
   Moreover, as the cocycle $\psi$ is normalized and cyclic, it follows 
   from~(\ref{eq:cyclic.pairing-KT}) and~(\ref{eq:property.normalized.cyclic.cochain}) that 
     \begin{equation*}
      \acou{\psi}{e}=c_{k}\tr \#\psi\biggl(e-\frac{1}{2}, e,\ldots,e\biggr)=c_{k}\tr 
      \#\psi(e,e,\ldots,e). 
  \end{equation*}
  Combining this with~(\ref{eq:decomposition-phi})--(\ref{eq:pairing-phi0-e}) then gives
  \begin{equation*}
      \acou{\varphi}{e}=c_{k}\tr \#(S^{k}\varphi_{0}+\psi)(e,\ldots,e)=c_{k}\tr\#\varphi(e,\ldots,e),
  \end{equation*}
  which proves~(\ref{eq:pairing.cyclic.cocycle.Chern.char}). 
 \end{remark}

\section{Connes-Chern Character and Index Formula}
\label{sec:Connes-Chern}

In this section, we give a direct construction of the Connes-Chern character of a twisted spectral triple. Combining 
with the results of Section~\ref{sec:IndexMapSigmaConnections} we shall obtain a reformulation of the Atiyah-Singer index formula~(\ref{eq:Dirac.Atiyah-Singer}) for 
twisted spectral triples. 

Throughout this section we let  $(\cA, \cH, D)_{\sigma}$ be a twisted spectral triple. For $p\geq 1$ we denote by $\cL^{p}(\cH)$ the Schatten ideal of 
operators $T\in \cL(\cH)$ such that $\Tr |T|^{p}<\infty$. We recall that $\cL^{p}(\cH)$ is a Banach ideal with respect to 
the $p$-norm, 
\begin{equation*}
    \|T\|_{p}=\left( \Tr|T|^{p}\right)^{\frac{1}{p}}, \qquad T\in \cL^{p}(\cH). 
\end{equation*}
In what follows we assume that the twisted spectral triple $(\cA, \cH, D)_{\sigma}$ is $p$-\emph{summable}, i.e.,
\begin{equation*}
    D^{-1}\in \cL^{p}(\cH).
\end{equation*}

\subsection{Invertible Case} We start by assuming that $D$ is invertible. We shall explain later how to remove this assumption. We recall the following 
result. 

\begin{lemma}[{\cite[Lemma 7.1]{Ho:CPAM}}; see also~{\cite[p.~304]{Co:NCDG}}]
\label{lem:index.SchattenP.trace}
Let $\cH_1$ and $\cH_2$ be Hilbert spaces and $T\in\cL(\cH_1, \cH_2)$ a Fredholm operator. Let $S\in\cL(\cH_2, \cH_1)$ be such that $1-ST\in\cL^p(\cH_1)$ and $1-TS\in\cL^p(\cH_2).$ 
Then 
\begin{equation*}
\ind T=\Tr\left((1-ST)^q\right)-\Tr\left((1-TS)^q\right) \qquad \forall q\ge p.
\end{equation*}
\end{lemma}

The main impetus for our construction of the Connes-Chern character is the following index formula.

\begin{proposition}\label{lem:CCC.index-formula-Des}
Let $e\in M_q(\cA), q\ge 1,$ be an idempotent. Then, for any integer $k\ge\frac12 p,$
\begin{equation}
\label{eq:Index.Supertrace}
\ind D_{e, \sigma}=\frac12\Str\left((D^{-1}[D, e]_{\sigma})^{2k+1}\right).
\end{equation}
\end{proposition}
\begin{proof}
We defined $D_{e,\sigma}$ as an unbounded operator from $e\cH^q$ to $\sigma(e)\cH^q$.
Alternatively, let $\cH_1$ be the Hilbert space given by the vector space $\dom(D)$ equipped with the Hermitian inner product,
\begin{equation*}
\acou{\xi}{\eta}_1:=\acou{\xi}{\eta}+\acou{D\xi}{D\eta}, \qquad \xi, \eta\in\dom(D).
\end{equation*}
We denote by $\|\cdot\|_1$ the norm of $\cH_{1}$. This norm is complete since $D$ is a closed operator.
We then can regard $D$ as an invertible bounded operator from $\cH_1$ to $\cH$. 
Let $a\in\cA$ and $\xi\in\dom(D)$. 
Upon writing $Da\xi=aD\xi+[D,a]\xi$ we see that 
\begin{align*}
\|a\xi\|_1^2=\|a\xi\|^2+\|Da\xi\|^2\le\|a\xi\|^2+2(\|aD\xi\|^2+\|[D, a]\xi\|^2)
\le 2(\|a\|^2+\|[D,a]\|^2)\|\xi\|_1^{2}.
\end{align*}
Thus $a$ induces a bounded operator on $\cH_1$. It then follows that $e$ induces a bounded operator on $\cH_1^q$, 
so that $e\cH_1^q$ is a closed subspace of $\cH_1^q.$ We regard $D_{e,\sigma}$ as a bounded operator from $e\cH_1^q$ to $\sigma(e)\cH^q$. 
Set $Q=eD^{-1}\sigma(e)\in\cL\left(\sigma(e)\cH^q, e\cH_1^q\right)$. Then the product $D_{e, \sigma}Q$ is equal to
\begin{equation}
\sigma(e)DeD^{-1}\sigma(e)= \sigma(e)\left(\sigma(e)D+[D, e]_{\sigma}\right)D^{-1}\sigma(e)
=1+\sigma(e)[D, e]_{\sigma}D^{-1}\sigma(e),
\label{eq:D_esigma.Q}
\end{equation} 
where we have used the fact that $e=1$ on $e\cH^{q}$. Likewise, the operator $QD_{e,\sigma}$ is equal to
\begin{equation}
\label{eq:Q.D_esigma}
eD^{-1}\sigma(e)De=eD^{-1}(De-[D, e]_{\sigma})e
=1-eD^{-1}[D, e]_{\sigma}e.
\end{equation}

We observe that $D^{-1}[D, e]_{\sigma}D^{-1}[D, e]_{\sigma} e$ is equal to 
\begin{equation}
\label{eq:LpLpLp/2.1}
(e-D^{-1}\sigma(e)D)(e-D^{-1}\sigma(e)D)e= e-eD^{-1}\sigma(e)De=   eD^{-1}[D, e]_{\sigma}e.
\end{equation}
Similarly, the operator $\sigma(e)[D, e]_{\sigma}D^{-1}[D, e]_{\sigma}D^{-1}$ is equal to
\begin{equation}
\label{eq:LpLpLp/2.2}
\sigma(e)(DeD^{-1}-\sigma(e))(DeD^{-1}-\sigma(e))
=-\sigma(e)DeD^{-1}\sigma(e)+\sigma(e)
=-\sigma(e)[D, e]_{\sigma}D^{-1}\sigma(e).
\end{equation}
As $D^{-1}\in\cL^p(\cH)$ and $[D, e]_{\sigma}$ is bounded, using~(\ref{eq:LpLpLp/2.1}) we see that $eD^{-1}[D, e]_{\sigma}e$ is in the 
Schatten class $\cL^{\frac{p}{2}}(e\cH^q)$. Note also that $D^{-1}[D, e]_{\sigma}=D^{-1}([D, e]_{\sigma}D^{-1})D.$
As $D$ is an isomorphism from $\cH_1$ onto $\cH$, we then see that $[D, e]_{\sigma}D^{-1}$ induces on $\cH^q$ an 
operator in $\cL^p(\cH^q).$
Combing this with~(\ref{eq:LpLpLp/2.2}) then shows that $\sigma(e)[D, e]_{\sigma}D^{-1}\sigma(e)$ induces on $\sigma(e)\cH^q$ an operator in 
$\cL^{\frac{p}{2}}\left(\sigma(e)\cH^q\right).$

The $\Z_2$-grading $\cH=\cH^+\oplus\cH^-$ induces a $\Z_2$-grading $\cH_1=\cH_1^+\oplus\cH_1^-$ with $\cH_1^{\pm}=\cH_1\cap\cH^{\pm}.$ This gives rise to splittings $e\cH_1^q=e(\cH_1^+)^q\oplus e(\cH_1^-)^q$ and $\sigma(e)\cH^q=\sigma(e)(\cH^+)^q\oplus\sigma(e)(\cH^-)^q.$ With respect to these splittings, the operator $Q=eD\sigma(e)$ takes the form, 
\begin{equation*}
 Q=\begin{pmatrix}0 & Q^- \\ Q^+ & 0 \end{pmatrix} \qquad \text{where}\ Q^{\pm}:=e(D^{\mp})^{-1}\sigma(e). 
\end{equation*}  
Therefore,~(\ref{eq:D_esigma.Q}) and~(\ref{eq:Q.D_esigma}) can be rewritten as 
\begin{equation*}
D_{e,\sigma}^{\pm}Q^{\mp}=1+\sigma(e)[D^{\pm}, e]_{\sigma}(D^{\pm})^{-1}\sigma(e), \qquad 
Q^{\mp}D_{e,\sigma}^{\pm}=1-e[D^{\mp}, e]_{\sigma}(D^{\mp})^{-1}e,
\end{equation*}where the first equality holds in $\cL(\sigma(e)\cH^{\pm})$ and the second holds in 
$\cL(e\cH^{\pm}_{1})$. As shown above the operators $\sigma(e)[D^{\pm}, e]_{\sigma}(D^{\pm})^{-1}\sigma(e)$ and 
$e[D^{\mp}, e]_{\sigma}(D^{\mp})^{-1}e$ are in the Schatten classes $\cL^{\frac{p}{2}}(\sigma(e)(\cH^{\pm})^q)$ and the second holds in 
$\cL^{\frac{p}{2}}(e(\cH^{\pm}_{1})^q)$, respectively. Therefore, we may apply Lemma~\ref{lem:index.SchattenP.trace} to 
obtain that, for all $k\ge\frac12 p$,
\begin{equation*}
\ind D_{e, \sigma}^{\pm}=\Tr\left(e(D^{\mp})^{-1}[D^{\mp}, e]_{\sigma}e\right)^k-\Tr\left(-\sigma(e)[D^{\pm}, e]_{\sigma}(D^{\pm})^{-1}\sigma(e)\right)^k.
\end{equation*}
Thus,
\begin{equation*}
2\ind D_{e, \sigma}=\Str\left(eD^{-1}[D, e]_{\sigma}e\right)^k-\Str\left(-\sigma(e)[D, e]_{\sigma}D^{-1}\sigma(e)\right)^k. 
\end{equation*}
Combining this with~(\ref{eq:LpLpLp/2.1}) and~(\ref{eq:LpLpLp/2.2}) we then get 
\begin{equation*}
2\ind D_{e, \sigma}=\Str\left((D^{-1}[D, e]_{\sigma})^{2k}e\right)+\Str\left(\sigma(e)([D, e]_{\sigma}D^{-1})^{2k}\right).
\end{equation*}
We observe that $\Str(\sigma(e)\left([D, e]_{\sigma}D^{-1})^{2k}\right)$ is equal to
\begin{equation*}
\Str\left(\sigma(e)D(D^{-1}[D, e]_{\sigma})^{2k}D^{-1}\right)=-\Str\left(D^{-1}\sigma(e)D(D^{-1}[D, e]_{\sigma})^{2k}\right).
\end{equation*}
Therefore, we obtain
\begin{align*}
\ind D_{e, \sigma}=\frac{1}{2}\Str\left((e-D^{-1}\sigma(e)D)(D^{-1}[D, e]_{\sigma})^{2k}\right)
=\frac{1}{2}\Str\left((D^{-1}[D, e]_{\sigma})^{2k+1}\right).
\end{align*}
The lemma is thus proved.
\end{proof}

\begin{definition}\label{eq:CC.tauD}
For $k\ge\frac{1}{2}(p-1)$ let $\tau^{D}_{2k}$ be the $2k$-cochain on $\cA$ defined by 
\begin{equation}
\label{eq:tau_2k}
\tau^{D}_{2k}(a^0, \ldots, a^{2k})=c_k\Str\big(D^{-1}[D, a^0]_{\sigma}\cdots D^{-1}[D, a^{2k}]_{\sigma}\big) \qquad \forall a^j\in\cA,
\end{equation}
where we have set $c_k=\frac12(-1)^k\frac{k!}{(2k)!}$ .  
\end{definition}

We note that $\tau^{D}_{2k}$ is a normalized cyclic cochain. Moreover, using~(\ref{eq:pairing.cyclic.cocycle.Chern.char}), for $k\geq \frac{1}{2}p$, we can rewrite the index 
formula~(\ref{eq:Index.Supertrace}) in the form,
\begin{equation}
\label{eq:ind=pairing.CCChar.K}
\ind D_{e,\sigma} = \acou{\tau^{D}_{2k}}{e} \qquad \forall e\in M_{q}(\cA), \ e^{2}=e. 
\end{equation}

For $m\ge p$, we let $\varphi_m$ and $\psi_m$ be the normalized $m$-cochains on $\cA$ defined by 
\begin{gather}
\label{eq:phi_m}
\varphi_m(a^0, \ldots, a^m)=\Str(a^0D^{-1}[D, a^1]_{\sigma}\cdots D^{-1}[D, a^m]_{\sigma}),\\ 
\psi_m(a^0, \ldots, a^m)=\Str(\sigma(a^0)[D, a^1]_{\sigma}D^{-1}\cdots[D, a^m]_{\sigma}D^{-1}), \qquad a^j\in\cA.
\label{eq:psi_m} 
\end{gather}
We observe that $\psi_m(a^0, \ldots, a^m)$ is equal to
\begin{equation}
\label{eq:phi_m&psi_m}
-\Str(D^{-1}\sigma(a^0)DD^{-1}[D, a^1]_{\sigma}\cdots D^{-1}[D, a^m]_{\sigma})
=-\varphi_m(D^{-1}\sigma(a^0)D, a^1, \ldots, a^m).
\end{equation}
Using the equality $D^{-1}[D, a^0]_{\sigma}=a^0-D^{-1}\sigma(a^0)D,$ we then see that, for $k\ge\frac12 p,$
\begin{align}
\label{eq:tau_2kphi_2kpsi_2k}
c_{k}^{-1}\tau^{D}_{2k}(a^0, \ldots, a^{2k}) 
= & \Str\left\{a^0D^{-1}[D, a^1]_{\sigma}\cdots D^{-1}[D, a^{2k}]_{\sigma}\right\} \nonumber\\
& -\Str\left\{D^{-1}\sigma(a^0)DD^{-1}[D, a^1]_{\sigma}\cdots D^{-1}[D, a^{2k}]_{\sigma}\right\} \\
=& \varphi_{2k}(a^0, \ldots, a^{2k})+\psi_{2k}(a^0, \ldots, a^{2k})\nonumber
\end{align}

\begin{lemma}\label{lem:B.phi_2k+1}
Let $k\ge\frac12(p-1).$ Then 
\begin{equation*}
B\varphi_{2k+1}=-B\psi_{2k+1}=(2k+1)c_k^{-1}\tau^{D}_{2k}.
\end{equation*}
\end{lemma}
\begin{proof}
As $\varphi_{2k+1}$ and $\psi_{2k+1}$ are normalized cochains, we know that $B\varphi_{2k+1}= AB_{0}\varphi_{2k+1}$ and $B\psi_{2k+1}= AB_{0}\psi_{2k+1}$. Moreover, it follows from~(\ref{eq:phi_m}) and~(\ref{eq:psi_m}) that \[B_0\varphi_{2k+1}=-B_0\psi_{2k+1}=c_k^{-1}\tau^{D}_{2k}.\] 
 As $\tau^{D}_{2k}$ is a cyclic cochain, we then deduce that 
\[B\varphi_{2k+1}=-B\psi_{2k+1}=c_k^{-1}A\tau^{D}_{2k}=(2k+1)c_k^{-1}\tau^{D}_{2k}.\] This proves the lemma.
\end{proof}

\begin{lemma}
\label{lem:b.phi_2k-1=phi_2k}
Let $k\ge\frac12(p+1).$ Then
\begin{equation}
\label{eq:b.phi_2k-1=phi_2k}
b\varphi_{2k-1}=\varphi_{2k}\qquad\text{and}\qquad b\psi_{2k-1}=-\psi_{2k}.
\end{equation}
\end{lemma}
\begin{proof}
For $j=1,\ldots,2k$ let $\theta_{j}'$ and $\theta_{j}''$ be the $2k$-cochains on $\cA$ defined by
\begin{gather*}
   \theta_{j}'(a^0, \ldots, a^{2k})=\Str\left( a^0 D^{-1}[D, a^1]_{\sigma}\cdots a^{j}\cdots 
   D^{-1}[D, a^{2k}]_{\sigma}\right),\\
   \theta_{j}''(a^0, \ldots, a^{2k})=\Str\left(a^0 D^{-1}[D, a^1]_{\sigma}\cdots D^{-1}\sigma(a^j)D^{-1}\cdots D^{-1}[D, a^{2k}]_{\sigma}\right), \qquad  a^l\in\cA. 
\end{gather*}
We note that 
\begin{align*}
\theta_{j}'(a^0, \ldots, a^{2k})- \theta_{j}''(a^0, \ldots, a^{2k}) & =  \Str\left(a^0 D^{-1}[D, a^1]_{\sigma}\cdots 
(a^{j}-D^{-1}\sigma(a^j)D)\cdots D^{-1}[D, a^{2k}]_{\sigma}\right)  \\ &= \varphi_{2k}(a^{0},\ldots,a^{2k}).
\end{align*}
Using the equality $D^{-1}[D, a^j a^{j+1}]_{\sigma}=D^{-1}[D, a^j]_{\sigma}a^{j+1}+D^{-1}\sigma(a^j)D\cdot D^{-1}[D, a^{j+1}]_{\sigma}$ we also find that
\begin{align*}
    b_{j}\varphi_{2k-1}(a^{0},\ldots,a^{2k}) & = \Str\left( a^0 D^{-1}[D, a^1]_{\sigma}\cdots D^{-1}[D, a^j a^{j+1}]_{\sigma}\cdots 
   D^{-1}[D, a^{2k}]_{\sigma}\right),    \\
     & = \theta_{j+1}'(a^{0},\ldots,a^{2k}) + \theta_{j}''(a^{0},\ldots,a^{2k}).
\end{align*}
Thus $ \sum_{j=1}^{2k-1}(-1)^{j}b_{j}\varphi_{2k-1}$ is equal to
\begin{align}
   \sum_{j=1}^{2k-1}(-1)^{j}(\theta_{j+1}'+\theta_{j}'') & = -\theta_{1}''+ 
    \sum_{j=2}^{2k-1}(-1)^{j-1}(\theta_{j-1}'-\theta_{j}'')-\theta_{2k}'\nonumber \\ 
   & = -\theta_{1}''+ 
    \sum_{j=2}^{2k-1}(-1)^{j-1}\varphi_{2k}-\theta_{2k}'  \label{eq:CCC.bvphi}\\
    &= -\theta_{1}''-\theta_{2k}'. \nonumber
\end{align}
We also note that
\begin{equation*}
    b_{0}\varphi_{2k-1}(a^{0},\ldots,a^{2k})  =\Str\left(a^0 a^{1}D^{-1}[D, a^2]_{\sigma}\cdots 
D^{-1}[D, a^{2k}]_{\sigma}\right)  = \theta_{1}'(a^{0},\ldots,a^{2k}).
\end{equation*}
Moreover, the cochain $b_{2k}\varphi_{2k-1}(a^{0},\ldots,a^{2k})$ is equal to
\begin{align*}
 \Str\left(a^{2k}a^{1}D^{-1}[D, a^1]_{\sigma}\cdots D^{-1}[D, 
a^{2k-1}]_{\sigma}\right)  
   & =\Str\left(a^{0}D^{-1}[D, a^1]_{\sigma}\cdots D^{-1}[D, a^{2k-1}]_{\sigma}a^{2k}\right)  \\ 
  & = \theta_{2k}'(a^{0},\ldots,a^{2k}).   
\end{align*}
Therefore, we find that
\begin{equation*}
    b\varphi_{2k-1}=  \sum_{j=0}^{2k}(-1)^{j}b_{j}\varphi_{2k-1}= 
    b_{0}\varphi_{2k-1}-\theta_{1}''-\theta_{2k}'+b_{2k}\varphi_{2k-1}=\theta_{1}'-\theta_{1}''=\varphi_{2k}.
\end{equation*}

As $\psi_{2k-1}(a^{0},\ldots,a^{2k-1})= -\varphi_{2k-1}(D^{-1}\sigma(a^{0})D,a^{1},\ldots,a^{2k-1})$, using~(\ref{eq:CCC.bvphi}) we get
\begin{equation*}
  \sum_{j=1}^{2k-1}(-1)^{j}b_{j}\psi_{2k-1} =
  \theta_{1}''(D^{-1}\sigma(a^{0})D,a^{1},\ldots,a^{2k})+\theta_{2k}'(D^{-1}\sigma(a^{0})D,a^{1},\ldots,a^{2k}).  
\end{equation*}
We observe that $b_{0}\psi_{2k-1}(a^{0},\ldots,a^{2k})$ is equal to
\begin{equation*}
 -  \Str\left(D^{-1}a^0 a^{1}D^{-1}[D, a^2]_{\sigma}\cdots 
D^{-1}[D, a^{2k}]_{\sigma}\right)  =  -  \theta_{1}''(D^{-1}\sigma(a^{0})D,a^{1},\ldots,a^{2k-1}).
\end{equation*}Moreover, we have
\begin{align*}
b_{2k}\psi_{2k-1}(a^{0},\ldots,a^{2k}) & = - \Str\left(D^{-1}\sigma(a^{2k})\sigma(a^{0})D\cdot D^{-1}[D, 
a^1]_{\sigma}\cdots D^{-1}[D, a^{2k-1}]_{\sigma}\right)\\
& =-\Str\left(D^{-1}a^{0}D\cdot D^{-1}[D, a^1]_{\sigma}\cdots D^{-1}[D, a^{2k-1}]_{\sigma}D^{-1}\sigma(a^{2k})D\right)  \\ 
  & = -\theta_{2k}''(D^{-1}\sigma(a^{0})D,a^{1},\ldots,a^{2k}).   
\end{align*}
Therefore, we see that $b\psi_{2k-1}(a^{0},\ldots,a^{2k})=\sum_{j=0}^{2k}(-1)^{j}b_{j}\psi_{2k-1}(a^{0},\ldots,a^{2k})$ is equal to
\begin{align*}
     \theta_{2k}'(D^{-1}\sigma(a^{0})D,a^{1},\ldots,a^{2k}) -\theta_{2k}''(D^{-1}\sigma(a^{0})D,a^{1},\ldots,a^{2k}) & 
     = \varphi_{2k}(D^{-1}\sigma(a^{0})D,a^{1},\ldots,a^{2k}) \\ 
    & =- \psi_{2k}(a^{0},a^{1},\ldots,a^{2k}) .
\end{align*}
The proof is complete.
\end{proof}

\begin{proposition}[\cite{CM:TGNTQF}]\label{prop:Cochian.ConnesChernChar}
Let $k\geq \frac12(p-1)$. Then 
\begin{enumerate}
\item The cochain $\tau^{D}_{2k}$ in~(\ref{eq:tau_2k}), is a normalized cyclic cocycle. 
\item The class of $\tau^{D}_{2k}$ in $\HP^0(\cA)$ is independent of the value of $k.$
\end{enumerate}
\end{proposition}
\begin{proof}
We already know that $\tau^{D}_{2k}$ is a cyclic normalized 
cochain. We also note that
\begin{equation*}
c_{k+1}=\frac12(-1)^{k+1}\frac{(k+1)!}{(2k+2)!}=-\frac{c_{k}}{2(2k+1)}.
\end{equation*}
Combining this with Lemma~\ref{lem:B.phi_2k+1} we get
\begin{equation}
\label{eq:tau.Bphipsi}
\tau^{D}_{2k}=\frac{c_k}{2(2k+1)}B(\varphi_{2k+1}-\psi_{2k+1})= -c_{k+1}B(\varphi_{2k+1}-\psi_{2k+1}). 
\end{equation}
Using~(\ref{eq:tau.Bphipsi}) and the fact that $bB=-Bb$ we then see that 
\begin{equation*}
b\tau^{D}_{2k}=-c_{k+1}bB(\varphi_{2k+1}-\psi_{2k+1})=c_{k+1}Bb(\varphi_{2k+1}-\psi_{2k+1}).
\end{equation*}
Moreover, using~(\ref{eq:tau_2kphi_2kpsi_2k}) and Lemma~\ref{lem:b.phi_2k-1=phi_2k} we get
\begin{equation}
\label{eq:tau.b.phi.psi}
\tau^{D}_{2k+2}=c_{k+1}(\varphi_{2k+2}+\psi_{2k+2})=c_{k+1}b(\varphi_{2k+1}-\psi_{2k+1}).
\end{equation}As $B$ is annihilated by cyclic cochains we then deduce that
\begin{equation*}
 b\tau^{D}_{2k}= B\left(c_{k+1}b(\varphi_{2k+1}-\psi_{2k+1})\right)=B\tau^{D}_{2k+2}=0.
\end{equation*}That is, $\tau^{D}_{2k}$ is a cocycle. We also see that
\begin{equation*}
    \tau^{D}_{2k+2}-\tau_{2k}^{D}=c_{k+1}b(\varphi_{2k+1}-\psi_{2k+1})+c_{k+1}B(\varphi_{2k+1}-\psi_{2k+1}). 
\end{equation*}
This shows that $\tau^{D}_{2k+2}$ and $\tau^{D}_{2k}$ define the same class in $\HP^0(\cA)$.
It then follows that the class of $\tau^{D}_{2k}$ in $\HP^0(\cA)$ is independent of the value of $k$. The proof is complete.
\end{proof}

\begin{remark}\label{rmk:CCC.unitality}
    The proof of Lemma~\ref{lem:B.phi_2k+1} uses the fact that the unit of $\cA$ is represented by the identity of $\cH$. Otherwise 
    the equalities $B_0\varphi_{2k+1}=-B_0\psi_{2k+1}=c_k^{-1}\tau^{D}_{2k}$ need not hold. Therefore, the unitality of 
    $\cA$ is a crucial ingredient of the proof of Proposition~\ref{prop:Cochian.ConnesChernChar}. 
\end{remark}

\begin{remark}
\label{rmk:tau2k+2.cohomologous.Stau2k}
    As in~\cite{Co:NCDG}, we can get a more precise relationship between the cocycles $\tau_{2k}^{D}$ and $\tau_{2k+2}^{D}$ by using 
    the $S$-operator. Indeed, using~(\ref{eq:SB=-b}) and~(\ref{eq:b.phi_2k-1=phi_2k}) we get
    \begin{equation*}
      SB(\varphi_{2k+1}-\psi_{2k+1})=-b(\varphi_{2k+1}-\psi_{2k+1})=-(\varphi_{2k+2}+\psi_{2k+2}) \quad \text{in 
      $\HC^{2k}(\cA)$}.  
    \end{equation*}Combining this with~(\ref{eq:tau.Bphipsi}) and~(\ref{eq:tau.b.phi.psi}) we then deduce that
    \begin{equation*}
        S\tau_{2k}^{D}=-c_{k+1}SB(\varphi_{2k+1}-\psi_{2k+1})=c_{k+1}(\varphi_{2k+2}+\psi_{2k+2})=\tau_{2k+2}^{D} \quad \text{in 
      $\HC^{2k}(\cA)$}. 
    \end{equation*}
    As $\tau_{2k}^{D}$ and $S\tau_{2k}^{D}$ are cohomologous in $\HP^{0}(\cA)$, this provides us with an 
    alternative argument showing that $\tau_{2k}^{D}$ and $\tau_{2k+2}^{D}$ define the same class in $\HP^{0}(\cA)$.
\end{remark}

\begin{definition}[\cite{CM:TGNTQF}]\label{def:CCC.invertible}
Let $(\cA,\cH,D)_{\sigma}$ be a $p$-summable twisted spectral triple with $D$ invertible. Then its Connes-Chern character, denoted  $\Ch(D)_{\sigma}$, 
is the class in $\HP^{0}(\cA)$ of any of the cocycles $\tau_{2k}^{D}$, $k\geq \frac{1}{2}(p-1)$.
\end{definition}

We are now in a position to reformulate the Atiyah-Singer index formula~(\ref{eq:Dirac.Atiyah-Singer}) for twisted spectral triples in the 
invertible case.  

\begin{theorem}\label{thm:CCC-index-formula.invertible}
Let $(\cA,\cH,D)_{\sigma}$ be a $p$-summable twisted spectral triple with $D$ invertible. Then, for any finitely 
generated projective right $\cA$-module $\cE$ and any $\sigma$-connection on $\cE$, we have 
\begin{equation*}
\ind D_{\nabla^{\cE}}=\acou{\Ch(D)_{\sigma}}{[\cE]}. 
\end{equation*}
\end{theorem}
\begin{proof}
 Thanks to Theorem~\ref{thm.IndexTwisted-connection} we know that  $\ind D_{\nabla^{\cE}}=\ind_{D,\sigma}[\cE]$. Let $e$ be an idempotent in some 
 matrix algebra $M_{q}(\cA)$ such that $\cE\simeq e\cA^{q}$. Then~(\ref{eq:ind=pairing.CCChar.K}) shows that, for $k\geq \frac{1}{2}p$, 
 \begin{equation*}
     \ind_{D,\sigma}[e]=\ind D_{e,\sigma}=\acou{\tau_{2k}^{D}}{e}=\acou{\Ch(D)_{\sigma}}{[e]}.
 \end{equation*}As $\cE$ and $e$ defines the same class in $K_{0}(\cA)$ we then deduce that
 \begin{equation*}
     \ind D_{\nabla^{\cE}}= \ind_{D,\sigma}[e]=\acou{\Ch(D)_{\sigma}}{[\cE]}.
 \end{equation*}The proof is complete. 
\end{proof}

\subsection{General case}
The assumption on the invertibility of $D$ can be removed by passing to the unital invertible 
double as follows. Consider the Hilbert space $\tilde{\cH}=\cH\oplus \cH$, which we equip with the $\Z_{2}$-grading given by
\begin{equation*}
    \tilde{\gamma}=\begin{pmatrix}\gamma & 0 \\ 0 & -\gamma \end{pmatrix},
\end{equation*}where $\gamma$ is the grading operator of $\cH$. On $\tilde{\cH}$ consider the selfadjoint operators,
\begin{equation*}
    \tilde{D}_{0}=\begin{pmatrix} D & 0 \\  0  &  -D \end{pmatrix}, \qquad 
    J=\begin{pmatrix} 0 & 1 \\ 1 &  0 \end{pmatrix}, \qquad 
    \tilde{D}=\tilde{D}_{0}+J=\begin{pmatrix} D & 1 \\ 1 &  -D \end{pmatrix}, 
\end{equation*}where the domain of $\tilde{D}_{0}$ and $\tilde{D}$ is $\dom(D) \oplus \dom(D)$. As 
$\tilde{D}_{0}J+J\tilde{D}_{0}=0$ and $J^{2}=1$ we get
\begin{equation}
    \tilde{D}^{2}=\tilde{D}_{0}^{2}+1=\begin{pmatrix}
        D^{2}+1 & 0 \\
        0 & D^{2}+1
    \end{pmatrix}.\label{eq:CCC.tD2}
\end{equation}It then follows that $\tilde{D}$ is invertible. Moreover, 
\begin{equation}
    \Tr |\tilde{D}|^{-p}=\Tr ( \tilde{D}^{2})^{-\frac{p}{2}}=2 \Tr (D^{2}+1)^{-\frac{p}{2}}\leq 2 
    \Tr |D|^{-p}<\infty.
    \label{eq:CCC.p-summability}
\end{equation}That is, $\tilde{D}^{-1}\in \cL^{p}(\tilde{\cH})$. 

Let $\pi:\cA\rightarrow \cL(\tilde{\cH})$ be the linear map given by
\begin{equation*}
     {\pi}(a)=\begin{pmatrix} a & 0 \\ 0 & 0 \end{pmatrix} \qquad \forall a \in \cA. 
\end{equation*}We note that $\pi$ is multiplicative, but as $\pi(1)\neq 0$ this is not a representation of the unital 
algebra $\cA$. As mentioned in Remark~\ref{rmk:CCC.unitality} the representation of the unit $1$ by the identity of $\cH$ is essential to 
the construction of the Connes-Chern character in the invertible case. To remedy to the lack of unitality of $\pi$  we pass to the $*$-algebra $\tilde{\cA}=\cA\oplus \C$ with 
product and involution given by
\begin{equation*}
    (a,\lambda)(b,\mu)=(ab+\lambda b+\mu a,\lambda\mu), \qquad (a,\lambda)^{*}=(a^{*},\overline{\lambda}), \qquad 
    a,b\in \cA, \quad \lambda,\mu\in \C.
\end{equation*}The unit of $\tilde{A}$ is $1_{\tilde{A}}=(0,1)$. Thus, identifying any element $a\in \cA$ with 
$(a,0)$, any element $\tilde{a}=(a,\lambda)\in \tilde{\cA}$ can be uniquely written as $(a,\lambda)=a+\lambda 1_{\tilde{\cA}}$. 
We extend $\pi$ into the (unital) representation $\tilde{\pi}:\tilde{\cA}\rightarrow \cL(\tilde{\cH})$ given 
by
\begin{equation*}
    \tilde{\pi}(a+\lambda 1_{\tilde{\cA}})=\pi(a)+\lambda \qquad \forall (a,\lambda)\in \cA\times \C.
\end{equation*}We also extend the automorphism $\sigma$ into the 
automorphism $\tilde{\sigma}:\tilde{\cA}\rightarrow \tilde{\cA}$ given by
\begin{equation*}
    \tilde{\sigma}(a+\lambda 1_{\tilde{\cA}})=\sigma(a)+\lambda 1_{\tilde{\cA}} \qquad \forall (a,\lambda)\in \cA\times \C.
\end{equation*}
For $a\in \cA$ and $\lambda \in \C$, the twisted commutator $ [\tilde{D},\tilde{\pi}(a+\lambda 
1_{\tilde{\cA}})]_{\tilde{\sigma}}$ is equal to
\begin{equation*} 
    \begin{pmatrix}D & 1 \\ 1 & -D\end{pmatrix}\begin{pmatrix}a +\lambda & 0 \\ 0 & \lambda \end{pmatrix} 
-\begin{pmatrix}\sigma(a) + \lambda & 0 \\ 0 & \lambda \end{pmatrix}\begin{pmatrix}D & 1 \\ 1 & -D\end{pmatrix}= 
\begin{pmatrix}[D, a]_{\sigma} & -\sigma(a) \\ a & 0\end{pmatrix} \in \cL(\tilde{\cH}).
\end{equation*} Combining all these we obtain the following statement.

\begin{proposition}
    $(\tilde{\cA},\tilde{\cH},\tilde{D})_{\tilde{\sigma}}$ is a $p$-summable twisted spectral triple. 
\end{proposition}

As $\tilde{D}$ is invertible, we can form the normalized cyclic cocycles $\tau_{2k}^{\tilde{D}}$, $k\geq 
\frac{1}{2}(p-1)$, as in Definition~\ref{eq:CC.tauD}.  We note that if $\tilde{\varphi}$ is a cyclic $m$-cochain on 
$\tilde{\cA}$, then its restriction $\overline{\varphi}$ to $\cA^{m+1}$ is a cyclic cochain on $\cA$. Moreover, 
it follows from the formulas~(\ref{eq:HochCoboundary}) 
 and~(\ref{eq:CC.operatorSj}) for the operators $b$ and $S$ that
 \begin{equation}
     \overline{b\varphi}=b\overline{\varphi} \qquad \text{and} \qquad \overline{S\varphi}=S\overline{\varphi}.
     \label{eq:CCC.bS-extensions}
 \end{equation}
If in addition $\varphi$ is normalized, then the normalization condition 
and~(\ref{eq:property.normalized.cyclic.cochain}) imply that
 \begin{equation}
    \tilde\varphi(a^{0}+\lambda^{0} 1_{\tilde{\cA}},\ldots,a^{m}+\lambda^{m} 1_{\tilde{\cA}})= 
    \varphi(a^{0},\ldots,a^{m}) \qquad \forall a^{j}\in \cA\ \forall \lambda^{j}\in \C.
     \label{eq:CCC.cochainsA-cochainstA}
 \end{equation}Thus $\tilde{\varphi}$ is uniquely determined by its restriction $\overline{\varphi}$ to $\cA^{m+1}$.  
 Conversely, any  
 cyclic $m$-cochain $\varphi$ on $\cA$ uniquely  extends to a normalized cyclic $m$-cochain $\tilde{\varphi}$ on 
 $\tilde{\cA}$ satisfying~(\ref{eq:CCC.cochainsA-cochainstA}).
 
\begin{definition}
    Let $k\geq \frac{1}{2}(p-1)$. Then $\overline{\tau}_{2k}^{D}$ is the cyclic $2k$-cochain on $\cA$ given by the 
    restriction of $\tau_{2k}^{\tilde{D}}$ to $\cA^{2k+1}$. 
\end{definition}
 
\begin{proposition}\label{prop:CCC.otautD2k}
     Let $k\geq \frac{1}{2}(p-1)$. Then
     \begin{enumerate}
         \item  The $2k$-cochain $\overline{\tau}_{2k}^{D}$ is a cyclic cocycle whose class in 
         ${\HP}^{0}(\cA)$ is independent of $k$.
     
         \item  For any idempotent $e \in M_{q}(\cA)$, we have
         \begin{equation*}
             \ind D_{\sigma, e}=\acou{\overline{\tau}_{2k}^{D}}{e}.
         \end{equation*}
     \end{enumerate}
 \end{proposition}
 \begin{proof}
     It follows from~(\ref{eq:CCC.bS-extensions}) that 
     $b\overline{\tau}_{2k}^{D}=\overline{b\tau^{\tilde{D}}_{2k}}=0$, so $\overline{\tau}_{2k}^{D}$ is 
     a cyclic cocycle. By Remark~\ref{rmk:tau2k+2.cohomologous.Stau2k} the cocycles $S\tau_{2k}^{\tilde{D}}$ and $\tau_{2k+2}^{\tilde{D}}$ are 
     cohomologous in $\HC^{2k+2}(\tilde{\cA})$. Therefore, by~(\ref{eq:CCC.bS-extensions}) their respective 
     restrictions to $\cA$, namely, $S\overline{\tau}_{2k}^{D}$ and $\overline{\tau}_{2k+2}^{D}$, are cohomologous in 
     $\HC^{2k+2}(\cA)$. As $S\overline{\tau}_{2k}^{D}$ and $\overline{\tau}_{2k}^{D}$ define the same class in 
     $\HP^{0}(\cA)$, we deduce that so do $\overline{\tau}_{2k}^{D}$ and $\overline{\tau}_{2k+2}^{D}$. It then 
     follows that the class of $\overline{\tau}_{2k}^{D}$ in $\HP^{0}(\cA)$ is independent of $k$.      
     
Let $e\in M_{q}(\cA)$, $e^{2}=e$. Regarding $e$ as an idempotent in $M_{q}(\tilde{\cA})$, we have
\begin{equation*}
    \tilde{\pi}(e)\tilde{\cH}^{q}=
    \begin{pmatrix}
        e & 0 \\
        0 & 0
    \end{pmatrix}\left(\cH^{q}\oplus \cH^{q}\right)=e(\cH^{q})\oplus \{0\}\simeq e\cH^{q}.
\end{equation*}Likewise, $\tilde{\pi}(\sigma(e))\tilde{\cH}^{q}=\sigma(e)(\cH^{q})\oplus \{0\}\simeq \sigma(e)\cH^{q}$. 
Moreover, 
\begin{equation*}
\sigma(e)\tilde{D}e =\begin{pmatrix}\sigma(e) & 0 \\ 0 & 0\end{pmatrix}\begin{pmatrix} D & 1 \\ 1 & -D 
\end{pmatrix}\begin{pmatrix}e & 0 \\ 0 & 0\end{pmatrix} =\begin{pmatrix}\sigma(e)De & 0 \\ 0 & 0\end{pmatrix}.
\end{equation*}
Thus, under the identifications $ \tilde{\pi}(e)\tilde{\cH}^{q}\simeq e\cH^{q}$ and 
$\tilde{\pi}(\sigma(e))\tilde{\cH}^{q}\simeq \sigma(e)\cH^{q}$ above, the operators $\tilde{D}_{e,\sigma}$ and 
$D_{e,\sigma}$ agree. Therefore, using~(\ref{eq:ind=pairing.CCChar.K}) we get $\ind D_{e, \sigma}=\ind\tilde{D}_{e, 
\sigma}=\acou{\tau_{2k}^{\tilde{D}}}{e}$. As~(\ref{eq:CCC.cochainsA-cochainstA}) implies that 
$\acou{\tau_{2k}^{\tilde{D}}}{e}=\acou{\overline{\tau}_{2k}^{D}}{e}$, we then deduce that
$\ind D_{e, \sigma}=\acou{\overline{\tau}_{2k}^{D}}{e}$. 
%
The proof is complete. 
\end{proof}
 
\begin{definition}\label{def:CCC.general}
Let  $(\cA,\cH,D)_{\sigma}$ be a $p$-summable twisted spectral triple. Then its Connes-Chern character, denoted  $\Ch(D)_{\sigma}$, 
is the class in $\HP^{0}(\cA)$ of any of the cocycles $\overline{\tau}_{2k}^{D}$, $k\geq \frac{1}{2}(p-1)$.
\end{definition}

Assume now that $D$ is invertible. We then have two~definitions of the Connes-Chern character: one 
in terms of the cocycles $\tau_{2k}^{D}$ and the other in terms of the cocycles $\overline{\tau}_{2k}^{D}$. We shall now show that these definitions are equivalent. 

Consider the homotopy of operators,
\begin{equation*}
    \tilde{D}_{t}=\tilde{D}_{0}+tJ, \qquad 0\leq t \leq 1.
\end{equation*}In the same way as in~(\ref{eq:CCC.tD2}) we have
\begin{equation*}
    \tilde{D}_{t}^{2}=\tilde{D}_{0}^{2}+t^{2}= 
    \begin{pmatrix}
        D^{2}+t^{2} & 0 \\
        0 & D^{2}+t^{2}
    \end{pmatrix},
\end{equation*}which shows that $\tilde{D}_{t}$ is invertible for all $t\in [0,1]$. Moreover, as in~(\ref{eq:CCC.p-summability}) we have
  \begin{equation*}
    \Tr |\tilde{D}_{t}|^{-p}=2 \Tr (D^{2}+t^{2})^{-\frac{p}{2}}\leq 2 
    \Tr |D|^{-p}.
\end{equation*}Thus $(\tilde{D}_{t}^{-1})_{0\leq t \leq 1}$ is a bounded family in $\cL^{p}(\tilde{\cH})$. Therefore, 
the family $(\tilde{D}_{t})_{0\leq t \leq 1}$ satisfies the assumption of Proposition~\ref{prop:homotopy-invariance} in 
Appendix~\ref{app:homotopy-invariance} on the 
homotopy invariance of the Connes-Chern character. We then deduce that
$(\tilde{\cA},\tilde{\cH},\tilde{D}_{t})_{\tilde{\sigma}}$ is a $p$-summable twisted spectral triple for all 
$t\in[0,1]$ and, for $k\geq \frac{1}{2}(p+1)$ the cocycles 
$\tau_{2k}^{\tilde{D}_{0}}$ and $\tau_{2k}^{\tilde{D}_{1}}=\tau_{2k}^{\tilde{D}}$ are cohomologous in 
$\HC^{2k}(\tilde{\cA})$.  Denote by $\overline{\tau}_{2k}^{D_{0}}$ the restriction to $\cA^{2k+1}$ of 
$\tau_{2k}^{\tilde{D}_{0}}$. Then~(\ref{eq:CCC.bS-extensions}) shows that $\overline{\tau}_{2k}^{D_{0}}$  and 
$\overline{\tau}_{2k}^{D}$ are cohomologous in $\HC^{2k}(\cA)$.

Bearing this in mind, we note that, for $a\in \cA$, we have
\begin{equation*}
    [\tilde{D}_{0},\tilde{\pi}(a)]_{\tilde{\sigma}}=  \begin{pmatrix}D & 0 \\ 0  & -D\end{pmatrix}\begin{pmatrix}a  & 0 \\ 0 & 0 \end{pmatrix} 
-\begin{pmatrix}\sigma(a)  & 0 \\ 0 & 0 \end{pmatrix}\begin{pmatrix}D & 0 \\ 0 & -D\end{pmatrix}= 
\begin{pmatrix}[D, a]_{\sigma} & 0 \\ 0 & 0\end{pmatrix}.
\end{equation*}Thus, 
\begin{equation*}
    \tilde{D}_{0}^{-1}[\tilde{D}_{0},\tilde{\pi}(a)]_{\tilde{\sigma}}= \begin{pmatrix}D^{-1}[D, a]_{\sigma} & 0 \\ 0 & 
    0\end{pmatrix} \quad \text{and}\quad
    \tilde{\gamma}\tilde{D}_{0}^{-1}[\tilde{D}_{0},\tilde{\pi}(a)]_{\tilde{\sigma}}= \begin{pmatrix}\gamma D^{-1}[D, a]_{\sigma} & 0 \\ 0 & 
    0\end{pmatrix}. 
\end{equation*}
It then follows that, for $a^{0},\ldots,a^{2k}$ in $\cA$, we have
\begin{equation*}
  \tau_{2k}^{\tilde{D}_{0}}(a^{0},\ldots,a^{2k})    = \Tr \left( \gamma D^{-1}[D,a^{0}]_{\sigma}\cdots  
  D^{-1}[D,a^{2k}]_{\sigma}\right)
      =  \tau_{2k}^{D}(a^{0},\ldots,a^{2k}) .
\end{equation*}
That is, the restriction of $\tau_{2k}^{\tilde{D}_{0}}$ to $\cA^{2k+1}$ is precisely the cocycle $\tau_{2k}^{D}$. 
Therefore, we arrive at the following statement. 

\begin{proposition}\label{prop:CC.equivalence-tau-overlinetau}
    If $D$ is invertible, then, for any $k\geq \frac{1}{2}(p+1)$, the cyclic cocycles $\tau_{2k}^{D}$ and 
    $\overline{\tau}^{D}_{2k}$ are cohomologous in $\HC^{2k}(\cA)$, and hence define the same class in $\HP^{0}(\cA)$. 
\end{proposition}

It follows from this that, when $D$ is invertible, Definition~\ref{def:CCC.invertible} and Definition~\ref{def:CCC.general} provide us with equivalent definitions of the Connes-Chern 
character of $(\cA,\cH,D)_{\sigma}$.

Finally, using Proposition~\ref{prop:CCC.otautD2k} and arguing as in the proof of Theorem~\ref{thm:CCC-index-formula.invertible} enables us to remove the invertibility assumption in 
Theorem~\ref{thm:CCC-index-formula.invertible}. 
We thus obtain the following index formula. 

\begin{theorem}\label{thm:CCC-index-formula}
Let $(\cA,\cH,D)_{\sigma}$ be a $p$-summable twisted spectral triple.  Then, for any finitely generated 
projective right $\cA$-module $\cE$ and  any $\sigma$-connection on $\cE$, we have 
\begin{equation*}
\ind D_{\nabla^{\cE}}=\acou{\Ch(D)_{\sigma}}{[\cE]}.
\end{equation*}
\end{theorem}

\begin{remark}
    The cocycles $\tau^{D}_{2k}$ and $\overline{\tau}^{D}_{2k}$ may be difficult to compute in practice, even in the case of a Dirac spectral 
    triple (see Theorem 6.5 of \cite[Part I]{Co:NCDG} and \cite{BF:APDOIT}).  In the ordinary case, a representation of the Connes-Chern character in entire cyclic cohomology is given by 
    the JLO cocycle~\cite{JLO:CMP, Co:ECCBACTSFM}, the existence of which only requires $\theta$-summability. We refer to 
    the paper of Quillen~\cite{Qu:ACCC} for interpretations of the Connes-Chern character and JLO cocycle in terms of Chern characters 
    of superconnections on cochains. 
\end{remark}

\begin{remark}
    Under further assumptions a representative in periodic cyclic cohomology is given by the CM cocycle~\cite{CM:LIFNCG} (see also~\cite{Hi:RITCM}). 
    The components of the CM cocycle are given by formulas that are \emph{local} in the sense that they involve universal linear combination of functionals of the form, 
    \begin{equation*}
    \bint a^0[D, a^1]^{[\alpha_1]}\cdots[D, a^{2k}]^{[\alpha_{2k}]}D^{-2(|\alpha|+k)}, \qquad a^j\in\cA,
    \end{equation*}
    where $T^{[j]}$ is the $j$-th iterated commutator of $T$ with $D^2$ and $\bint$ is an analogue of the noncommutative 
    residue trace of Guillemin~\cite{Gu:NPWF} and Wodzicki~\cite{Wo:LISA}. 
     This thus expresses the index pairing as a linear combinations of residues of various zeta 
     functions, in the spirit of the index formula of Atiyah-Bott~\cite{At:GATEDO}. We refer to~\cite{Wo:Thesis, Ka:RNC, 
     MN:HPDO1MB, MN:IHPDOMB} for other  types of residue index formulas. 
\end{remark}

\begin{remark}
    Let $(M^n, g)$ be a compact Riemannian manifold. The Connes-Chern character of the Dirac spectral triple $(C^{\infty}(M), L^2(M, \sS), \sD_g)$ is represented by the CM cocycle. 
    This CM cocycle can be computed by heat kernel techniques~\cite{CM:LIFNCG, Po:CMP}. We obtain the even cocycle $\varphi=(\varphi_{2k})$ given by
    \begin{equation*}
    \varphi_{2k}(f^0, \ldots, f^{2k})=\frac{(2i\pi)^{-\frac{n}{2}}}{(2k)!}\int_M\hat{A}(R^M)\wedge f^0df^1\wedge\ldots\wedge df^k.
    \end{equation*}
    In other words $\varphi=(2i\pi)^{-\frac{n}{2}}\varphi_{\hat{A}(R^M)^{\wedge}}$ in the sense of~(\ref{eq:current.cochain}), where $\hat{A}(R^M)^{\wedge}$ is the 
    Poincar\'e dual current of the $\hat{A}$-form $\hat{A}(R^M)$. Let us explain how this enables us to recover the 
    Atiyah-Singer index formula. Let $e\in M_{q}(C^{\infty}(M))$, $e^{2}=e$, and form the vector bundle $E=\ran e$, which we equip with its Grassmannian 
    connection $\nabla^{E}$. Then by~(\ref{eq:s-connections.Dirac}) we have
    \begin{equation*}
        \ind \sD_{\nabla^{E}}=\ind_{\sD_{g},\sigma}[e]. 
    \end{equation*}As $(2i\pi)^{-\frac{n}{2}}\varphi_{\hat{A}(R^M)^{\wedge}}$ represents the Connes-Chern character, by 
    Theorem~\ref{thm:CCC-index-formula} we have
    \begin{equation*}
        \ind_{\sD_{g},\sigma}[e]=(2i\pi)^{-\frac{n}{2}}\acou{\varphi_{\hat{A}(R^M)^{\wedge}}}{e}.
    \end{equation*}Moreover, using~(\ref{eq:CC.manifold-Chern-character}) we have 
    \begin{equation*}
        \acou{\varphi_{\hat{A}(R^M)^{\wedge}}}{e}=\acou{\hat{A}(R^M)^{\wedge}}{\Ch(F^{E})}= \int_{M} 
        \hat{A}(R^M)\wedge\Ch(F^{E}), 
    \end{equation*}where $F^{E}$ is the curvature of $\nabla^{E}$. Therefore, we obtain
    \begin{equation*}
       \ind \sD_{\nabla^{E}}= (2i\pi)^{-\frac{n}{2}} \int_{M} 
        \hat{A}(R^M)\wedge\Ch(F^{E}), 
    \end{equation*}which is the Atiyah-Singer index formula. 
\end{remark}

\begin{remark}
    It remains an open question to construct a version of the CM cocycle for twisted spectral triples. 
    Moscovici~\cite{Mo:LIFTST} devised an Ansatz for such a cocycle and verified it in the case of twistings of ordinary spectral triples by scaling automorphisms. 
    To date this seems to be the only known example of twisted spectral triples satisfying Moscovici's Ansatz. It would 
    already be interesting to have a version of Connes's Hochschild character formula~\cite{Co:NCG}. We refer to~\cite{FK:TSTCCF} 
    for a Hochschild character formula in the special case of twistings of ordinary spectral triples by scaling automorphisms. 
\end{remark}

\begin{remark}
    We refer to~\cite{PW:NCGCGI.PartI} for the computation of the Connes-Chern character of the conformal 
    Dirac spectral triple of~\cite{CM:TGNTQF} (the construction of which is recalled in
    Section~\ref{sec:TwistedST}). 
\end{remark}

\appendix

\section{Proof of Lemma~\ref{lem:CanonicalHermitMetric-eA^q}}
\label{app:PfLemCanoHermMetric}
 It is immediate that the restriction of $\acoup{\cdot}{\cdot}_{0}$ is $\cA$-sesquilinear and positive. The only 
    issue at stake is nondegeneracy.
    
  \begin{lemma}
  \label{lem:NondegenPairingE^*&E}
     Set $\cE^{*}=e^{*}\cA^{q}$. Then the restriction of $\acoup{\cdot}{\cdot}$ to $\cE^{*}\times \cE$ is nondegenerate. 
  \end{lemma} 
    \begin{proof}[Proof of Lemma~\ref{lem:NondegenPairingE^*&E}]
     We need to show that $\Phi:\cE^{*}\ni \xi \rightarrow  \left.\acoup{\xi}{\cdot}_{0}\right|_{\cE}\in \cE'$ is an $\cA$-antilinear isomorphism. Let $\xi=(\xi_{j})\in \cE^{*}$.  Then
    \begin{equation*}
        \acoup{\xi}{e\xi}_0=\acoup{e^{*}\xi}{\xi}_0=\acoup{\xi}{\xi}_0=\sum \xi_{j}^{*}\xi_{j}.
    \end{equation*}It then follows that if $\acoup{\xi}{\cdot}_0$ vanishes on $\cE$, then all the positive operators 
    $\xi_{j}^{*}\xi_{j}$ vanish on $\cH$ and hence $\xi=0$. This shows that $\Phi$ is 
    injective. 
    
    Let $\varphi \in \cE'$ and let $\tilde{\varphi}\in (\cA^{q})'$ be defined by $\tilde{\varphi}(\xi)=\varphi(e\xi)$ 
    for all $\xi \in \cA^{q}$. The nondegeneracy of $\acoup{\cdot}{\cdot}_0$ implies that there is $\tilde{\eta} \in 
    \cA^{q}$ such that $\tilde{\varphi}(\xi)=\acoup{\tilde{\eta}}{\xi}_0$ for all $\xi\in \cA^{q}$. Set $\eta=e^{*}\eta\in 
    \cE^{*}$. Then, for all $\xi\in \cE$, 
    \begin{equation*}
        \varphi(\xi)=\tilde{\varphi}(e\xi)=\acoup{\tilde{\eta}}{e\xi}_0=\acoup{e^{*}\tilde{\eta}}{\xi}_0=\acoup{\eta}{\xi}_0.
    \end{equation*}Thus $\varphi=\Phi(\eta)$. This shows that $\Phi$ is surjective.  Therefore, $\Phi$ is an 
    $\cA$-antilinear isomorphism. Likewise,  $\Psi: \cE\ni \eta\rightarrow \left.\acoup{\cdot}{\eta}_{0}\right|_{\cE^{*}}\in \left(\cE^{*}\right)'$ is an 
    $\cA$-linear isomorphism. This completes the proof of the lemma.
    \end{proof}

 \begin{lemma}
 \label{lem:ALinearIsoEE^*}
  Denote by $\ft:\cE\rightarrow \cE^{*}$ the $\cA$-linear map defined by 
  \begin{equation*}
      \ft \xi =e^{*} \xi \qquad \forall \xi \in \cE.
  \end{equation*}Then $\ft$ is an $\cA$-linear isomorphism from $\cE$ onto $\cE^{*}$.
 \end{lemma}
 \begin{proof}[Proof of Lemma~\ref{lem:ALinearIsoEE^*}]
   If $\cF$ is a right submodule of $\cA^{q}$ we shall denote by $\cF^{\perp}$ its orthogonal complement with respect to the 
   canonical Hermitian metric of $\cA^{q}$. For $a\in M_{q}(\cA)$ we shall identify $a$ with the associated 
   $\cA$-linear map $\cA^{q}\rightarrow \cA^{q}$. We observe that with this convention $a^{*}$ is identified with the 
   adjoint of $a$ with respect to $\acoup{\cdot}{\cdot}_{0}$, i.e., 
   \begin{equation*}
       \acoup{a^{*}\xi}{\eta}_{0}=\acoup{\xi}{a\eta}_{0} \qquad \forall \xi,\eta\in \cA^{q}.
   \end{equation*}We deduce from this that, for any idempotent $f\in M_{q}(\cA)$, 
   \begin{equation}
   \label{eq:fA^qPerp(1-f^*)A^q}
       \left(f\cA^{q}\right)^{\perp}=\left( \ran f\right)^{\perp}= \ker f^{*}=\ran (1-f^{*})=(1-f^{*})\cA^{q}.
   \end{equation}
  We note this implies that $\left( \left(f\cA^{q}\right)^{\perp}\right)^{\perp}=f\cA^{q}$. 
  
  Using~(\ref{eq:fA^qPerp(1-f^*)A^q}) we get
     \begin{equation*}
       \ker \ft = \ker e^{*} \cap \ran e = \left( \ran e \right)^{\perp}\cap \ran e =\{0\},
   \end{equation*}which shows that $\ft$ is one-to-one. Moreover, as $\cA$ is closed under holomorphic functional 
   calculus there is $g \in \GL_{q}(\cA)$ such that $f:=e^{*}eg$ is a selfadjoint idempotent which is similar to $e^{*}$ 
   (cf.\ \cite{Bl:KTOA}). Thus, 
   \begin{equation*}
   \ran e^{*}e= \ran f=  \left( \left(f\cA^{q}\right)^{\perp}\right)^{\perp}=\left(  \left( \ran 
   ee^{*}\right)^{\perp}\right)^{\perp}= \left( \ker e^{*}e\right)^{\perp}.
   \end{equation*}Obviously $\ker e\subset \ker e^{*}e$. As 
   $\acoup{e\xi}{e\xi}_{0}=\acoup{e^{*}e\xi}{\xi}_{0}$ for all $\xi \in \cA^{q}$, we see that $\ker e^{*}e$ is 
   contained in $\ker e$, and so the two submodules agree. Thus, 
   \begin{equation*}
        \ran e^{*}e=  \left( \ker e^{*}e\right)^{\perp} =  \left( \ker e\right)^{\perp} = \ran e^{*}. 
   \end{equation*}This shows that $\ft(\cE)=e^{*}(\ran e^{*})=\ran e^{*}e=\ran e^{*}=\cE^{*}$, that is, $\ft$ is onto. 
   Therefore, the $\cA$-linear map $\ft$ is an isomorphism. 
 \end{proof}
 
Let us go back to the proof of Lemma~\ref{lem:CanonicalHermitMetric-eA^q}. 
For all $\xi_{1}$ and $\xi_{2}$ in $\cE$, we have
 \begin{equation*}
     \acoup{\xi_{1}}{\xi_{2}}_{0}=\acoup{\xi_{1}}{e\xi_{2}}_{0}=\acoup{e^{*}\xi_{1}}{\xi_{2}}_{0}=\acoup{\ft\xi_{1}}{\xi_{2}}_{0}.
 \end{equation*}As $\acoup{\cdot}{\cdot}_{0}$ is nondegenerate on $\cE^{*}\times \cE$ by Lemma~\ref{lem:NondegenPairingE^*&E} and $\ft$ is an $\cA$-linear 
 isomorphism by Lemma~\ref{lem:ALinearIsoEE^*}, we then deduce that $\acoup{\cdot}{\cdot}_{0}$ is nondegenerate on $\cE\times \cE$. This completes the proof of 
 Lemma~\ref{lem:CanonicalHermitMetric-eA^q}. 

\section{Proof of Lemma~\ref{lem:H(E)topIndepHermitianMetricE}}\label{app:H(E)topIndepHermitianMetricE}
Let us first assume that $\cE=\cA^{q}$ for some $q\in \N_{0}$.  Let us denote by $\cH(\cA^{q})_{0}$ the pre-Hilbert space associated to the canonical Hermitian metric 
         $\acoup{\cdot}{\cdot}_{0}$ on $\cA^{q}$. There is a canonical isomorphism 
         $U_{0}:\cA^{q}\otimes_{\cA}\cH\rightarrow \cH^{q}$ given by 
         \begin{equation*}
             U_{0}(\xi\otimes \zeta)=  \left( \xi_{1}\zeta ,\ldots, \xi_{q}\zeta \right) \qquad \text{for all $\xi=(\xi_{j})\in 
             \cA^{q}$ and $\zeta\in \cH$}.
         \end{equation*}The inverse of $U_{0}$ is given by 
\begin{equation*}
    U^{*}\left( \zeta_{1}, \ldots,  \zeta_{q} \right) = \varepsilon_{1} 
    \otimes \zeta_{1}+ \cdots +  \varepsilon_{q} \otimes \zeta_{q}, \qquad   \zeta_{j}\in \cH,
\end{equation*}where $\varepsilon_{1},\ldots,\varepsilon_{q}$ form the canonical basis of $\cA^{q}$. We also observe 
that, for $\xi \in \cA^{q}$ and $\zeta \in \cH$, 
\begin{equation*}
    \left\| U_{0}(\xi\otimes \zeta)\right\|^{2}=\sum_{i}\acou{\xi_{j}\zeta}{\xi_{j}\zeta}= \sum_{i}\acou{\zeta}{
    \xi_{j}^{*}\xi_{j}\zeta}=\acou{\zeta}{\acoup{\xi}{\xi}_{0}\zeta}=\left\|\xi\otimes \zeta\right\|_{0}^{2},
\end{equation*}
where $\|\cdot\|_{0}$ is the norm of $\cH(\cA^{q})_{0}$ This shows that $U_{0}$ is an isometric isomorphism from $\cH(\cA^{q})_{0}$ 
onto $\cH^{q}$. As $\cH^{q}$ is complete, we deduce 
that so is $\cH(\cA^{q})_{0}$, i.e., $\cH(\cA^{q})_{0}$ is a Hilbert space. 

Let $\acoup\cdot\cdot$ be a Hermitian metric on $\cA^{q}$. We denote by $\cH(\cE)$ the associated pre-Hilbert space and 
by $\acou{\cdot}{\cdot}$ its inner product. The nondegeneracy of $\acoup\cdot\cdot$ and $\acoup{\cdot}{\cdot}_{0}$ implies 
there is a selfadjoint element $g \in \GL_{q}(\cA)$ such that
\begin{equation*}
\acoup{\xi}{\eta}  =\acoup{g\xi}{\eta}_0 \qquad \forall \xi,\eta\in \cA^{q}.
\end{equation*}
We also denote 
by $g$ the representation of $g$ as a selfadjoint bounded operator of $\cH^{q}$. Let $\xi=(\xi_{j})$ and 
$\xi'=(\xi_{j}')$ be in $\cA^{q}$ and let $\zeta$ and $\zeta' $ be in $\cH$. Then
\begin{align*}
\acou{\xi\otimes \zeta}{\xi'\otimes \zeta'} = 
\acou{\zeta}{\acoup{\xi}{\xi'}\zeta'}=\acou{\zeta}{\acoup{g\xi}{\xi'}_{0}\zeta'} = 
\sum_{i,j}\acou{\zeta}{(\xi_{j}^{*}g_{ij}^{*}\xi_{i}')\zeta'} 
& =\sum_{i,j}\acou{g_{ij}(\xi_{j}\zeta)}{\xi_{i}'\zeta'}\\
&= \acou{gU_{0}(\xi\otimes \zeta)}{U_{0}(\xi'\otimes \zeta')}.
\end{align*}
By bilinearity it then follows that $\acou{\eta}{\eta'}=\acou{gU_{0}\eta}{U_{0}\eta}$ for all $\eta$ and $\eta'$ in 
$\cH(\cA^{q})$. Thus, for all $\eta\in \cH(\cA^{q})$ and $\zeta\in \cH^{q}$, 
\begin{equation}
\label{eq:EqualityFor g}
    \|\eta\|^{2}=\acou{gU_{0}\eta}{U_{0}\eta} \qquad \text{and} \qquad \acou{g\zeta}{\zeta}=\|U_{0}^{-1}\zeta\|^{2}. 
\end{equation}

The 2nd equality in~(\ref{eq:EqualityFor g}) shows that $g$ is a positive operator of $\cH^{q}$. As $g$ is invertible, we see that its spectrum 
is contained in an interval $[c^{-1},c]$ for some $c>1$, and so, for all $\zeta\in \cH^{q}$, 
\begin{equation}
\label{eq:Estimate<gzeta,zeta>}
   c^{-1} \|\zeta\|^{2}\leq \acou{g\zeta}{\zeta}\leq c\|\zeta\|^{2} \
\end{equation}
Combining~(\ref{eq:EqualityFor g}) and the fact that $U_{0}$ is an isometry from $\cH(\cA^{q})_{0}$ onto $\cH^{q}$ we 
deduce that, for all $\eta\in \cA^{q}\otimes_{\cA}\cH$, we have
\begin{equation*}
    \|\eta\|^{2}=\acou{gU_{0}\eta}{U_{0}\eta}\in 
    \left[c^{-1}\|U_{0}\eta\|^{2},c\|U_{0}\eta\|^{2}\right]=\left[c^{-1}\|\eta\|^{2}_{0},c\|\eta\|^{2}_{0}\right].
\end{equation*}This shows that the norms $\|\cdot\|$ and $\|\cdot\|_{0}$ are equivalent on  $\cA^{q}\otimes_{\cA}\cH$. 
Therefore $\cH(\cA^{q})$ has 
same topology as $\cH(\cA^{q})_{0}$. In particular, $\cH(\cA^{q})$ is complete, and hence is a Hilbert space. This 
proves Lemma~\ref{lem:H(E)topIndepHermitianMetricE} in the special case $\cE=\cA^{q}$. 

Let us now assume that $\cE=e\cA^{q}$ with $e=e^{2}\in M_{q}(\cA)$. By Lemma~\ref{lem:CanonicalHermitMetric-eA^q} the canonical Hermitian metric of 
$\cA^{q}$ induces a Hermitian metric on $\cE$. We denote by $\acou{\cdot}{\cdot}$ and $\cH(\cE)_{0}$ the associated 
inner product and pre-Hilbert space. We also denote by $e$ the representation of $e$ as a bounded operator on 
$\cH^{q}$. We note that as $e$ is idempotent $e\cH^{q}$ is a closed subspace of $\cH^{q}$. 

Let $\xi=(\xi_{j})\in \cA^{q}$ and $\zeta \in \cH$. For $i=1,\ldots,q$,  we have 
\begin{equation*}
    U_{0}\left( (e\xi)\otimes \zeta\right)_{i} = 
    (e\xi)_{i}\zeta=\sum_{j}e_{ij}\xi_{j}\zeta=\sum_{j}e_{ij}U_{0}(\xi\otimes \zeta)_{j}=\left( eU_{0}(\xi\otimes 
    \zeta)\right)_{i}.
\end{equation*}That is,
\begin{equation*}
    U_{0}((e \xi)\otimes \zeta)= eU_{0}(\xi\otimes \zeta).
\end{equation*}
As $U_{0}$ is an isometric isomorphism from $\cH(\cA^{q})_{0}$ onto $\cH^{q}$ we see that $U_{0}$ induces an isometric 
isomorphism from $\cH(\cE)_{0}$ onto $e\cH^{q}$. As $e\cH^{q}$ is complete (since this is a closed subspace of 
$\cH^{q}$) we deduce that $\cH(\cE)_{0}$ is a Hilbert space. 

Let $\acoup{\cdot}{\cdot}$ be a Hermitian metric on $\cE$. Thanks to the nondegeneracy of $\acoup{\cdot}{\cdot}_{0}$ and $\acoup{\cdot}{\cdot}$ 
there is a unique $\cA$-linear isomorphism $a:\cE\rightarrow \cE$ such that 
\begin{equation*}
    \acoup{\xi}{\eta}=\acoup{a\xi}{\eta}_0 \qquad \text{for all $\xi$ and $\eta$ in $\cE$}.
\end{equation*}
We then extend $\acoup{\cdot}{\cdot}$ into the $\cA$-sesquilinear form on $\cA^{q}$ defined by
\begin{equation}
\label{eq:SesquilinearFormA^q}
   \acoup{\xi}{\eta}:=\acoup{ae\xi}{\eta}_{0}+\acoup{(1-e)\xi}{(1-e)\eta}_0 \qquad \text{for all $\xi$ and $\eta$ in $\cA^{q}$}.
\end{equation}
We note that $\acoup{\cdot}{\cdot}$ is positive on $\cA^{q}$, and 
\begin{equation*}
    \acoup{\xi}{\eta}=\acoup{g\xi}{\eta}_0 \qquad \text{for all $\xi$ and $\eta$ in $\cA^{q}$}.
\end{equation*}where we have set $g=e^{*}ae+(1-e)^{*}(1-e)$. 

By Lemma~\ref{lem:ALinearIsoEE^*} we know that $e^{*}$ induces an $\cA$-linear isomorphism from $e\cA^{q}$ onto $e^{*}\cA^{q}$ and $(1-e^{*})$ 
induces an isomorphism from $(1-e)\cA^{q}$ onto $(1-e^{*})\cA^{q}$. As $a$ is an isomorphism from $\cE=e\cA^{q}$ onto 
itself we deduce that $g$ is a right-module isomorphism from $\cA^{q}$ onto itself. Combining this with~(\ref{eq:SesquilinearFormA^q}) we then 
see that $\acoup\cdot\cdot$ is nondegenerate on $\cA^q\times \cA^q$. Thus $\acoup{\cdot}{\cdot}$ is a Hermitian metric on 
$\cA^{q}$. Therefore, by the first part of the proof, the associated norm on $\cA^{q}\otimes_{\cA}\cH$ is equivalent to 
the norm of $\cH(\cA^{q})$. As these norms restrict to the norms of $\cH(\cE)$ and $\cH(\cE)_{0}$ on $\cE=e\cA^{q}$, we 
then deduce that the norms of $\cH(\cE)$ and $\cH(\cE)_{0}$ are equivalent. This proves Lemma~\ref{lem:H(E)topIndepHermitianMetricE} in the special 
case $\cE=e\cA^{q}$, $e=e^{2}\in M_{q}(\cA)$. 

Let us now prove Lemma~\ref{lem:H(E)topIndepHermitianMetricE} when $\cE$ is an arbitrary finitely generated projective module, i.e., it is the direct 
summand of a free module $\cE_{0}$. Let $\phi:\cE_{0}\rightarrow \cA^{q}$ be an $\cA$-linear isomorphism. Then 
$\phi(\cE)=e\cA^{q}$ for some idempotent $e\in M_{q}(\cA)$. If $\acoup{\cdot}{\cdot}$ is a Hermitian metric on $\cE$, 
then we define a Hermitian metric $\acoup{\cdot}{\cdot}_{\phi}$ on $e\cA^{q}$ by
\begin{equation*}
    \acoup{\xi}{\eta}_{\phi}=\acoup{\phi^{-1}(\xi)}{\phi^{-1}(\eta)} \qquad \text{for all $\xi$ and $\eta$ in $e\cA^{q}$}.
\end{equation*}We denote by $\acou{\cdot}{\cdot}_{\phi}$ and $\cH(e\cA^{q})_{\phi}$the associated inner product and Hilbert space. 

Set $U_{\phi}:=\phi \otimes 1_{\cH}$. This a vector bundle isomorphism from $\cE\otimes_{\cA}\cH$ onto 
$(e\cA^{q})\otimes_{\cA}\cH$. Let $\xi$ and $\xi'$ be in $\cE$ and let $\zeta$ and $\zeta'$ be in $\cH$. Then
\begin{equation*}
    \acou{U_{\phi}(\xi\otimes \zeta)}{U_{\phi}(\xi'\otimes \zeta')}_{\phi}= 
    \acou{\zeta}{\acoup{\phi(\xi)}{\phi(\xi')}_{\phi}\zeta'}= \acou{\zeta}{\acoup{\xi}{\xi'}\zeta'}=\acou{\xi\otimes 
    \zeta}{\xi'\otimes \zeta'}.
\end{equation*}Thus $U_{\phi}$ is an isometric isomorphism from $\cH(\cE)$ and $\cH(e\cA^{q})_{\phi}$. As 
$\cH(e\cA^{q})_{\phi}$ is a Hilbert space, we then deduce that $\cH(\cE)$ is a Hilbert space as well. 

Finally, we observe that pushforwarding norms by $U_{\phi}$ gives rise to a one-to-one correspondence between norms on 
$\cE\otimes_{\cA}\cH$ and $\left( e\cA\right)^{q}\otimes \cH$ arising from Hermitian metrics on $\cE$ and $e\cA^{q}$. 
As all those norms on $e\cA^{q}$ are equivalent to each other, we then deduce that the same result holds on $\cE$. That 
is, the topology of $\cH(\cE)$ is independent of the choice of the Hermitian metric. The proof of Lemma~\ref{lem:H(E)topIndepHermitianMetricE} is 
complete. 

\section{Homotopy Invariance of the Connes-Chern Character}\label{app:homotopy-invariance}
In this appendix, we give a proof of the homotopy invariance of the Connes-Chern character in the following form.  

\begin{proposition}\label{prop:homotopy-invariance}
    Let $(\cA,\cH,D)_{\sigma}$ be a $p$-summable twisted spectral triple. Consider an operator 
    homotopy of the form,
    \begin{equation*}
        D_{t}=D+V_{t}, \qquad 0\leq t \leq 1,
    \end{equation*}where $(V_{t})_{0\leq t \leq 1}$ is a $C^{1}$ selfadjoint family in  $\cL(\cH)$ such that $D_{t}$ is 
    invertible for all $t\in [0,1]$ and $(D_{t}^{-1})_{0\leq t \leq 1}$ is a bounded family in $\cL^{p}(\cH)$. Then
    \begin{enumerate}
        \item  $(\cA,\cH,D_{t})_{\sigma}$ is a $p$-summable twisted spectral triple for all $t\in [0,1]$. 
    
        \item  For any $k\geq \frac{1}{2}(p-1)$, the cocycles $\tau_{2k}^{D_{0}}$ and $\tau_{2k}^{D_{1}}$ are 
        cohomologous in $\HC^{2k}(\cA)$.
        
       \item  The twisted spectral triples $(\cA,\cH,D_{0})_{\sigma}$ and $(\cA,\cH,D_{1})_{\sigma}$ have same 
       Connes-Chern character in $\HP^{0}(\cA)$. 
    \end{enumerate}
 \end{proposition}

By assumption the resolvent $D_{t}^{-1}$ lies in $\cL^{p}(\cH)$. Moreover, for all $a \in \cA$,
\begin{equation}
    [D_{t},a]_{\sigma}=[D,a]_{\sigma}+(V_{t}a-\sigma(a)V_{t})\in \cL(\cH).
    \label{eq:homotopy.[D,a]}
\end{equation}
Therefore $(\cA,\cH,D_{t})_{\sigma}$ is a $p$-summable twisted spectral triple, and so, for any integer 
$k\geq \frac{1}{2}(p-1)$, we can form the cyclic $2k$-cocycle,
\begin{equation*}
   \tau^{D_{t}}_{2k}(a^{0},\ldots, a^{2k})=c_{k}\Str\left( 
   D_{t}^{-1}[D_{t},a^{0}]_{\sigma}\cdots 
   D_{t}^{-1}[D_{t},a^{2k}]_{\sigma}\right), \quad a^{j}\in \cA.
\end{equation*}
The rest of the proof is devoted to comparing the cocycles $\tau_{2k}^{D_{1}}$ and $\tau_{2k}^{D_{0}}$. 

In what follows, we set $ \dot{V}_{t}=\frac{d}{dt}V_{t}$ and, for $a \in \cA$, we define
\begin{equation*}
  \delta_{t}(a)=D_{t}^{-1}[\dot{V}_{t}D_{t}^{-1},\sigma(a)]D_{t}. 
\end{equation*}
We note that
\begin{equation*}
     \delta_{t}(a)=[D_{t}^{-1}\dot{V}_{t},D_{t}^{-1}\sigma(a)D_{t}]= 
     [D_{t}^{-1}\dot{V}_{t},a]-[D_{t}^{-1}\dot{V}_{t},[D_{t},a]_{\sigma}].
\end{equation*}
As~(\ref{eq:homotopy.[D,a]}) shows that $([D_{t},a]_{\sigma})_{0\leq t \leq 1}$ is a continuous family in $\cL(\cH)$ and 
$(D_{t}^{-1}\dot{V}_{t})_{0\leq t \leq 1}$ is a continuous family in $\cL^{p}(\cH)$, we then see that 
$(\delta_{t}(a))_{0\leq t \leq 1}$ is a continuous family in $\cL^{p}(\cH)$. For $j=1,\ldots,2k+1$, we let $\eta^{t}_{j}$ be the $(2k+1)$-cochain on $\cA$ defined by
\begin{equation*}
    \eta_{j}^{t}(a^{0},\ldots,a^{2k+1})= \Str\left( \alpha_{j}(a^{0})
   D_{t}^{-1}[D_{t},a^{1}]_{\sigma}\cdots \delta_{t}(a^{j}) \cdots 
   D_{t}^{-1}[D_{t},a^{2k+1}]_{\sigma}\right), \quad a^{l}\in \cA,
\end{equation*}where $\alpha_{j}(a)=a$ if $j$ is even and $\alpha_{j}(a)=D_{t}^{-1}\sigma(a) D_{t}$ 
if $j$ is odd. Note that $ \eta_{j}^{t}$ is a normalized cochain. 

In what follows we shall say that a family $(\varphi^{t})_{0\leq t \leq 1}\subset C^{m}(\cA)$ is $C^{\alpha}$, 
$\alpha\geq 0$, when, for all 
$a^{0},\ldots,a^{m}$, the function $t\rightarrow \varphi^{t}(a^{0},\ldots,a^{m})$ is $C^{\alpha}$ on $[0,1]$. Given 
a $C^{1}$-family $(\varphi^{t})_{0\leq t \leq 1}$ of $m$-cochains, we define $m$-cochains $\frac{d}{dt}\varphi^{t}$, $t\in [0,1]$, by
\begin{equation*}
    \left( \frac{d}{dt}\varphi^{t}\right)(a^{0},\ldots,a^{m}):=\frac{d}{dt}\left( 
    \varphi^{t}\right)(a^{0},\ldots,a^{m}), \qquad a^{j}\in \cA.
\end{equation*}
Given a $C^{0}$-family $(\psi^{t})_{0\leq t \leq 1}$ of $m$-cochains we define the integral $\int \psi^{t}$ as the 
$m$-cochain given by 
\begin{equation*}
    \left( \int_{0}^{1} \psi^{t}dt \right)(a^{0},\ldots,a^{m}):=\int_{0}^{1}
    \psi^{t}(a^{0},\ldots,a^{m})dt, \qquad a^{j}\in \cA.
\end{equation*}
If $F$ is any of the operators $b$, $A$, $B_{0}$ or $B$, then
\begin{equation}
    F \left( \frac{d}{dt}\varphi^{t}\right)=\frac{d}{dt}\left( F\varphi^{t}\right) \qquad \text{and} \qquad 
    F \left( \int_{0}^{1} \psi^{t}dt \right)= \int_{0}^{1}F\psi^{t}dt.
    \label{eq:homotopy.F-ddt-int}
\end{equation}
Moreover, we have
\begin{equation}
    \int_{0}^{1} \left(\frac{d}{dt}\varphi^{t}\right)dt=\varphi^{1}-\varphi^{0}.
    \label{eq:homotopy.int-ddt}
\end{equation}

\begin{lemma}\label{lem:homotopy.differentiability-taut}
    The family $(\tau^{D_{t}}_{2k})_{0\leq t \leq 1}$ is a $C^{1}$-family of $2k$-cochains and we have
    \begin{equation*}
        \frac{d}{dt}\tau^{D_{t}}_{2k}= \frac{c_{k}}{2k+1}\sum_{j=1}^{2k+1}B\eta_{j}^{t}.
    \end{equation*}
\end{lemma}
\begin{proof}
 It follows from~(\ref{eq:homotopy.[D,a]}) that $([D_{t},a]_{\sigma})_{0\leq t \leq 1}$ is a $C^{1}$-family in $\cL(\cH)$ and we have
 \begin{equation}
     \frac{d}{dt}[D_{t},a]_{\sigma}=\dot{V}_{t}a-\dot{V_{t}}\sigma(a).
             \label{eq:homotopy.ddt[D,a]}
 \end{equation}
  By assumption the family $(D_{t}^{-1})_{0\leq t \leq 1}$ is bounded in $\cL^{p}(\cH)$. Moreover, 
  \begin{equation*}
      D_{t+s}^{-1}-D_{t}^{-1}=-D_{t+s}^{-1}(D_{t+s}-D_{t})D_{t}^{-1}=-D_{t+s}^{-1}(V_{t+s}-V_{t})D_{t}^{-1}.
  \end{equation*}
 We then deduce that $(D_{t}^{-1})_{0\leq t \leq 1}$ is a continuous family in 
      $\cL^{p}(\cH)$. Combining this with the above equality then shows that $(D_{t}^{-1})_{0\leq t 
      \leq 1}$ is a differentiable family in $\cL^{p}(\cH)$ with
      \begin{equation}
          \frac{d}{dt}D_{t}^{-1}=-D_{t}^{-1}\dot{V}_{t}D_{t}^{-1}.
          \label{eq:homotopy.ddtD-1}
      \end{equation}As the above right-hand side is a continuous family  in 
      $\cL^{p}(\cH)$ we eventually see that $(D_{t}^{-1})_{0\leq t \leq 1}$ is a $C^{1}$ family in 
      $\cL^{p}(\cH)$.
      The product in $\cL(\cH)$ induces a continuous bilinear map from $\cL^{p}(\cH)\times 
      \cL(\cH)$ to $\cL^{p}(\cH)$. Therefore, we deduce that $\left(D_{t}^{-1}[D_{t},a]_{\sigma}\right)_{0\leq t \leq 1}$ is a 
    $C^{1}$-family in $\cL^{p}(\cH)$, and using~(\ref{eq:homotopy.ddt[D,a]}) and~(\ref{eq:homotopy.ddtD-1}) we obtain
    \begin{align}
       \frac{d}{dt} D_{t}^{-1}[D_{t},a]_{\sigma} & = 
       -D_{t}^{-1}\dot{V}_{t}D_{t}^{-1} [D_{t},a]_{\sigma}  +  D_{t}^{-1}\left(\dot{V}_{t}a- 
      \sigma(a)\dot{V}_{t}\right) \nonumber  \\
      & = D_{t}^{-1}\left( -\dot{V}_{t}a+ \dot{V}_{t}D_{t}^{-1}\sigma(a) + \dot{V}_{t}a- 
     \sigma(a)\dot{V}_{t}\right)  \label{eq:homotopy.ddtD-1[D,a]}
     \\ & = \delta_{t}(a). \nonumber
    \end{align}
    
    Let $a^{0},\ldots,a^{2k}$ be in $\cA$. As $2k+1\geq p$ the product of $\cL(\cH)$ induces a continuous 
$(2k+1)$-linear map from $\cL^{p}(\cH)^{2k+1}$ to $\cL^{1}(\cH)$. Therefore, 
the map $t\rightarrow  D_{t}^{-1}[D_{t},a^{0}]_{\sigma}\cdots 
   D_{t}^{-1}[D_{t},a^{2k+1}]_{\sigma}$ is a $C^{1}$-map from $[0,1]$ to 
   $\cL^{1}(\cH)$. Composing it with the supertrace on $\cL^{1}(\cH)$ we then deduce that the function 
   $t \rightarrow \tau^{D_{t}}_{2k}(a^{0},\ldots,a^{2k})$ is $C^{1}$ on $[0,1]$. Moreover, using~(\ref{lem:homotopy.differentiability-taut}) we get 
\begin{equation*}
    \frac{d}{dt}\tau^{D_{t}}_{2k}(a^{0},\ldots,a^{2k})= c_{k}\sum_{j=0}^{2k}\Str\left( 
   D_{t}^{-1}[D_{t},a^{0}]_{\sigma}\cdots \delta_{t}(a^{j}) \cdots 
   D_{t}^{-1}[D_{t},a^{2k}]_{\sigma}\right).
\end{equation*}Noting that $\alpha_{j}(1)=1$ we see that
\begin{equation*}
 \Str\left( 
   D_{t}^{-1}[D_{t},a^{0}]_{\sigma}\cdots \delta_{t}(a^{j}) \cdots 
   D_{t}^{-1}[D_{t},a^{2k}]_{\sigma}\right)=B_{0}\eta_{j+1}(a^{0}, \ldots,a^{2k}).   
\end{equation*}Therefore, we see that $(\tau^{D_{t}}_{2k})_{0\leq t \leq 1}$ is a $C^{1}$-family of cochains and 
$\frac{d}{dt}\tau^{D_{t}}_{2k}=c_{k}\sum_{j=1}^{2k+1}B_{0}\eta_{j}^{t}$.
As the $\eta_j^t$ are normalized cochains and $\tau_{2k}^{D_{t}}$ is a cyclic cocycle, using~(\ref{eq:homotopy.F-ddt-int}) we get
\begin{equation*}
   \frac{d}{dt}\tau^{D_{t}}_{2k}=\frac{1}{2k+1} \frac{d}{dt}A\tau^{D_{t}}_{2k}=\frac{1}{2k+1}A\left(  \frac{d}{dt}\tau^{D_{t}}_{2k}\right)= 
   \frac{c_{k}}{2k+1}\sum_{j=1}^{2k+1}AB_{0}\eta_{j}^{t}= \frac{c_{k}}{2k+1}\sum_{j=1}^{2k+1}
   B\eta_{j}^{t}.
\end{equation*}
The proof is complete. 
\end{proof}

\begin{lemma}\label{lem:homotopy.betaj=0}
 For $t\in[0,1]$ and $j=1,\ldots,2k+1$ the cochain $\eta_{j}^{t}$ is a Hochschild cocycle, i.e., $b\eta_{j}^{t}=0$. 
\end{lemma}
\begin{proof}
 Let $\beta$ and $\gamma$ be the $(2k+2)$-cochains on $\cA$ given by 
 \begin{gather*}
     \beta(a^{0},\ldots,a^{2k+2})= \Str \left( \alpha_{j}(a^{0})D^{-1}[D,a^{1}]_{\sigma}\cdots \delta_{t}(a^{j+1}) \cdots 
     D^{-1}[D,a^{2k+2}]_{\sigma}\right), \\
    \gamma(a^{0},\ldots,a^{2k+2})= \Str \left( \alpha_{j}(a^{0})D^{-1}[D,a^{1}]_{\sigma}\cdots \delta_{t}(a^{j}) \cdots 
     D^{-1}[D,a^{2k+2}]_{\sigma}\right), \quad a^{j}\in \cA. 
 \end{gather*}
 For $l=1,\ldots,j$ we let $\beta'_{l}$ and $\beta''_{l}$ be the $(2k+2)$-cochains defined by
 \begin{gather*}
\beta'_{l}(a^{0},\ldots,a^{2k+2})= \Str \left( \alpha_{j}(a^{0})D^{-1}[D,a^{1}]_{\sigma}\cdots a^{l}\cdots  \delta_{t}(a^{j+1}) \cdots 
     D^{-1}[D,a^{2k+2}]_{\sigma}\right), \\
    \beta''_{l}(a^{0},\ldots,a^{2k+2})= \Str \left( \alpha_{j}(a^{0})D^{-1}[D,a^{1}]_{\sigma}\cdots D^{-1}\sigma(a^{l})D\cdots  \delta_{t}(a^{j+1}) \cdots 
     D^{-1}[D,a^{2k+2}]_{\sigma}\right).
 \end{gather*}We note that $ \beta'_{l}(a^{0},\ldots,a^{2k+2})-\beta''_{l}(a^{0},\ldots,a^{2k+2})$ is equal to
 \begin{multline}
    \Str \left( 
    \alpha_{j}(a^{0})D^{-1}[D,a^{1}]_{\sigma}\cdots (a^{l}-D^{-1}\sigma(a^{l})D)\cdots  \delta_{t}(a^{j+1}) \cdots 
     D^{-1}[D,a^{2k+2}]_{\sigma}\right)\\
      = \beta(a^{0},\ldots,a^{2k+2}). 
      \label{eq:homotopy.beta'-beta''}
 \end{multline}Moreover, from the equality 
 $D^{-1}[D,a^{l}a^{l+1}]_{\sigma}=D^{-1}[D,a^{l}]_{\sigma}a^{l+1}+D^{-1}\sigma(a^{l})D\cdot D^{-1}[D,a^{l+1}]_{\sigma}$ 
 we deduce that
 \begin{equation}
     b_{l}\eta_{j}^{t}=\beta_{l+1}'+\beta_{l}''.
     \label{eq:homotopy.beta'+beta''}
 \end{equation}
 For  $l=j+1,\ldots,2k+1$ we let $\gamma'_{l}$ and $\gamma''_{l}$ be the $(2k+2)$-cochains on $\cA$ defined by
 \begin{gather}
\gamma'_{l}(a^{0},\ldots,a^{2k+2})= \Str \left( \alpha_{j}(a^{0})D^{-1}[D,a^{1}]_{\sigma}\cdots  \delta_{t}(a^{j}) \cdots a^{l}\cdots 
     D^{-1}[D,a^{2k+2}]_{\sigma}\right), \\
    \gamma''_{l}(a^{0},\ldots,a^{2k+2})= \Str \left( \alpha_{j}(a^{0})D^{-1}[D,a^{1}]_{\sigma}\cdots  \delta_{t}(a^{j}) 
    \cdots D^{-1}\sigma(a^{l})D \cdots 
     D^{-1}[D,a^{2k+2}]_{\sigma}\right).
 \end{gather}As in~(\ref{eq:homotopy.beta'-beta''}) and~(\ref{eq:homotopy.beta'+beta''}) we have
 \begin{equation}
     \gamma'_{l}-\gamma_{l}''=\gamma \qquad \text{and} \qquad b_{l}\eta=\gamma'_{l+1}+\gamma_{l}.
 \end{equation}
In addition, using the equality $\delta_{t}(a^{j}a^{j+1})= \delta_{t}(a^{j})D^{-1}\sigma(a^{j})D +D^{-1}a^{j}D\delta_{t}(a^{j+1})$ 
we find that
\begin{align}
    b_{j}\eta_{j}^{t}(a^{0},\ldots,a^{2k+2}) &=  \Str \left( \alpha_{j}(a^{0})D^{-1}[D,a^{1}]_{\sigma}\cdots \delta_{t}(a^{j}a^{j+1}) \cdots 
     D^{-1}[D,a^{2k+1}]_{\sigma}\right) \nonumber \\
    & = \gamma''_{j+1}(a^{0},\ldots,a^{2k+2})+\beta''_{j}(a^{0},\ldots,a^{2k+2}).
    \label{eq:homotopy.bjetaj}
\end{align}

Using~(\ref{eq:homotopy.beta'-beta''})--(\ref{eq:homotopy.bjetaj}) we obtain
\begin{align*} 
  \sum_{l=1}^{2k+1}b\eta_{j}^{t}&=\sum_{l=1}^{j-1}(-1)^{l}(\beta_{l+1}'+\beta_{l}'')+(-1)^{j}(\beta''_{j+1}+\beta''_{j}) + 
  \sum_{l=j+1}^{2k+1}(-1)^{l}(\beta_{l+1}'+\beta_{l}'') \\
    &= -\beta''_{1}+\sum_{l=2}^{j}(-1)^{l}(\beta''_{l}-\beta'_{l})+ 
    \sum_{l=j+2}^{2k+1}(-1)^{l}(\gamma''_{l}-\gamma'_{l})- \gamma'_{2k+1}\\
   & = -\beta''_{1}+ \sum_{l=2}^{j}(-1)^{l-1}\beta + \sum_{l=j+2}^{2k+1}(-1)^{l-1} \gamma - \gamma'_{2k+1}.
\end{align*}
Noting that $\sum_{l=2}^{j}(-1)^{l-1}=-\frac{1}{2}\left(1+(-1)^{j}\right)$ and $ \sum_{l=j+2}^{2k+1}(-1)^{l-1} 
=\frac{1}{2} \left(1-(-1)^{j}\right)$ we see that
\begin{equation}
    b\eta_{j}^{t}=  \sum_{l=0}^{2k+2}b\eta_{j}^{t}=b_{0}\eta -\beta''_{1} -\frac{1}{2}\left(1+(-1)^{j}\right)\beta 
    +\frac{1}{2}\left(1-(-1)^{j}\right) \gamma
    +b_{2k_{2}}\eta- \gamma'_{2k+1}.
    \label{eq:homotopy.beta}
\end{equation}
We note that $ b_{0}\eta_{j}^{t}(a^{0},\ldots,a^{2k+2})-\beta''_{1}(a^{0},\ldots,a^{2k+2})$ is equal to
\begin{equation}
   \Str \left( \alpha_{j}(a^{0}) \left( \alpha_{j}(a^{1})-D^{-1}\sigma(a^{1})D\right) D^{-1}[D,a^{2}]_{\sigma}\cdots \delta_{t}(a^{j+1}) \cdots 
     D^{-1}[D,a^{2k+2}]_{\sigma}\right). 
     \label{eq:homotopy.b0+2etaj-beta1}
\end{equation}We also observe that 
\begin{align*}
b_{2k+2}\eta_{j}^{t}(a^{0},\ldots,a^{2k+2})     & =  \Str \left( \alpha_{j}(a^{2k+2})\alpha_{j}(a^{0})D^{-1}[D,a^{1}]_{\sigma}\cdots \delta_{t}(a^{j}) \cdots 
     D^{-1}[D,a^{2k+2}]_{\sigma}\right) \\
     & = \Str \left( \alpha_{j}(a^{0})D^{-1}[D,a^{1}]_{\sigma}\cdots \delta_{t}(a^{j}) \cdots 
     D^{-1}[D,a^{2k+2}]_{\sigma}\alpha_{j}(a^{2k+2})\right). 
\end{align*}Thus $ b_{2k+2}\eta_{j}^{t}(a^{0},\ldots,a^{2k+2})-\beta_{2k+1}(a^{0},\ldots,a^{2k+2})$ is equal to
\begin{equation}
     \Str \left( \alpha_{j}(a^{0})D^{-1}[D,a^{1}]_{\sigma}\cdots \delta_{t}(a^{j}) \cdots 
     D^{-1}[D,a^{2k+2}]_{\sigma} \left( \alpha_{j}(a^{2k+2})-a^{2k+2}\right)\right).
     \label{eq:homotopy.b2k+2etaj-beta2k+1}
\end{equation}

Suppose that $j$ is even, so that $\alpha_{j}(a)=a$. Then~(\ref{eq:homotopy.b2k+2etaj-beta2k+1}) shows that $b_{2k+2}\eta-\beta_{2k+1}=0$. Moreover, 
$\alpha_{j}(a)-D^{-1}\sigma(a)D=D^{-1}[D,a]_{\sigma}$, and so using~(\ref{eq:homotopy.b0+2etaj-beta1}) we see that $b_{0}\eta-\beta''_{1}=\beta$. 
Therefore, in this case~(\ref{eq:homotopy.beta}) gives 
\begin{equation*}
    b\eta_{j}^{t}= \beta -\frac{1}{2}\left(1+1\right)\beta +\frac{1}{2}\left(1-1\right) 
    \gamma+0=0. 
\end{equation*}
When $j$ is odd, $\alpha_{j}(a)=D^{-1}\sigma(a)D$, and we similarly find that $b_{0}\eta-\beta''_{1}=0$ and 
$b_{2k+2}\eta-\beta_{2k+1}=-\gamma$. Thus, in this case~(\ref{eq:homotopy.beta}) gives
\begin{equation*}
    b\eta_{j}^{t}=0+ -\frac{1}{2}\left(1-1\right)\beta +\frac{1}{2}\left(1+1\right) \gamma-\gamma=0.
\end{equation*}
In any case, $\eta_{j}^{t}$ is a Hochschild cocycle. The  proof is complete.
\end{proof}

Let us go back to the proof of Proposition~\ref{prop:homotopy-invariance}. In the same way as in the proof of Lemma~\ref{lem:homotopy.differentiability-taut} it can be shown that each family $(\eta^{t}_{j})_{0\leq t \leq 1}$ is a 
continuous family of cochains. Note also that these cochains are normalized. Let $\eta$ be the $(2k+1)$-cochain defined by
\begin{equation*}
    \eta = \sum_{j=1}^{2k+1}\int_{0}^{1}\eta_{j}^{t}dt. 
\end{equation*}
It follows from~(\ref{eq:homotopy.F-ddt-int})--(\ref{eq:homotopy.int-ddt}) and Lemma~\ref{lem:homotopy.differentiability-taut} that
\begin{equation}
    B\eta = \int_{0}^{1} \biggl(\sum_{j=1}^{2k+1}B\eta_{j}^{t}\biggr)dt= (2k+1)c_{k}^{-1}\int_{0}^{1} \left( 
    \frac{d}{dt}\tau_{2k}^{D_{t}}\right)dt= (2k+1)c_{k}^{-1}(\tau_{2k}^{D_{1}}-\tau_{2k}^{D_{0}}).
    \label{eq:homotopy.Beta}
\end{equation}
Moreover, using~(\ref{eq:homotopy.F-ddt-int}) and Lemma~\ref{lem:homotopy.betaj=0} we get
\begin{equation*}
    b\eta =  \sum_{j=1}^{2k+1}\int_{0}^{1}b\eta_{j}^{t}dt=0.
\end{equation*}
In particular, as $\eta$ is a normalized cochain and $b\eta$ is cyclic, we may apply~(\ref{eq:SB=-b}) to get
\begin{equation*}
(2k+1)c_{k}^{-1}(S\tau_{2k}^{D_{1}}-S\tau_{2k}^{D_{0}})=SB\eta=-b\eta=0 \qquad \text{in $\HC^{2k+2}(\cA)$}.
\end{equation*}As by Remark~\ref{rmk:tau2k+2.cohomologous.Stau2k} we know that $\tau_{2k+2}^{D_{j}}$ and $S\tau_{2k}^{D_{j}}$ are cohomologous in 
$\HC^{2k+2}(\cA)$, it follows that so are the cocycles $ \tau_{2k+2}^{D_{0}}$ and $\tau_{2k+2}^{D_{1}}$. This proves 
the 2nd part of Proposition~\ref{prop:homotopy-invariance}. This also implies that $\tau_{k}^{D_{0}}$ and $\tau_{k}^{D_{1}}$ define the same 
class in $\HP^{0}(\cA)$, and so the twisted spectral triples $(\cA,\cH,D_{0})_{\sigma}$ and $(\cA,\cH,D_{1})_{\sigma}$ 
have same Connes-Chern character in $\HP^{0}(\cA)$. This completes the proof of Proposition~\ref{prop:homotopy-invariance}

 \begin{remark}
 By using the bounded Fredholm module pairs associated to a twisted  spectral triple in~\cite{CM:TGNTQF}, we also can deduce 
 Proposition~\ref{prop:homotopy-invariance}  from the homotopy invariance of the Connes-Chern character of a bounded 
 Fredholm module in~\cite[Part~I, \S 5]{Co:NCDG}. 
\end{remark}

\end{document}